\documentclass[12pt,oneside]{amsart}
\usepackage{amsmath,amsopn,amsfonts,amsthm,amssymb}
\usepackage{setspace}
\usepackage{fullpage}
\usepackage{mathdots}
\usepackage{mathtools}
\usepackage{manfnt}
\usepackage{marginnote}

\reversemarginpar

\usepackage{tikz}
\usetikzlibrary{arrows,decorations.pathmorphing,backgrounds,positioning,fit,matrix,calc}

\newtheorem{thm}{Theorem}[section]

\newtheorem{lemma}[thm]{Lemma}
\newtheorem{conj}[thm]{Conjecture}
\newtheorem{coro}[thm]{Corollary}
\newtheorem{prop}[thm]{Proposition}

\newtheorem{innerprethm}{Theorem}
\newenvironment{prethm}[1]
  {\renewcommand\theinnerprethm{#1}\innerprethm}
  {\endinnerprethm}

\theoremstyle{definition}
\newtheorem{defn}[thm]{Definition}

\theoremstyle{remark}

\newenvironment{ex}{\refstepcounter{thm}\begin{proof}[Example \emph{\thethm}]}{\end{proof}}
\newenvironment{rem}{\refstepcounter{thm}\begin{proof}[Remark \emph{\thethm}]}{\end{proof}}
\newenvironment{conv}{\refstepcounter{thm}\begin{proof}[Convention \emph{\thethm}]}{\end{proof}}
\newenvironment{warn}{\refstepcounter{thm}\begin{proof}[Warning \emph{\thethm}]}{\end{proof}}
\newenvironment{nota}{\refstepcounter{thm}\begin{proof}[Notation \emph{\thethm} ]}{\end{proof}}
\numberwithin{equation}{section}

\def\ZZ{\mathbb{Z}}

\def\NN{\mathbb{N}}

\usepackage{stackengine}
\stackMath
\usepackage{graphicx}

\newcommand\reversal[1]{\overline{#1}}

\colorlet{gpurple}{red!35!blue}
\colorlet{ggreen}{green!50!black}
\colorlet{dark purple}{red!35!blue}
\colorlet{dark green}{green!70!black}
\colorlet{dark red}{red!80!black}
\colorlet{dark blue}{blue!80!black!80!cyan}

\tikzstyle{mutable}=[inner sep=0.5mm,circle,draw,minimum size=2mm]
\tikzstyle{frozen}=[inner sep=.9mm,rectangle,draw]
\tikzstyle{dot} = [fill=black!25,inner sep=0.5mm,circle,draw,minimum size=1mm]
\tikzstyle{marked}=[inner sep=0.5mm,circle,draw,blue!75!black,fill=blue!50]

\title{Linear recurrences indexed by $\ZZ$}

\author{Greg Muller}

\def\C{\mathsf{C}}
\def\D{\mathsf{D}}

\def\k{\mathsf{k}}
\def\ker{\mathrm{ker}}
\def\adj{\mathsf{Adj}}
\def\sol{\mathsf{Sol}}
\def\spl{\mathsf{Spl}}
\def\P{\mathsf{P}}
\def\dddots{\rotatebox{-45}{$\cdots$}}
\def\vvdots{\rotatebox{90}{$\cdots$}}

\def\rt{\vec{\tau}}
\def\lt{\cev{\tau}}

\newcommand{\bad}[1]{{\color{dark red} \it #1}}

\begin{document}

\pagestyle{plain}

\begin{abstract}
This note considers linear recurrences (also called linear difference equations) in unknowns indexed by the integers. We characterize a unique \emph{reduced} linear recurrence with the same solutions as a given linear recurrence, and construct a \emph{solution matrix} which parametrizes the space of solutions. Several properties of solution matrices are shown, including a combinatorial characterization of bases and dimension of the space of solutions.
\end{abstract}

\makeatletter
\@namedef{subjclassname@2020}{%
  \textup{2020} Mathematics Subject Classification}
\makeatother

\keywords{
Linear recurrences, linear difference equations, determinants, friezes, positroids
}
\subjclass[2020]{
Primary 39A, 
Secondary 11B37, 
15A06, 
15A15, 
}

\maketitle


\vspace{1cm}

{ \centering
\begin{tikzpicture}[baseline=(current bounding box.center),
	ampersand replacement=\&,
	]
	\clip[use as bounding box] (-9.5,-2.7) rectangle (6.3,0.3);
	\matrix[matrix of math nodes,
		matrix anchor = M-2-24.center,
		nodes in empty cells,
		inner sep=0pt,
		nodes={anchor=center,node font=\scriptsize,rotate=45},
		column sep={0.4cm,between origins},
		row sep={0.4cm,between origins},
	] (M) at (0,0) {
 \&  \&  \&  \&  \&  \&  \&  \&  \&  \&  \&  \&  \&  \&  \&  \&  \&  \&  \&  \&  \&  \&  \&  \&  \&  \&  \&  \&  \&  \&  \&  \&  \&  \&  \&  \&  \&  \&  \\
 \& 1 \&  \& 1 \&  \& 1 \&  \& 1 \&  \& 1 \&  \& 1 \&  \& 1 \&  \& 1 \&  \& 1 \&  \& 1 \&  \& 1 \&  \& 1 \&  \& 1 \&  \& 1 \&  \& 1 \&  \& 1 \&  \& 1 \&  \& 1 \&  \& 1 \&  \\
\cdots \&  \& 5 \&  \& 1 \&  \& 1 \&  \& 3 \&  \& 2 \&  \& 1 \&  \& 5 \&  \& 1 \&  \& 5 \&  \& 1 \&  \& 1 \&  \& 3 \&  \& 2 \&  \& 1 \&  \& 5 \&  \& 1 \&  \& 5 \&  \& 1 \&  \& \cdots \\
 \& 3 \&  \& 2 \&  \&  \&  \&  \&  \& 5 \&  \& 1 \&  \& 3 \&  \& 2 \&  \& 3 \&  \& 2 \&  \&  \&  \&  \&  \& 5 \&  \& 1 \&  \& 3 \&  \& 2 \&  \& 3 \&  \& 2 \&  \&  \&  \\
\cdots \&  \& 1 \&  \&  \&  \& -2 \&  \&  \&  \& 2 \&  \& 1 \&  \& 1 \&  \& 1 \&  \& 1 \&  \&  \&  \& -2 \&  \&  \&  \& 2 \&  \& 1 \&  \& 1 \&  \& 1 \&  \& 1 \&  \&  \&  \& \cdots \\
 \&  \&  \&  \&  \& -1 \&  \& -3 \&  \&  \&  \& 1 \&  \&  \&  \&  \&  \&  \&  \&  \&  \& -1 \&  \& -3 \&  \&  \&  \& 1 \&  \&  \&  \&  \&  \&  \&  \&  \&  \& -1 \&  \\
 \&  \&  \&  \&  \&  \& -1 \&  \& -1 \&  \&  \&  \&  \&  \& -1 \&  \&  \&  \&  \&  \&  \&  \& -1 \&  \& -1 \&  \&  \&  \&  \&  \& -1 \&  \&  \&  \&  \&  \&  \&  \& \cdots \\
 \&  \&  \& 1 \&  \&  \&  \&  \&  \&  \&  \&  \&  \&  \&  \&  \&  \&  \&  \& 1 \&  \&  \&  \&  \&  \&  \&  \&  \&  \&  \&  \&  \&  \&  \&  \& 1 \&  \&  \&  \\
 \&  \&  \&  \&  \&  \&  \&  \&  \&  \&  \&  \&  \&  \&  \&  \&  \&  \&  \&  \&  \&  \&  \&  \&  \&  \&  \&  \&  \&  \&  \&  \&  \&  \&  \&  \&  \&  \& \\
	};

\end{tikzpicture}\\
 }

\vspace{1cm}

\section{The problem}


We consider the following version of a \textbf{linear recurrence}\footnote{Sometimes called a `system of linear recurrence relations' or a `{system of linear difference equations}'.}: a system of linear equations 
in the sequence of variables $...,x_{-1},x_0,x_1,x_2,...$ which equate each variable to a linear combination of the preceding variables. 
%
%
This includes the (bi-infinite) Fibonacci recurrence:
\begin{equation}
\label{eq: fib}
x_i = x_{i-1}+x_{i-2},\;\;\; \forall i\in \mathbb{Z}
\end{equation}
but it also includes linear recurrences where the equations (and their length) may vary:
\begin{equation}
\label{eq: linrec}
\begin{array}{ll}
x_i = x_{i-1} - x_{i-2}  &  \forall i\in 2 \mathbb{Z} \\
x_i = 2x_{i-1} - x_{i-2} + x_{i-4} &  \forall i\in 2 \mathbb{Z}+1
\end{array}
\end{equation}
The coefficients of the equations are taken from an arbitrary field $\k$, which we fix throughout.



A \textbf{solution} of a linear recurrence is a sequence of numbers in $\k$ satisfying the equations. 
For example, the linear recurrences \eqref{eq: fib} and \eqref{eq: linrec} are respectively solved by the sequences:
\begin{align*}
 ...,-21,13,-8,5,-3, 2, -1, 1, &0, 1, 1, 2,3,5,8,13,21,... \\
..., 1, 0, -1, 0, 1, 2, 1, 0, -1, 0, &1, 2, 1, 0, -1, 0, 1, 2,... 
\end{align*}
The motivating problem of this note is to construct all solutions to a given linear recurrence. 

\begin{rem}
We also consider \textbf{affine recurrences}, in which the equations may have a constant term. The solutions to these systems are constructed in Section \ref{section: affine}.
\end{rem}

\newpage

{\centering
\it Sections 2\--4 summarize the results, and Sections 5\--8 contain the proofs and details.
\\}

\section{Simplifying the problem}

We first reformulate the problem using ideas from finite dimensional linear algebra.

\subsection{Recurrence matrices}

Like finite systems of linear equations, linear recurrences can be reformulated as matrix equations. 
By moving the variables to the left-hand side of each equation and collecting the coefficients into a $\mathbb{Z}\times \mathbb{Z}$-matrix $\C$, we may rewrite the system of linear equations as
\[ \C\mathsf{x} =\mathsf{0} \]
where $\mathsf{x}$ 
is a $\mathbb{Z}$-vector of variables. 

For example, the Fibonacci recurrence may be rewritten as:
\[ 
\begin{tikzpicture}[baseline=(current bounding box.center),
	ampersand replacement=\&,
	]
	\matrix[matrix of math nodes,
		nodes in empty cells,
		inner sep=0pt,
		nodes={anchor=center},
		column sep={.65cm,between origins},
		row sep={.65cm,between origins},
		left delimiter={[},
		right delimiter={]},
	] (M) at (0,0) {
		\dddots \& \& \&  \& \& \&  \\
		 \& 1 \& 0 \& 0 \& 0 \& 0 \& \\
		 \& -1 \& 1 \& 0 \& 0 \& 0 \& \\
		 \& -1 \& -1 \& 1 \& 0 \& 0 \& \\
		 \& 0 \& -1 \& -1 \& 1 \& 0 \& \\
		 \& 0 \& 0 \& -1 \& -1 \& 1 \& \\
		 \& \& \& \& \& \& \dddots \\
	};
\end{tikzpicture}
\begin{tikzpicture}[baseline=(current bounding box.center),
	ampersand replacement=\&,
	]
	\matrix[matrix of math nodes,
		nodes in empty cells,
		inner sep=0pt,
		nodes={anchor=center},
		column sep={.65cm,between origins},
		row sep={.65cm,between origins},
		left delimiter={[},
		right delimiter={]},
	] (M) at (0,0) {
		\vvdots \\
		x_{-2} \\
		x_{-1} \\
		x_0 \\
		x_1 \\
		x_2 \\
		\vvdots \\
	};
\end{tikzpicture}
=\begin{tikzpicture}[baseline=(current bounding box.center),
	ampersand replacement=\&,
	]
	\matrix[matrix of math nodes,
		nodes in empty cells,
		inner sep=0pt,
		nodes={anchor=center},
		column sep={.65cm,between origins},
		row sep={.65cm,between origins},
		left delimiter={[},
		right delimiter={]},
	] (M) at (0,0) {
		\vvdots \\
		0 \\
		0 \\
		0 \\
		0 \\
		0 \\
		\vvdots \\
	};
\end{tikzpicture}
\]

The $\ZZ\times \ZZ$-matrices $\C$ which arise this way are precisely those which are:
\begin{itemize}
	\item \textbf{lower unitriangular}; that is, $\C_{a,b}=0$ if $a<b$ and $\C_{a,a}=1$ for all $a$, and
	\item \textbf{horizontally bounded}; that is, $\C_{a,b}=0$ if $b\ll a$ or $b \gg a$ for all $a$.\footnote{The bounds need not be uniform; i.e. they may depend on $a$. These are also called \emph{row finite matrices}.}
\end{itemize}
We define a \textbf{recurrence matrix} to be a $\ZZ\times \ZZ$-matrix with these two properties. Hence, the original problem is equivalent to describing the kernel of a given recurrence matrix.




\begin{conv}
As can be seen above, recurrence matrices make inefficient use of space: their non-zero coefficients are concentrated below and sufficiently near the main diagonal. To remedy this, we adopt the atypical convention of \textbf{rotating $\mathbb{Z}\times\mathbb{Z}$-matrices $45^\circ$ counterclockwise}.\footnote{This alignment is also chosen to match the existing literature on \emph{friezes}; see Section \ref{section: frieze}.} 
So, a lower unitriangular matrix $\C$ would be formatted as follows:
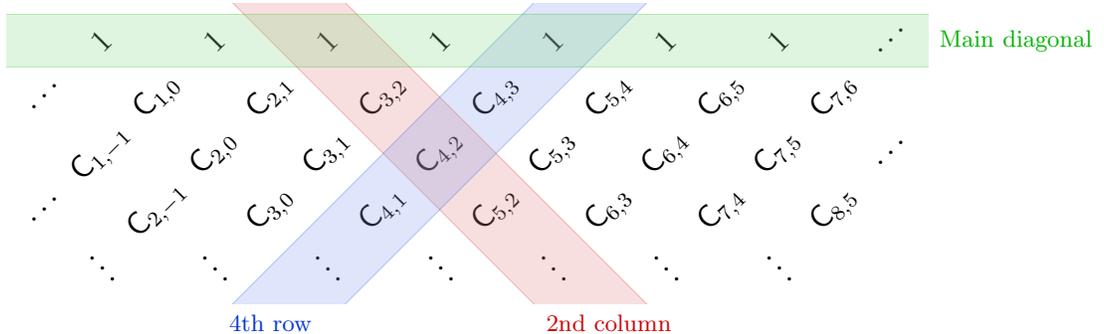
\[ 
\begin{tikzpicture}[baseline=(current bounding box.center),
	ampersand replacement=\&,
	]
	\path[use as bounding box] (-5.25-.5,.5) rectangle (6+.5,-4);
	\begin{scope}
	\clip (-5.25-.5,.5) rectangle (6+.5,-3-.5);
	\matrix[matrix of math nodes,
		matrix anchor = M-4-8.center,
		throw/.style={},
		origin/.style={dark green,draw,circle,inner sep=0.25mm,minimum size=2mm},
		pivot/.style={draw,circle,inner sep=0.25mm,minimum size=2mm},		
		nodes in empty cells,
		inner sep=0pt,
		nodes={anchor=center,rotate=45},
		column sep={.75cm,between origins},
		row sep={.75cm,between origins},
	] (M) at (0,0) {
	 \& \&  \& \&  \& \&  \& \&  \& \&  \& \&  \& \&  \&  \\
	 \& \&  \& \&  \& \&  \& \&  \& \&  \& \&  \& \&  \&  \\
	 \& \&  \& \&  \& \&  \& \&  \& \&  \& \&  \& \&  \&  \\
	 \& |[throw]| 1 \& \& |[throw]| 1 \& \& |[throw]| 1 \& \& |[throw]| 1 \& \&|[throw]| 1 \& \& |[throw]| 1 \& \& |[throw]| 1 \& \& \cdots \\
	\cdots \& \& \C_{1,0} \& \& \C_{2,1} \& \& \C_{3,2} \& \&\C_{4,3} \& \&\C_{5,4} \& \&\C_{6,5} \& \&\C_{7,6} \&  \\
	 \&\C_{1,-1} \& \&\C_{2,0} \& \&\C_{3,1} \& \&\C_{4,2} \& \&\C_{5,3} \& \&\C_{6,4} \& \&\C_{7,5} \& \& \cdots \\
	\cdots \& \&\C_{2,-1} \& \&\C_{3,0} \& \&\C_{4,1} \& \&\C_{5,2} \& \&\C_{6,3} \& \&\C_{7,4} \& \&\C_{8,5} \&  \\
	 \& \vvdots \& \& \vvdots \& \& \vvdots \& \& \vvdots \& \& \vvdots \& \& \vvdots \& \& \vvdots \& \&  \\
	 \& \&  \& \&  \& \&  \& \&  \& \&  \& \&  \& \&  \&  \\
	};
	
	\draw[dark red, fill= dark red!50,opacity=.25,rounded corners] (M-1-2.center) -- (M-9-10.center) -- (M-9-12.center) -- (M-1-4.center) -- cycle;
	\draw[dark blue, fill= dark blue!50,opacity=.25,rounded corners] (M-1-14.center) -- (M-9-6.center) -- (M-9-4.center) -- (M-1-12.center) -- cycle;
	\draw[dark green, fill= dark green!50,opacity=.25,rounded corners] (-8,.35) rectangle (8.75,-.35);
	\end{scope}
	\node[dark red] at (M-9-11) {\scriptsize $2$nd column};
	\node[dark blue] at (M-9-5) {\scriptsize $4$th row};
	\node[dark green,right] at (6.5,0) {\scriptsize Main diagonal};

\end{tikzpicture}
\]
Any omitted entries (including those above the main diagonal) are interpreted as $0$.
\end{conv}

\begin{ex}\label{ex: fib}
The Fibonacci recurrence \eqref{eq: fib} corresponds to the recurrence matrix:
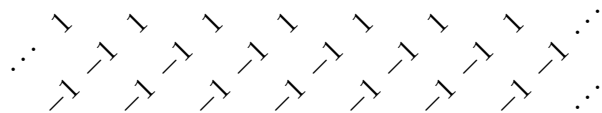
\[ 
\begin{tikzpicture}[baseline=(current bounding box.center),
	ampersand replacement=\&,
	]
	\matrix[matrix of math nodes,
		matrix anchor = M-4-8.center,
		throw/.style={},
		origin/.style={},
		nodes in empty cells,
		inner sep=0pt,
		nodes={anchor=center,rotate=45},
		column sep={.5cm,between origins},
		row sep={.5cm,between origins},
	] (M) at (0,0) {
	 \& |[throw]| 1 \& \& |[throw]| 1 \& \& |[throw]| 1 \& \& |[origin]| 1 \& \&|[throw]| 1 \& \& |[throw]| 1 \& \& |[throw]| 1 \& \& \cdots \\
	\cdots \& \& -1 \& \& -1 \& \& -1 \& \& -1 \& \& -1 \& \& -1 \& \& -1 \&  \\
	 \& -1 \& \& -1 \& \& -1 \& \& -1 \& \& -1 \& \& -1 \& \& -1 \& \& \cdots \\
	};
\end{tikzpicture}
\]
The linear recurrence \eqref{eq: linrec} corresponds to the recurrence matrix:
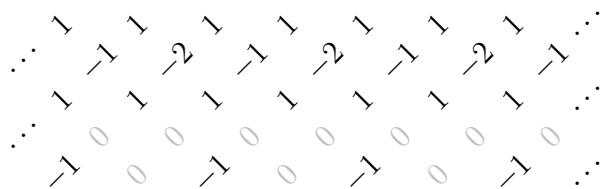
\[ 
\begin{tikzpicture}[baseline=(current bounding box.center),
	ampersand replacement=\&,
	]
	\matrix[matrix of math nodes,
		matrix anchor = M-4-8.center,
		throw/.style={},
		origin/.style={},
		faded/.style={black!25},
		nodes in empty cells,
		inner sep=0pt,
		nodes={anchor=center,rotate=45},
		column sep={.5cm,between origins},
		row sep={.5cm,between origins},
	] (M) at (0,0) {
	 \& |[throw]| 1 \& \& |[throw]| 1 \& \& |[throw]| 1 \& \& |[origin]| 1 \& \&|[throw]| 1 \& \& |[throw]| 1 \& \& |[throw]| 1 \& \& \cdots \\
	\cdots \& \& -1 \& \& -2 \& \& -1 \& \& -2 \& \& -1 \& \& -2 \& \& -1 \&  \\
	 \& 1 \& \& 1 \& \& 1 \& \& 1 \& \& 1 \& \& 1 \& \& 1 \& \& \cdots \\
	\cdots \& \& |[faded]| 0 \& \& |[faded]| 0 \& \& |[faded]| 0 \& \& |[faded]| 0 \& \& |[faded]| 0 \& \& |[faded]| 0 \& \& |[faded]| 0 \&  \\
	 \& -1 \& \& |[faded]| 0 \& \& -1 \& \& |[faded]| 0 \& \& -1 \& \& |[faded]| 0 \& \& -1 \& \& \cdots \\
	};
\end{tikzpicture}
\]
The zeroes in grey could have been suppressed.
\end{ex}


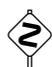
\begin{warn}{\marginnote{\dbend}[.05cm]}
The reader should keep the following two pathologies in mind when working with infinite matrices and vectors.
\begin{itemize}
	\item The product of an arbitrary $\mathbb{Z}\times \mathbb{Z}$-matrix with a $\mathbb{Z}$-vector or another $\mathbb{Z}\times\mathbb{Z}$-matrix may involve summing an infinite number of non-zero terms in $\k$. When this occurs, we simply say \bad{the product does not exist}.\footnote{While one might allow infinite sums that absolutely converge in some topology on $\k$, we won't consider this.}
	\item Even when all constituent products exist, \bad{multiplication may not be associative}. The corresponding entries of $(\mathsf{AB})\mathsf{C}$ and $\mathsf{A}(\mathsf{BC})$ are defined by double infinite sums which coincide except for the order of summation, which cannot be exchanged in general. 
\end{itemize}


Mercifully, the product $\mathsf{AC}$ exists and the equality $(\mathsf{AB})\mathsf{C}=\mathsf{A}(\mathsf{BC})$ holds whenever (a) $\mathsf{A}$, $\mathsf{B}$, and $\mathsf{C}$ are lower unitriangular, or (b) $\mathsf{A}$ and $\mathsf{B}$ are horizontally bounded and $\mathsf{C}$ is any matrix or vector. We will make extensive use of both of these special cases.
%
%
\end{warn}

\subsection{Reduced recurrence matrices}

As we are primarily interested in the kernel of a recurrence matrix, let us say...
\begin{itemize}
	\item ...a recurrence matrix is \textbf{trivial} if its kernel is trivial (i.e.~only the zero vector), and...
	\item ...two recurrence matrices are \textbf{equivalent} if their kernels are equal.
\end{itemize}
For example, a recurrence matrix is trivial if and only if it is equivalent to the $\ZZ\times \ZZ$ identity matrix, whose kernel is clearly $\{0\}$. This example can be extended to the following theorem.\footnote{All theorems stated in Sections 2\--4 are proven in a later section; e.g. Theorem \ref{thm: triv-equiv} is proven in Section 5.}

\begin{prethm}{\ref{thm: triv-equiv}}
Two recurrence matrices $\C$ and $\C'$ are equivalent if and only $\C'=\D\C$ for some trivial recurrence matrix $\D$.
\end{prethm}

We would like to characterize a representative of each equivalence class which is `minimal' in some sense. In finite linear algebra, this is accomplished by {row reducing} a matrix, but translating row reduction into our current setting runs into several issues.
\begin{itemize}
	\item Only one type of elementary row operation can send recurrence matrices to recurrence matrices; specifically, adding a multiple of one row to a lower row.
	\item There is no `upper left corner' in which to start an elimination algorithm.
	\item Sequences of row operations need not terminate, and two infinite sequences of row operations may have different limits (in an appropriate topology, see Section \ref{section: limits}).
\end{itemize}
%
%

Nevertheless, we may characterize those recurrence matrices which cannot be row reduced any further. A recurrence matrix is \textbf{reduced} if the first non-zero entry in each row (called the \textbf{pivot} entry in that row) is the last non-zero entry in its column. 

\begin{ex}\label{ex: pivots}
The following recurrence matrix is reduced.
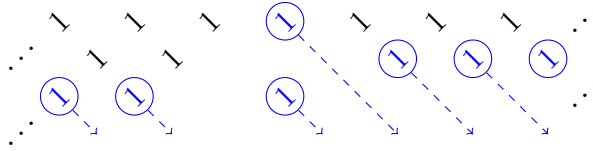
\[
\begin{tikzpicture}[
	nodes={anchor=center},
	ascending/.style={blue,opacity=.5},
	descending/.style={red,opacity=.5},
	]
	\clip[use as bounding box] (-6.30,-1.8) rectangle (5.80,.35);
	\matrix[matrix of math nodes,
		throw/.style={},
		pivot/.style={blue,draw,circle,inner sep=0mm,minimum size=5mm},
		matrix anchor = M-1-9.west,
		nodes in empty cells,
		nodes={anchor=center,rotate=45},
		column sep={.5cm,between origins},
		row sep={.5cm,between origins},
	] (M) at (0,0) {
	 & |[throw]| 1 & & |[throw]| 1 & & |[throw]| 1 & & |[pivot]| 1 & &|[throw]| 1 & & |[throw]| 1 & & |[throw]| 1 & & \cdots \\
	\cdots & & 1 & & 1 & &  & &  & & |[pivot]| 1 & & |[pivot]| 1 & & |[pivot]| 1 &  \\
	 & |[pivot]| 1 & & |[pivot]| 1 & &  & & |[pivot]| 1 & &  & &  & &  & & \cdots \\
	\cdots & &  & &  & &  & &  & &  & &  & &  &  \\
	};
	\draw[blue,dashed,->] (M-3-2) to (M-4-3.center);
	\draw[blue,dashed,->] (M-3-4) to (M-4-5.center);
	\draw[blue,dashed,->] (M-3-8) to (M-4-9.center);
	\draw[blue,dashed,->] (M-1-8) to (M-4-11.center);
	\draw[blue,dashed,->] (M-2-11) to (M-4-13.center);
	\draw[blue,dashed,->] (M-2-13) to (M-4-15.center);
\end{tikzpicture}
\]
Below each pivot (in \textcolor{blue}{blue circles}) there are only zeroes (along the dashed arrows).
\end{ex}


\begin{prethm}{\ref{thm: reduced1}}
Every recurrence matrix is equivalent to a unique reduced recurrence matrix.
\end{prethm}

We provide two independent proofs of this fact. The proof in Section \ref{section: GZ} uses Zorn's Lemma to show that the the reduced matrix the limit of all `generalized row reductions'. The proof in Section \ref{section: rank} (Theorem \ref{thm: reduced2}) constructs the reduced matrix directly from the space of solutions.
%



\section{Constructing the solutions}

In this section, we construct all elements of the kernel of a reduced recurrence matrix, and thus all solutions to a reduced linear recurrence.

\begin{nota}
For $a\leq b\in \ZZ$, let $[a,b]\coloneqq\{a,a+1,...,b\}\subset \ZZ$. For a $\ZZ\times\ZZ$-matrix $\C$ and two sets $I,J\subset \ZZ$, let $\C_{I,J}$ denote the submatrix on row set $I$ and column set $J$.
\end{nota}

\subsection{Adjugates}
Given a lower unitriangular $\ZZ \times \ZZ$-matrix $\C$, define the \textbf{adjugate} $\adj(\C)$ of $\C$ to be 
the lower unitriangular $\mathbb{Z}\times\mathbb{Z}$-matrix whose subdiagonal entries are defined by
\[ \adj(\mathsf{C})_{a,b} \coloneqq (-1)^{a+b}\mathrm{det}(\mathsf{C}_{[b+1,a],[b,a-1]}) \]
Three examples of adjugates are given in Figure \ref{fig: adjsol}.
\begin{rem}
The determinant of the finite matrix $\mathsf{C}_{[b+1,a],[b,a-1]}$ coincides with the determinant\footnote{While general $\ZZ\times \ZZ$-matrices may not have well-defined determinants, the cofactor matrices of lower unitriangular matrices are lower unitriangular outside of a finite square, and thus have a well-defined determinant.} of the infinite cofactor matrix $\C_{\ZZ\smallsetminus \{b\},\ZZ\smallsetminus \{a\}}$, justifying the name `adjugate'. 
\end{rem}

Like its finite matrix counterpart, the adjugate matrix has many useful properties.
\begin{prop}\label{prop: adjprops}
Let $\C$ and $\mathsf{D}$ be lower unitriangular $\mathbb{Z}\times \mathbb{Z}$-matrices.
\begin{enumerate}
	\item $\adj(\C)\C =\C\adj(\C) = \mathrm{Id}$.
	\item $\adj(\adj(\C))=\C$.
	\item $\adj(\C\mathsf{D}) = \adj(\mathsf{D})\adj(\C)$.
	\item For $I,J\subset [a,b]$ with $|I|=|J|$,  $\det(\adj(\C)_{I,J}) = (-1)^{\sum I+\sum J}\det({\C}_{[a,b]\smallsetminus J,[a,b]\smallsetminus I}) $.
\end{enumerate}
\end{prop}
\begin{proof}
Since each entry of $\adj(\C)$ only depends on a finite submatrix of $\C$, these follow from their finite matrix counterparts. E.g.~(4) follows by considering the submatrix $\C_{[a,b],[a,b]}$.
\end{proof}

The first three parts of the prior proposition can be rephrased as follows.

\begin{prop}
The lower unitriangular $\ZZ \times \ZZ$-matrices form a group with inverse $\adj$.
\end{prop}

\begin{warn}
{\marginnote{\dbend}[.05cm]}
Since multiplication is not always associative, \bad{inverses may not be unique} in the larger set of $\ZZ\times \ZZ$-matrices (see Remark \ref{rem: threeinverses}), and so we write $\adj(\C)$ instead of $\C^{-1}$.
\end{warn}

The recurrence matrices do not form a subgroup of the lower unitriangular matrices, as they are not closed under $\adj$. In fact, this only holds in uninteresting cases: the only recurrence matrices whose adjugate is also a recurrence matrix are the trivial ones.
\begin{prethm}{\ref{thm: triv-inv}}
A recurrence matrix $\C$ is trivial if and only if $\adj(\C)$ is a recurrence matrix.
\end{prethm}

\noindent Since $\adj(\C)$ is lower unitriangular by construction, $\C$ is trivial if and only if $\adj(\C)$ is horizontally bounded.

\begin{rem}
The recurrence matrices form a semigroup whose subgroup of invertible elements is the set of trivial recurrence matrices (by Theorem \ref{thm: triv-inv}). The left orbits of this subgroup on the set of recurrence matrices are the equivalence classes (by Theorem \ref{thm: triv-equiv}), and the reduced recurrence matrices are a transverse of these orbits (by Theorem \ref{thm: reduced1}).
%
\end{rem}

\begin{figure}[h!!]
\begin{tabular}{|c||c||c|}
\hline 
\multicolumn{3}{|c|}{Reduced recurrence matrices} \\
\hline
\begin{tikzpicture}[baseline=(current bounding box.center),
	ampersand replacement=\&,
	]
	\clip[use as bounding box] (-2.25,.3) rectangle (2.55,-.9);
	\matrix[matrix of math nodes,
		matrix anchor = M-4-8.center,
		origin/.style={},
		throw/.style={},
		nodes in empty cells,
		inner sep=0pt,
		nodes={anchor=center,node font=\scriptsize,rotate=45},
		column sep={.3cm,between origins},
		row sep={.3cm,between origins},
	] (M) at (0,0) {
	 \& \&  \& \&  \& \&  \& \&  \& \&  \& \&  \& \&  \&  \\
	 \& \&  \& \&  \& \&  \& \&  \& \&  \& \&  \& \&  \&  \\
	 \& \&  \& \&  \& \&  \& \&  \& \&  \& \&  \& \&  \&  \\
	 \& |[throw]| 1 \& \& |[throw]| 1 \& \& |[throw]| 1 \& \& |[origin]| 1 \& \&|[throw]| 1 \& \& |[throw]| 1 \& \& |[throw]| 1 \& \& \cdots \\
	\cdots \& \& -1 \& \& -1 \& \& -1 \& \& -1 \& \& -1 \& \& -1 \& \& -1 \&  \\
	 \& -1 \& \& -1 \& \& -1 \& \& -1 \& \& -1 \& \& -1 \& \& -1 \& \& \cdots \\
	 \& \&  \& \&  \& \&  \& \&  \& \&  \& \&  \& \&  \&  \\
	 \& \&  \& \&  \& \&  \& \&  \& \&  \& \&  \& \&  \&  \\
	 \& \&  \& \&  \& \&  \& \&  \& \&  \& \&  \& \&  \&  \\
	};
	
\end{tikzpicture}
&
\begin{tikzpicture}[baseline=(current bounding box.center),
	ampersand replacement=\&,
	]
	\clip[use as bounding box] (-2.25,.3) rectangle (2.55,-.9);
	\matrix[matrix of math nodes,
		matrix anchor = M-4-8.center,
		origin/.style={},
		throw/.style={},
		nodes in empty cells,
		inner sep=0pt,
		nodes={anchor=center,node font=\scriptsize,rotate=45},
		column sep={.3cm,between origins},
		row sep={.3cm,between origins},
	] (M) at (0,0) {
	 \& \&  \& \&  \& \&  \& \&  \& \&  \& \&  \& \&  \&  \\
	 \& \&  \& \&  \& \&  \& \&  \& \&  \& \&  \& \&  \&  \\
	 \& \&  \& \&  \& \&  \& \&  \& \&  \& \&  \& \&  \&  \\
	 \& |[throw]| 1 \& \& |[throw]| 1 \& \& |[throw]| 1 \& \& |[origin]| 1 \& \&|[throw]| 1 \& \& |[throw]| 1 \& \& |[throw]| 1 \& \& \cdots \\
	\cdots \& \& -2 \& \& -1 \& \& -2 \& \& -1 \& \& -2 \& \& -1 \& \& -2 \&  \\
	 \& 1 \& \& 1 \& \& 1 \& \& 1 \& \& 1 \& \& 1 \& \& 1 \& \& \cdots \\
	 \& \&  \& \&  \& \&  \& \&  \& \&  \& \&  \& \&  \&  \\
	 \& \&  \& \&  \& \&  \& \&  \& \&  \& \&  \& \&  \&  \\
	 \& \&  \& \&  \& \&  \& \&  \& \&  \& \&  \& \&  \&  \\
	};
\end{tikzpicture}
&
\begin{tikzpicture}[baseline=(current bounding box.center),
	ampersand replacement=\&,
	]
	\clip[use as bounding box] (-2.25,.3) rectangle (2.55,-.9);
	\matrix[matrix of math nodes,
		matrix anchor = M-4-8.center,
		origin/.style={},
		throw/.style={},
		pivot/.style={},
		nodes in empty cells,
		inner sep=0pt,
		nodes={anchor=center,node font=\scriptsize,rotate=45},
		column sep={.3cm,between origins},
		row sep={.3cm,between origins},
	] (M) at (0,0) {
	 \& \&  \& \&  \& \&  \& \&  \& \&  \& \&  \& \&  \&  \\
	 \& \&  \& \&  \& \&  \& \&  \& \&  \& \&  \& \&  \&  \\
	 \& \&  \& \&  \& \&  \& \&  \& \&  \& \&  \& \&  \&  \\
	 \& |[throw]| 1 \& \& |[throw]| 1 \& \& |[throw]| 1 \& \& |[origin]| 1 \& \&|[pivot]| 1 \& \& |[throw]| 1 \& \& |[throw]| 1 \& \& \cdots \\
	\cdots \& \& -1 \& \& -1 \& \& -1 \& \&  \& \&  \& \& |[pivot]|  -1 \& \& |[pivot]| -1 \& \\
	 \& |[pivot]| -1 \&  \& |[pivot]| -1 \& \& |[pivot]| -1 \& \&  \& \& |[pivot]| -1 \& \& \& \&  \& \& \cdots \\
	 \& \&  \& \&  \& \&  \& \& \&  \&  \& \&  \& \&  \&  \\
	 \& \&  \& \&  \& \&  \& \&  \& \&  \& \&  \& \&  \&  \\
	 \& \&  \& \&  \& \&  \& \&  \& \&  \& \&  \& \&  \&  \\
	};
\end{tikzpicture}
\\
\hline
\multicolumn{3}{|c|}{Adjugate matrices} \\
%
%
\hline
\begin{tikzpicture}[baseline=(current bounding box.center),
	ampersand replacement=\&,
	]
	\clip[use as bounding box] (-2.25,.3) rectangle (2.55,-2.4);
	\matrix[matrix of math nodes,
		matrix anchor = M-1-8.center,
		origin/.style={},
		throw/.style={},
		pivot/.style={draw,circle,inner sep=0.25mm,minimum size=2mm},		
		nodes in empty cells,
		inner sep=0pt,
		nodes={anchor=center,node font=\scriptsize,rotate=45},
		column sep={.3cm,between origins},
		row sep={.3cm,between origins},
	] (M) at (0,0) {
	 \& |[throw]| 1 \& \& |[throw]| 1 \& \& |[throw]| 1 \& \& |[origin]| 1 \& \&|[throw]| 1 \& \& |[throw]| 1 \& \& |[throw]| 1 \& \& \cdots \\
	\cdots \& \& 1 \& \& 1 \& \& 1 \& \& 1 \& \& 1 \& \& 1 \& \& 1 \&  \\
	 \& 2 \& \& 2 \& \& 2 \& \& 2 \& \& 2 \& \& 2 \& \& 2 \& \& \cdots \\
	\cdots \& \& 3 \& \& 3 \& \& 3 \& \& 3 \& \& 3 \& \& 3 \& \& 3 \&  \\
	 \& 5 \& \& 5 \& \& 5 \& \& 5 \& \& 5 \& \& 5 \& \& 5 \& \& \cdots \\
	\cdots \& \& 8 \& \& 8 \& \& 8 \& \& 8 \& \& 8 \& \& 8 \& \& 8 \&  \\
	 \& 13 \& \& 13 \& \& 13 \& \& 13 \& \& 13 \& \& 13 \& \& 13 \& \& \cdots \\
	\cdots \& \& \cdots \& \& \cdots \& \& \cdots \& \& \cdots \& \& \cdots \& \& \cdots \& \& \cdots \&  \\
	};
\end{tikzpicture}
&
\begin{tikzpicture}[baseline=(current bounding box.center),
	ampersand replacement=\&,
	]
	\clip[use as bounding box] (-2.25,.3) rectangle (2.55,-2.4);
	\matrix[matrix of math nodes,
		matrix anchor = M-1-8.center,
		origin/.style={},
		throw/.style={},
		pivot/.style={draw,circle,inner sep=0.25mm,minimum size=2mm},		
		nodes in empty cells,
		inner sep=0pt,
		nodes={anchor=center,node font=\scriptsize,rotate=45},
		column sep={.3cm,between origins},
		row sep={.3cm,between origins},
	] (M) at (0,0) {
	 \& |[throw]| 1 \& \& |[throw]| 1 \& \& |[throw]| 1 \& \& |[origin]| 1 \& \& |[throw]| 1 \& \&|[throw]| 1 \& \& |[throw]| 1 \& \& \cdots \\
	\cdots \& \& 2 \& \& 1 \& \& 2 \& \& 1 \& \& 2 \& \& 1 \& \& 2 \&  \\
	 \& 1 \& \& 1 \& \& 1 \& \& 1 \& \& 1 \& \& 1 \& \& 1 \& \& \cdots \\
	\cdots \& \& 0 \& \& 0 \& \&0 \& \& 0 \& \& 0 \& \& 0 \& \& 0 \&  \\
	 \& -1 \& \& -1 \& \& -1 \& \& -1 \& \& -1 \& \& -1 \& \& -1 \& \& \cdots \\
	\cdots \& \& -2 \& \& -1 \& \& -2 \& \& -1 \& \& -2 \& \& -1 \& \& -2 \&  \\
	 \& -1 \& \& -1 \& \& -1 \& \& -1 \& \& -1 \& \& -1 \& \& -1 \& \& \cdots \\
	\cdots \& \& \cdots \& \& \cdots \& \& \cdots \& \& \cdots \& \& \cdots \& \& \cdots \& \& \cdots \& \\
	};
\end{tikzpicture}
&
\begin{tikzpicture}[baseline=(current bounding box.center),
	ampersand replacement=\&,
	]
	\clip[use as bounding box] (-2.25,.3) rectangle (2.55,-2.4);
	\matrix[matrix of math nodes,
		matrix anchor = M-1-8.center,
		origin/.style={},
		throw/.style={},
		pivot/.style={draw,circle,inner sep=0.25mm,minimum size=2mm},		
		nodes in empty cells,
		inner sep=0pt,
		nodes={anchor=center,node font=\scriptsize,rotate=45},
		column sep={.3cm,between origins},
		row sep={.3cm,between origins},
	] (M) at (0,0) {
	 \& |[throw]| 1 \& \& |[throw]| 1 \& \& |[throw]| 1 \& \& |[origin]| 1 \& \& |[throw,double]| 1 \& \&|[throw]| 1 \& \& |[throw]| 1 \& \& \cdots \\
	\cdots \& \& 1 \& \& 1 \& \& 1 \& \& 0 \& \& 0 \& \& 1 \& \& 1 \& \\
	\& 2 \& \& 2 \& \& 2 \& \& 0 \& \& 1 \& \& 0 \& \& 1 \& \& \cdots \\
	\cdots \& \& 3 \& \& 3 \& \& 0 \& \& 1 \& \& 1 \& \& 0 \& \& 1 \&   \\
	  \& 5 \& \& 5 \& \& 0 \& \& 2 \& \& 1 \& \& 1 \& \& 0 \& \&  \cdots \\
	\cdots \& \& 8 \& \& 0 \& \& 3 \& \& 2 \& \& 1 \& \& 1 \& \& 0 \& \\
	  \& 13 \& \& 0 \& \& 5 \& \& 3 \& \& 2 \& \& 1 \& \& 1 \& \&  \cdots \\
	\cdots \& \& \cdots \& \& \cdots \& \& \cdots \& \& \cdots \& \& \cdots \& \& \cdots \& \& \cdots \& \\
	};
\end{tikzpicture}
\\
\hline
\multicolumn{3}{|c|}{Solution matrices} \\
\hline
\begin{tikzpicture}[baseline=(current bounding box.center),
	ampersand replacement=\&,
	]
	\clip[use as bounding box] (-2.25,2.4) rectangle (2.55,-2.4);
	\matrix[matrix of math nodes,
		matrix anchor = M-8-8.center,
		origin/.style={},
		throw/.style={},
		pivot/.style={dark blue,draw,rounded rectangle,inner sep=0.25mm,minimum size=2mm},
		nodes in empty cells,
		inner sep=0pt,
		nodes={anchor=center,node font=\scriptsize,rotate=45},
		column sep={.3cm,between origins},
		row sep={.3cm,between origins},
	] (M) at (0,0) {
	 \& \& \cdots \& \& \cdots \& \& \cdots \& \& \cdots \& \& \cdots \& \& \cdots \& \& \cdots \&  \\
	 \& 5 \& \& 5 \& \& 5 \& \& 5 \& \& 5 \& \& 5 \& \& 5 \& \& \cdots \\
	\cdots \& \& -3 \& \& -3 \& \& -3 \& \& -3 \& \& -3 \& \& -3 \& \& -3 \&  \\
	 \& 2 \& \& 2 \& \& 2 \& \& 2 \& \& 2 \& \& 2 \& \& 2 \& \& \cdots \\
	\cdots \& \& -1 \& \& -1 \& \& -1 \& \& -1 \& \& -1 \& \& -1 \& \& -1 \&  \\
	 \& 1 \& \& 1 \& \& 1 \& \& 1 \& \& 1 \& \& 1 \& \& 1 \& \& \cdots \\
	\cdots \& \& 0 \& \& 0 \& \& 0 \& \& 0 \& \& 0 \& \& 0 \& \& 0 \&  \\
	 \& |[throw]| 1 \& \& |[throw]| 1 \& \& |[throw]| 1 \& \& |[origin]| 1 \& \&|[throw]| 1 \& \& |[throw]| 1 \& \& |[throw]| 1 \& \& \cdots \\
	\cdots \& \& 1 \& \& 1 \& \& 1 \& \& 1 \& \& 1 \& \& 1 \& \& 1 \&  \\
	 \& 2 \& \& 2 \& \& 2 \& \& 2 \& \& 2 \& \& 2 \& \& 2 \& \& \cdots \\
	\cdots \& \& 3 \& \& 3 \& \& 3 \& \& 3 \& \& 3 \& \& 3 \& \& 3 \&  \\
	 \& 5 \& \& 5 \& \& 5 \& \& 5 \& \& 5 \& \& 5 \& \& 5 \& \& \cdots \\
	\cdots \& \& 8 \& \& 8 \& \& 8 \& \& 8 \& \& 8 \& \& 8 \& \& 8 \&  \\
	 \& 13 \& \& 13 \& \& 13 \& \& 13 \& \& 13 \& \& 13 \& \& 13 \& \& \cdots \\
	\cdots \& \& \cdots \& \& \cdots \& \& \cdots \& \& \cdots \& \& \cdots \& \& \cdots \& \& \cdots \&  \\
	};

\end{tikzpicture}
&
\begin{tikzpicture}[baseline=(current bounding box.center),
	ampersand replacement=\&,
	]
	\clip[use as bounding box] (-2.25,2.4) rectangle (2.55,-2.4);
	\matrix[matrix of math nodes,
		matrix anchor = M-8-8.center,
		origin/.style={},
		throw/.style={},
		pivot/.style={dark blue,draw,rounded rectangle,inner sep=0.25mm,minimum size=2mm},
		nodes in empty cells,
		inner sep=0pt,
		nodes={anchor=center,node font=\scriptsize,rotate=45},
		column sep={.3cm,between origins},
		row sep={.3cm,between origins},
	] (M) at (0,0) {
	 \& \& \cdots \& \& \cdots \& \& \cdots \& \& \cdots \& \& \cdots \& \& \cdots \& \& \cdots \&  \\
	 \& 1 \& \& 1 \& \& 1 \& \& 1 \& \& 1 \& \& 1 \& \& 1 \& \& \cdots \\
	\cdots \& \& 0 \& \& 0 \& \& 0 \& \& 0 \& \& 0 \& \& 0 \& \& 0 \&  \\
	 \& -1 \& \& -1 \& \& -1 \& \& -1 \& \& -1 \& \& -1 \& \& -1 \& \& \cdots \\
	\cdots \& \& -2 \& \& -1 \& \& -2 \& \& -1 \& \& -2 \& \& -1 \& \& -2 \&  \\
	 \& -1 \& \& -1 \& \& -1 \& \& -1 \& \& -1 \& \& -1 \& \& -1 \& \& \cdots \\
	\cdots \& \& 0 \& \& 0 \& \& 0 \& \& 0 \& \& 0 \& \& 0 \& \& 0 \&  \\
	 \& |[throw]| 1 \& \& |[throw]| 1 \& \& |[throw]| 1 \& \& |[origin]| 1 \& \&|[throw]| 1 \& \& |[throw]| 1 \& \& |[throw]| 1 \& \& \cdots \\
	\cdots \& \& 2 \& \& 1 \& \& 2 \& \& 1 \& \& 2 \& \& 1 \& \& 2 \&  \\
	 \& 1 \& \& 1 \& \& 1 \& \& 1 \& \& 1 \& \& 1 \& \& 1 \& \& \cdots \\
	\cdots \& \& 0 \& \& 0 \& \&0 \& \& 0 \& \& 0 \& \& 0 \& \& 0 \&  \\
	 \& -1 \& \& -1 \& \& -1 \& \& -1 \& \& -1 \& \& -1 \& \& -1 \& \& \cdots \\
	\cdots \& \& -2 \& \& -1 \& \& -2 \& \& -1 \& \& -2 \& \& -1 \& \& -2 \&  \\
	 \& -1 \& \& -1 \& \& -1 \& \& -1 \& \& -1 \& \& -1 \& \& -1 \& \& \cdots \\
	\cdots \& \& \cdots \& \& \cdots \& \& \cdots \& \& \cdots \& \& \cdots \& \& \cdots \& \& \cdots \&  \\
	};

\end{tikzpicture}
&
\begin{tikzpicture}[baseline=(current bounding box.center),
	ampersand replacement=\&,
	]
	\clip[use as bounding box] (-2.25,2.4) rectangle (2.55,-2.4);
	\matrix[matrix of math nodes,
		matrix anchor = M-8-8.center,
		origin/.style={},
		throw/.style={},
		pivot/.style={dark blue,draw,rounded rectangle,inner sep=0.25mm,minimum size=2mm},
		nodes in empty cells,
		inner sep=0pt,
		nodes={anchor=center,node font=\scriptsize,rotate=45},
		column sep={.3cm,between origins},
		row sep={.3cm,between origins},
	] (M) at (0,0) {
	 \& \& \cdots \& \& \cdots \& \& \cdots \& \& \cdots \& \& \cdots \& \& \cdots \& \& \cdots \&  \\
	 \& 5 \&  \& 0 \& \& 2 \& \& -1 \& \& 1 \& \& 0 \& \& 1 \& \& \cdots \\
	\cdots \& \& -3 \& \& 0 \& \& -1 \& \& 1 \& \& 0 \& \& 1 \& \& 0 \&  \\
	 \& 2 \&  \& 2 \& \& 0 \& \& 1 \& \& 0 \& \& 1 \& \& 0 \& \& \cdots \\
	\cdots \& \& -1 \& \& -1 \& \& 0 \& \& 0 \& \& 1 \& \& 0 \& \& 1 \&  \\
	 \& 1 \&  \& 1 \& \& 1 \& \& 0 \& \& 1 \& \& 0 \& \& 1 \& \& \cdots \\
	\cdots \& \& 0 \& \& 0 \& \& 0 \& \& 0 \& \& 0 \& \& 1 \& \& 1 \&  \\
	 \& |[throw]| 1 \& \& |[throw]| 1 \& \& |[throw]| 1 \& \& |[origin]| 1 \& \& |[throw,double]| 0 \& \&|[throw]| 1 \& \& |[throw]| 1 \& \& \cdots \\
	\cdots \& \& 1 \& \& 1 \& \& 1 \& \& 0 \& \& 0 \& \& 1 \& \& 1 \& \\
	\& 2 \& \& 2 \& \& 2 \& \& 0 \& \& 1 \& \& 0 \& \& 1 \& \& \cdots \\
	\cdots \& \& 3 \& \& 3 \& \& 0 \& \& 1 \& \& 1 \& \& 0 \& \& 1 \&   \\
	  \& 5 \& \& 5 \& \& 0 \& \& 2 \& \& 1 \& \& 1 \& \& 0 \& \&  \cdots \\
	\cdots \& \& 8 \& \& 0 \& \& 3 \& \& 2 \& \& 1 \& \& 1 \& \& 0 \& \\
	  \& 13 \& \& 0 \& \& 5 \& \& 3 \& \& 2 \& \& 1 \& \& 1 \& \&  \cdots \\
	\cdots \& \& \cdots \& \& \cdots \& \& \cdots \& \& \cdots \& \& \cdots \& \& \cdots \& \& \cdots \& \\
	};

\end{tikzpicture} \\
\hline
\end{tabular}
\caption{Examples of adjugate and solution matrices}
\label{fig: adjsol}
\end{figure}


\subsection{The solution matrix}\label{section: sol}

In this section, we construct a matrix $\sol(\C)$ whose image is the kernel of a given reduced recurrence matrix $\C$.

The \textbf{shape} of a recurrence matrix $\C$ is the non-increasing function $S:\ZZ\rightarrow\ZZ$ defined by
\[ S(a) := \min\{ b \in \ZZ \mid \C_{a,b}\neq 0\} \]
The pivot entry in the $a$th row is then $\C_{a,S(a)}$, and so $\C$ is reduced if and only if $\C_{b,S(a)}=0$ for all $b>a$. In particular, the shape of a reduced recurrence matrix must be injective.

\begin{ex}
The shape of the Fibonacci recurrence matrix is $S(a)=a-2$.
%
The shape of the recurrence matrix depicted in Example \ref{ex: pivots} is 
\[ S(a) = \left\{ \begin{array}{cc}
a-2 & \text{if $a<0$} \\
a & \text{if $a=0$} \\
a-2 & \text{if $a=1$} \\
a-1 & \text{if $a>1$}
\end{array}\right\}\qedhere\] 
\end{ex}

%

%
Given a recurrence matrix $\C$ with shape $S$, define a $\mathbb{Z}\times \mathbb{Z}$-matrix $\mathsf{P}$ by
\[ \P_{a,b} \coloneqq \left\{ \begin{array}{cc}
(\C_{b,S(b)})^{-1} & \text{if }a=S(b) \\
0 & \text{otherwise}
\end{array}\right\}\]
This is the generalized permutation matrix such that $\mathsf{CP}$ moves the pivots to the main diagonal and rescales them to $1$.
Thus, $\C$ is reduced if and only if $\C\mathsf{P}$ is upper unitriangular.\footnote{A $\mathbb{Z}\times \mathbb{Z}$-matrix $\C$ is \textbf{upper unitriangular} if its transpose $\C^\top$ is lower unitriangular. Such matrices also have an adjugate, which may be defined as $\adj(\C^\top)^\top$, for which the analog of Proposition \ref{prop: adjprops} holds.}

Given a reduced linear recurrence $\C$, define the \textbf{solution matrix} $\sol(\C)$ of $\C$ by 
\[ \sol(\C) \coloneqq \adj(\C) - \mathsf{P}\adj(\C\mathsf{P}) \]
Examples are given in Figure \ref{fig: adjsol}.
The solution matrix is named for the following property.
\begin{prop}\label{prop: sol}
Each column of $\sol(\C)$ is a solution to $\C\mathsf{x}=\mathsf{0}$.
\end{prop}
\begin{proof}
$\displaystyle \C \sol(\C) = \C[\adj(\C) - \mathsf{P}\adj(\C\mathsf{P})]
= \C\adj(\C) - (\C\mathsf{P})\adj(\C\mathsf{P}) = \mathrm{Id} - \mathrm{Id} = \mathsf{0}$.
%
\end{proof}

\begin{ex}
If $\C$ is the Fibonacci recurrence matrix, the $a$th column of $\sol(\C)$ is given by
\begin{equation}\label{eq: fibsol}
 x_b = \left\{ \begin{array}{cc}
F_{b-a+1} & \text{if $b-a+1\geq 0$} \\
(-1)^{a-b}F_{-b+a-1} & \text{if $b-a+1<0$} 
\end{array}\right\} 
\end{equation}
where $F_i$ is the $i$th Fibonacci number (indexed so that $F_0=0$ and $F_1=1$).
\end{ex}

This proposition can be extended to a complete characterization of solutions to $\C\mathsf{x}=\mathsf{0}$, and thus an answer to our original problem. Given a $\ZZ\times \ZZ$ matrix $\mathsf{A}$, define the \textbf{image} of $\mathsf{A}$ to be the set of all sequences $\mathsf{v}\in \k^\ZZ$ equal to $\mathsf{Aw}$ for some $\mathsf{w\in\k^\ZZ}$.\footnote{This is the usual definition of `image', with the explicit caveat that $\mathsf{Aw}$ may not exist for all $\mathsf{w}$.}

\begin{prethm}{\ref{thm: sol}}
If $\C$ is reduced, the kernel of $\C$ equals the image of $\sol(\C)$.
\end{prethm}

\begin{ex}
Any solution to the Fibonacci linear recurrence can be written as a linear combination of the solutions of the form \eqref{eq: fibsol}.\footnote{In this case, it suffices to only use two adjacent columns of the solution matrix; see Example \ref{ex: fibbasis}.}
\end{ex}


\begin{rem}
If $\C$ is not reduced, the adjugate $\adj(\C \mathsf{P})$ may not be defined, and so $\sol(\C)$ may not be constructed.
\end{rem}

%


\subsection{Balls and juggling}\label{section: balls}

We now connect the shape of a reduced recurrence matrix to the geometry of its kernel. Much of the juggling language and intuition is taken from \cite{KLS13}.

Fix a reduced recurrence matrix $\C$ of shape $S$ for the remainder of the section.
An \textbf{$S$-ball} is a non-singleton equivalence class of integers under the equivalence relation $a\sim S(a)$; that is, an equivalence class which contains more than one integer.

\begin{prethm}{\ref{thm: balls}}
The dimension of $\mathrm{ker}(\C)$ is equal to the number of $S$-balls.
\end{prethm}

\begin{ex}
For the Fibonacci recurrence matrix, the shape $S$ is given by $S(a)=a-2$. There are two $S$-balls (the set of even numbers and the set of odd numbers) and the space of solutions to the Fibonacci recurrence is two dimensional (see Example \ref{ex: fibbasis}).
\end{ex}

\begin{ex}
The recurrence matrix in Example \ref{ex: pivots} has two $S$-balls:
\[ \{ ...,-6,-4,-2\} \text{ and } \{ ...,-5,-3,-1,1,2,3,...\} \]
Since $S(0)=0$, the singleton set $\{0\}$ is also an equivalence class, but it is \emph{not} an $S$-ball.
\end{ex}

\begin{rem}
As a non-increasing injection from $\ZZ\rightarrow \ZZ$, the shape $S$ can be thought of as a \emph{juggling pattern}: instructions for how a juggler catches and throws balls over time. At each moment $a$, the juggler catches the ball they threw at moment $S(a)$ and immediately throws it again...unless $S(a)=a$, in which case they neither catch nor throw a ball.\footnote{At any moment $a$ which is not in the image of $S$, the juggler throws the ball so high it never returns. The ability to throw at escape velocity is a small stretch of the imagination for a juggler who is also immortal.}

Each $S$-ball lists the complete history of when a fixed ball is caught and thrown, and so the number of $S$-balls equals the number of balls the juggler needs for the pattern.
\end{rem}

\begin{rem}\label{rem: balljugs}
The $S$-balls of $\C$ can be visualized as follows. Circle each pivot entry and each entry on the main diagonal, and connect pairs of circles in the same column or in the same row. Excluding any unconnected circles, the connected components of the resulting graph are the $S$-balls. There are four $S$-balls below, each drawn in a different color.
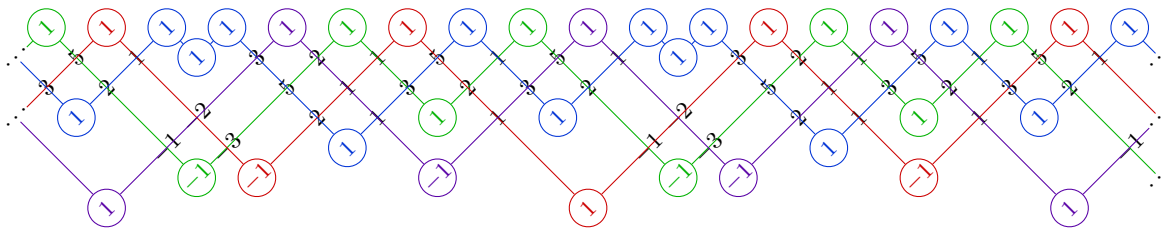
\[\begin{tikzpicture}[baseline=(current bounding box.center),
	ampersand replacement=\&,
	]
	\clip[use as bounding box] (-9.5,-2.7) rectangle (6.3,0.3);
	\matrix[matrix of math nodes,
		matrix anchor = M-2-24.center,
		nodes in empty cells,
		inner sep=0pt,
		gthrow/.style={dark green,draw,circle,inner sep=0mm,minimum size=5mm},
		rthrow/.style={dark red,draw,circle,inner sep=0mm,minimum size=5mm},
		bthrow/.style={dark blue,draw,circle,inner sep=0mm,minimum size=5mm},
		pthrow/.style={dark purple,draw,circle,inner sep=0mm,minimum size=5mm},
		nodes={anchor=center,node font=\scriptsize,rotate=45},
		column sep={0.4cm,between origins},
		row sep={0.4cm,between origins},
	] (M) at (0,0) {
 \&  \&  \&  \&  \&  \&  \&  \&  \&  \&  \&  \&  \&  \&  \&  \&  \&  \&  \&  \&  \&  \&  \&  \&  \&  \&  \&  \&  \&  \&  \&  \&  \&  \&  \&  \&  \&  \&  \\
 \& |[gthrow]| 1 \&  \& |[rthrow]| 1 \&  \& |[bthrow]| 1 \&  \& |[bthrow]| 1 \&  \& |[pthrow]| 1 \&  \& |[gthrow]| 1 \&  \& |[rthrow]| 1 \&  \& |[bthrow]| 1 \&  \& |[gthrow]| 1 \&  \& |[pthrow]| 1 \&  \& |[bthrow]| 1 \&  \& |[bthrow]| 1 \&  \& |[rthrow]| 1 \&  \& |[gthrow]| 1 \&  \& |[pthrow]| 1 \&  \& |[bthrow]| 1 \&  \& |[gthrow]| 1 \&  \& |[rthrow]| 1 \&  \& |[bthrow]| 1 \&  \\
\cdots \&  \& 5 \&  \& 1 \&  \& |[bthrow]| 1 \&  \& 3 \&  \& 2 \&  \& 1 \&  \& 5 \&  \& 1 \&  \& 5 \&  \& 1 \&  \& |[bthrow]| 1 \&  \& 3 \&  \& 2 \&  \& 1 \&  \& 5 \&  \& 1 \&  \& 5 \&  \& 1 \&  \& \cdots \\
 \& 3 \&  \& 2 \&  \&  \&  \&  \&  \& 5 \&  \& 1 \&  \& 3 \&  \& 2 \&  \& 3 \&  \& 2 \&  \&  \&  \&  \&  \& 5 \&  \& 1 \&  \& 3 \&  \& 2 \&  \& 3 \&  \& 2 \&  \&  \&  \\
\cdots \&  \& |[bthrow]| 1 \&  \&  \&  \& -2 \&  \&  \&  \& 2 \&  \& 1 \&  \& |[gthrow]| 1 \&  \& 1 \&  \& |[bthrow]| 1 \&  \&  \&  \& -2 \&  \&  \&  \& 2 \&  \& 1 \&  \& |[gthrow]| 1 \&  \& 1 \&  \& |[bthrow]| 1 \&  \&  \&  \& \cdots \\
 \&  \&  \&  \&  \& -1 \&  \& -3 \&  \&  \&  \& |[bthrow]| 1 \&  \&  \&  \&  \&  \&  \&  \&  \&  \& -1 \&  \& -3 \&  \&  \&  \& |[bthrow]| 1 \&  \&  \&  \&  \&  \&  \&  \&  \&  \& -1 \&  \\
 \&  \&  \&  \&  \&  \& |[gthrow]| -1 \&  \& |[rthrow]| -1 \&  \&  \&  \&  \&  \& |[pthrow]| -1 \&  \&  \&  \&  \&  \&  \&  \& |[gthrow]| -1 \&  \& |[pthrow]| -1 \&  \&  \&  \&  \&  \& |[rthrow]| -1 \&  \&  \&  \&  \&  \&  \&  \& \cdots \\
 \&  \&  \& |[pthrow]| 1 \&  \&  \&  \&  \&  \&  \&  \&  \&  \&  \&  \&  \&  \&  \&  \& |[rthrow]| 1 \&  \&  \&  \&  \&  \&  \&  \&  \&  \&  \&  \&  \&  \&  \&  \& |[pthrow]| 1 \&  \&  \&  \\
 \&  \&  \&  \&  \&  \&  \&  \&  \&  \&  \&  \&  \&  \&  \&  \&  \&  \&  \&  \&  \&  \&  \&  \&  \&  \&  \&  \&  \&  \&  \&  \&  \&  \&  \&  \&  \&  \& \\
	};

	\draw[dark green] (M-3-1) to (M-2-2) to (M-7-7) to (M-2-12) to (M-5-15) to (M-2-18) to (M-7-23) to (M-2-28) to (M-5-31) to (M-2-34) to (M-7-39);
	\draw[dark red] (M-5-1) to (M-2-4) to (M-7-9) to (M-2-14) to (M-8-20) to (M-2-26) to (M-7-31) to (M-2-36) to (M-5-39);
	\draw[dark blue] (M-3-1) to (M-5-3) to (M-2-6) to (M-3-7) to (M-2-8) to (M-6-12) to (M-2-16) to (M-5-19) to (M-2-22) to (M-3-23) to (M-2-24) to (M-6-28) to (M-2-32) to (M-5-35) to (M-2-38) to (M-3-39);
	\draw[dark purple] (M-5-1) to (M-8-4) to (M-2-10) to (M-7-15) to (M-2-20) to (M-7-25) to (M-2-30) to (M-8-36) to (M-5-39);
\end{tikzpicture}\]
Note that reducedness implies that non-zero entries of $\C$ only occur in circles and where two lines cross (See Remark \ref{rem: soljugs} for a contrasting property of the solution matrix $\sol(\C)$).

The connection to juggling is justified by considering the transpose of the above picture.
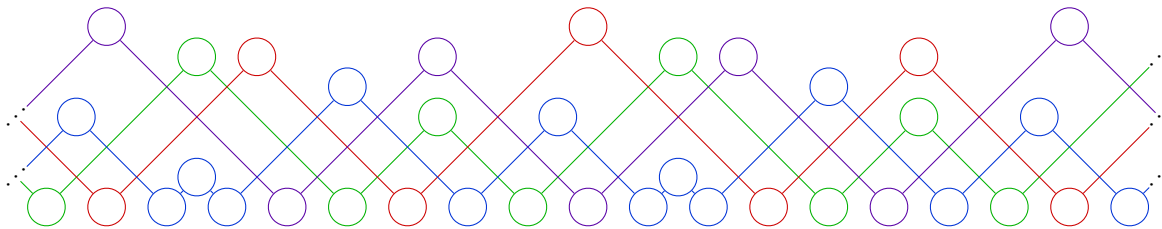
\[\begin{tikzpicture}[baseline=(current bounding box.center),
	ampersand replacement=\&,
	]
	\clip[use as bounding box] (-9.5,2.7) rectangle (6.3,-0.3);
	\matrix[matrix of math nodes,
		matrix anchor = M-2-24.center,
		nodes in empty cells,
		inner sep=0pt,
		gthrow/.style={dark green,draw,circle,inner sep=0mm,minimum size=5mm},
		rthrow/.style={dark red,draw,circle,inner sep=0mm,minimum size=5mm},
		bthrow/.style={dark blue,draw,circle,inner sep=0mm,minimum size=5mm},
		pthrow/.style={dark purple,draw,circle,inner sep=0mm,minimum size=5mm},
		nodes={anchor=center,node font=\scriptsize,rotate=45},
		column sep={0.4cm,between origins},
		row sep={-0.4cm,between origins},
	] (M) at (0,0) {
 \&  \&  \&  \&  \&  \&  \&  \&  \&  \&  \&  \&  \&  \&  \&  \&  \&  \&  \&  \&  \&  \&  \&  \&  \&  \&  \&  \&  \&  \&  \&  \&  \&  \&  \&  \&  \&  \&  \\
 \& |[gthrow]|  \&  \& |[rthrow]|  \&  \& |[bthrow]|  \&  \& |[bthrow]|  \&  \& |[pthrow]|  \&  \& |[gthrow]|  \&  \& |[rthrow]|  \&  \& |[bthrow]|  \&  \& |[gthrow]|  \&  \& |[pthrow]|  \&  \& |[bthrow]|  \&  \& |[bthrow]|  \&  \& |[rthrow]|  \&  \& |[gthrow]|  \&  \& |[pthrow]|  \&  \& |[bthrow]|  \&  \& |[gthrow]|  \&  \& |[rthrow]|  \&  \& |[bthrow]|  \&  \\
\cdots \&  \&  \&  \&  \&  \& |[bthrow]|  \&  \&  \&  \&  \&  \&  \&  \&  \&  \&  \&  \&  \&  \&  \&  \& |[bthrow]|  \&  \&  \&  \&  \&  \&  \&  \&  \&  \&  \&  \&  \&  \&  \&  \& \cdots \\
 \&  \&  \&  \&  \&  \&  \&  \&  \&  \&  \&  \&  \&  \&  \&  \&  \&  \&  \&  \&  \&  \&  \&  \&  \&  \&  \&  \&  \&  \&  \&  \&  \&  \&  \&  \&  \&  \&  \\
\cdots \&  \& |[bthrow]|  \&  \&  \&  \&  \&  \&  \&  \&  \&  \&  \&  \& |[gthrow]|  \&  \&  \&  \& |[bthrow]|  \&  \&  \&  \&  \&  \&  \&  \&  \&  \&  \&  \& |[gthrow]|  \&  \&  \&  \& |[bthrow]|  \&  \&  \&  \& \cdots \\
 \&  \&  \&  \&  \&  \&  \&  \&  \&  \&  \& |[bthrow]|  \&  \&  \&  \&  \&  \&  \&  \&  \&  \&  \&  \&  \&  \&  \&  \& |[bthrow]|  \&  \&  \&  \&  \&  \&  \&  \&  \&  \&  \&  \\
 \&  \&  \&  \&  \&  \& |[gthrow]|  \&  \& |[rthrow]|  \&  \&  \&  \&  \&  \& |[pthrow]|  \&  \&  \&  \&  \&  \&  \&  \& |[gthrow]|  \&  \& |[pthrow]|  \&  \&  \&  \&  \&  \& |[rthrow]|  \&  \&  \&  \&  \&  \&  \&  \& \cdots \\
 \&  \&  \& |[pthrow]|  \&  \&  \&  \&  \&  \&  \&  \&  \&  \&  \&  \&  \&  \&  \&  \& |[rthrow]|  \&  \&  \&  \&  \&  \&  \&  \&  \&  \&  \&  \&  \&  \&  \&  \& |[pthrow]|  \&  \&  \&  \\
 \&  \&  \&  \&  \&  \&  \&  \&  \&  \&  \&  \&  \&  \&  \&  \&  \&  \&  \&  \&  \&  \&  \&  \&  \&  \&  \&  \&  \&  \&  \&  \&  \&  \&  \&  \&  \&  \& \\
	};

	\draw[dark green] (M-3-1) to (M-2-2) to (M-7-7) to (M-2-12) to (M-5-15) to (M-2-18) to (M-7-23) to (M-2-28) to (M-5-31) to (M-2-34) to (M-7-39);
	\draw[dark red] (M-5-1) to (M-2-4) to (M-7-9) to (M-2-14) to (M-8-20) to (M-2-26) to (M-7-31) to (M-2-36) to (M-5-39);
	\draw[dark blue] (M-3-1) to (M-5-3) to (M-2-6) to (M-3-7) to (M-2-8) to (M-6-12) to (M-2-16) to (M-5-19) to (M-2-22) to (M-3-23) to (M-2-24) to (M-6-28) to (M-2-32) to (M-5-35) to (M-2-38) to (M-3-39);
	\draw[dark purple] (M-5-1) to (M-8-4) to (M-2-10) to (M-7-15) to (M-2-20) to (M-7-25) to (M-2-30) to (M-8-36) to (M-5-39);
\end{tikzpicture}\]
This can be viewed as an idealized plot of the height of the `balls' over time.
\end{rem}


Balls and juggling patterns can also be used to parametrize the space of solutions to a linear recurrence.
Given a shape $S$ and an integer $b$, the \textbf{throwing history} at $b$ is the set $T_b$ containing the largest element in each $S$-ball less than or equal to $b$; that is,
\[ T_b := \{a \text{ such that $a\leq b$ but there is no $c\leq b$ with $S(c)=a$} \} \]
\begin{rem}
In juggling terms, the throwing history $T_b$ considers the balls in the air just after moment $b$ and records when each ball was thrown.
\end{rem}


\begin{prethm}{\ref{thm: solextend}}
Restricting to the entries indexed by $T_b$ gives an isomorphism $\mathrm{ker}(\C)\rightarrow \k^{T_b}$.
%
\end{prethm}

\noindent That is, any choice of values for $\{x_b \mid b\in T_b\}$ uniquely extends to a solution to $\C\mathsf{x=0}$. 

We may also use $T_b$ to parametrize the space of solutions.
Let $\sol(\C)_{\ZZ,T_b}$ denote the submatrix consisting of the columns of $\sol(\C)$ indexed by $T_b$.

\begin{prethm}{\ref{thm: solbasis}}
Multiplication by $\sol(\C)_{\ZZ, T_b}$ gives an isomorphism $\k^{T_b}\rightarrow \mathrm{ker}(\C)$.
%
%
\end{prethm}

\noindent Equivalently, every solution to $\C\mathsf{x=0}$ can be written uniquely as a (possibly infinite) linear combination of the columns of $\sol(\C)$ indexed by $T_b$, and all such linear combinations exist (i.e. no infinite non-zero sums).\footnote{In the literature on linear recurrences (e.g. \cite{Bat27}), this is called a \emph{a fundamental system of solutions}.} When $T_b$ is finite, this is merely the definition of a basis.

\begin{coro}
If $\dim(\ker(\C))\!<\!\infty$, the columns of $\sol(\C)$ indexed by $T_b$ are a basis for $\ker(\C)$.
\end{coro}

\begin{ex}\label{ex: fibbasis}
Let us consider the Fibonacci recurrence matrix $\C$ once more. The two $S$-balls are the sets of even and odd numbers, and so $T_a=\{a-1,a\}$ for all $a$. Theorems \ref{thm: solextend} and \ref{thm: solbasis} state that
\begin{enumerate}
	\item any pair of adjacent columns of $\sol(\C)$ form a basis of $\ker(\C)$, and
	\item any pair of values $x_{a-1},x_a\in \k$ may be extended uniquely to an element of $\ker(\C)$.
\end{enumerate}
\end{ex}

\begin{rem}
Definition \ref{defn: schedule} generalizes $T_b$ to a broader class of sets, called \emph{schedules}, for which analogs of Theorems \ref{thm: solextend} and \ref{thm: solbasis} hold.
\end{rem}

\begin{warn}
{\marginnote{\dbend}[.05cm]}
The isomorphisms in Theorems \ref{thm: solextend} and \ref{thm: solbasis} are \bad{not mutually inverse}.
\end{warn}

\section{Generalizations and connections}

We consider a few variations of this problem and applications of these ideas.

\subsection{Affine recurrences} \label{section: affine}

Let us briefly consider the affine case. 
An \textbf{affine recurrence} is a system of equations 
in the sequence of variables $...,x_{-1},x_0,x_1,x_2,...$ which equates each variable to an affine combination of the previous variables (i.e. a degree $1$ polynomial). For example, we could add a constant terms $b_i\in \k$ to each equation in the Fibonacci recurrence:
\begin{equation}
\label{eq: afffib}
x_i = x_{i-1}+x_{i-2}+b_i,\;\;\; \forall i\in \mathbb{Z}
\end{equation}
As before, we may move the variables to the left and factor the coefficients into a matrix $\C$:
\[ \C\mathsf{x} = \mathsf{b}\]
Here, $\C$ is a recurrence matrix, and $\mathsf{b}$ collects the constant terms from each equation.


If $\C$ is reduced, define $\P$ as in Section \ref{section: sol}, and define the \textbf{splitting matrix} $\spl(\C)$ of $\C$ by
\[ \spl(\C)_{a,b}:= \left\{ \begin{array}{cc}
\adj(\C)_{a,b} & \text{if }b\geq 0\\
(\P\adj(\C\P))_{a,b} & \text{if }b<0 
\end{array}\right\} \]
The right half of this matrix is lower unitriangular, and the left half of this matrix is upper triangular, resulting in non-zero entries concentrated into two antipodal wedges.
%
\begin{ex}
The splitting matrix for the Fibonacci recurrence matrix is below.
\[\begin{tikzpicture}[baseline=(current bounding box.center),
	ampersand replacement=\&,
	]
	\node[right, dark blue] at (3,-1.5) {\scriptsize Entries from $\adj(\C)$};
	\node[left, dark red] at (-3,1.2) {\scriptsize Entries from $\P\adj(\C\P)$};
	\clip[use as bounding box] (-2.925,-2.925) rectangle (2.925,2.325);
	\matrix[matrix of math nodes,
		matrix anchor = M-10-10.center,
		nodes in empty cells,
		inner sep=0pt,
		nodes={anchor=center,node font=\scriptsize,rotate=45},
		column sep={0.3cm,between origins},
		row sep={0.3cm,between origins},
	] (M) at (0,0) {
 \&  \&  \&  \&  \&  \&  \&  \&  \&  \&  \&  \&  \&  \&  \&  \&  \&  \&  \\
 \&  \&  \&  \&  \&  \&  \&  \&  \&  \&  \&  \&  \&  \&  \&  \&  \&  \&  \\
\cdots \&  \&  \&  \&  \&  \&  \&  \&  \&  \&  \&  \&  \&  \&  \&  \&  \&  \&  \\
 \& -5 \&  \&  \&  \&  \&  \&  \&  \&  \&  \&  \&  \&  \&  \&  \&  \&  \&  \\
\cdots \&  \& 3 \&  \&  \&  \&  \&  \&  \&  \&  \&  \&  \&  \&  \&  \&  \&  \&  \\
 \& -2 \&  \& -2 \&  \&  \&  \&  \&  \&  \&  \&  \&  \&  \&  \&  \&  \&  \&  \\
\cdots \&  \& 1 \&  \& 1 \&  \&  \&  \&  \&  \&  \&  \&  \&  \&  \&  \&  \&  \&  \\
 \& -1 \&  \& -1 \&  \& -1 \&  \&  \&  \&  \&  \&  \&  \&  \&  \&  \&  \&  \&  \\
 \&  \&  \&  \&  \&  \&  \&  \&  \&  \&  \&  \&  \&  \&  \&  \&  \&  \&  \\
 \&  \&  \&  \&  \&  \&  \&  \&  \& 1 \&  \& 1 \&  \& 1 \&  \& 1 \&  \& 1 \&  \\
 \&  \&  \&  \&  \&  \&  \&  \&  \&  \& 1 \&  \& 1 \&  \& 1 \&  \& 1 \&  \& \cdots \\
 \&  \&  \&  \&  \&  \&  \&  \&  \&  \&  \& 2 \&  \& 2 \&  \& 2 \&  \& 2 \&  \\
 \&  \&  \&  \&  \&  \&  \&  \&  \&  \&  \&  \& 3 \&  \& 3 \&  \& 3 \&  \& \cdots \\
 \&  \&  \&  \&  \&  \&  \&  \&  \&  \&  \&  \&  \& 5 \&  \& 5 \&  \& 5 \&  \\
 \&  \&  \&  \&  \&  \&  \&  \&  \&  \&  \&  \&  \&  \& 8 \&  \& 8 \&  \& \cdots \\
 \&  \&  \&  \&  \&  \&  \&  \&  \&  \&  \&  \&  \&  \&  \& 13 \&  \& 13 \&  \\
 \&  \&  \&  \&  \&  \&  \&  \&  \&  \&  \&  \&  \&  \&  \&  \& 21 \&  \& \cdots \\
 \&  \&  \&  \&  \&  \&  \&  \&  \&  \&  \&  \&  \&  \&  \&  \&  \& 34 \&  \\
 \&  \&  \&  \&  \&  \&  \&  \&  \&  \&  \&  \&  \&  \&  \&  \&  \&  \& \cdots \\
	};
	
	\draw[dark blue,rounded corners,fill=dark blue, opacity=.1] (M-19-19)+(.2,-.5) to ($(M-10-10)+(-.3,0)$) to ($(M-10-10)+(-.2,.2)$) to ($(M-10-19)+(.2,.2)$);
	\draw[dark red,rounded corners,fill=dark red, opacity=.1] (M-3-1)+(-.2,.5) to ($(M-10-8)+(.3,0)$) to ($(M-10-8)+(.2,-.2)$) to ($(M-10-1)+(-.2,-.2)$);

\end{tikzpicture}\]
The lower/upper triangular conditions imply the non-zero entries coming from $\adj(\C)$ and $\P\adj(\C\P)$ must be contained in the blue and red cones, respectively.
\end{ex}



\begin{prop}\label{prop: split}
Let $\C$ be a reduced recurrence matrix. 
\begin{enumerate}
	\item $\spl(\C)$ is horizontally bounded.
	\item $\C\spl(\C)=\mathsf{Id}$; that is, $\spl(\C)$ is a right inverse to $\C$.
	\item For all $\mathsf{b}\in \k^\ZZ$,  the product $\mathsf{x}=\spl(\C)\mathsf{b}$ exists and is a solution to $\C\mathsf{x}=\mathsf{b}$.
\end{enumerate}
\end{prop}
\noindent In particular, $\C\mathsf{x}=\mathsf{b}$ has a solution for all $\mathsf{b}$.

\begin{proof}
(1) If $a\geq 0$, then the $a$th row of $\spl(\C)$ is zero outside the interval $[0,a]$. If $a<0$, then the $a$th row of $\spl(\C)$ is zero outside $[a,0]$.\footnote{This bound can be sharpened, though we won't need a sharp bound. When $a<0$, the $a$th row of $\spl(\C)$ is zero outside the interval $[b,0]$ when $S(b)=a$, and the row is entirely zero if there is no such $b$.} Thus, $\spl(\C)$ is horizontally bounded.

\noindent (2) If $a\geq 0$, then the $a$th column of $\C\spl(\C)$ equals the $a$th column of $\C\adj(\C)=\mathsf{Id}$. If $a<0$, then the $a$th column of $\C\spl(\C)$ equals the $a$th column of $\C(\P(\adj(\C\P)))=\mathsf{Id}$. Therefore, $\C\spl(\C)=\mathsf{Id}$.

\noindent (3) Since $\C$ and $\spl(\C)$ are horizontally bounded, $\C (\spl(\C)\mathsf{b}) = (\C\spl(\C) ) \mathsf{b} = \mathsf{Idb}= \mathsf{b} $.
\end{proof}


Given a solution to $\C\mathsf{x}=\mathsf{b}$, all other solutions are obtained by adding solutions to $\mathsf{C}\mathsf{x}=\mathsf{0}$.

\begin{prop}
If $\C$ is reduced, the solutions to $\C\mathsf{x} = \mathsf{b}$ consist of sequences of the form
\[ \spl(\C) \mathsf{b} + \sol(\C) \mathsf{v}, \]
running over all $\mathsf{v}\in \k^\ZZ$ such that the product $\sol(\C) \mathsf{v}$ exists.
\end{prop}
\begin{proof}
This follows immediately from Proposition \ref{prop: split} and Theorem \ref{thm: sol}.
\end{proof}


\begin{rem}\label{rem: threeinverses}
We have now given three right inverses to a reduced recurrence matrix $\C$, each with an additional property: $\adj(\C)$ is lower unitriangular and a left inverse, $\P\adj(\C\P)$ is upper unitriangular, and $\spl(\C)$ is horizontally bounded. If $\C\neq \mathsf{Id}$, these are all distinct.
\end{rem}

\subsection{Linear recurrences indexed by $\mathbb{N}$}

Variations of linear recurrences have been studied for centuries. Most often, one considers a system with variables indexed by $\mathbb{N}$ (rather than $\ZZ$) and relations defining each variable except at finitely many \emph{initial variables}.  For example, the \text{one-sided Fibonacci recurrence} has initial variables $x_0$ and $x_1$ and equations
	\begin{equation*}
	x_i = x_{i-1}+x_{i-2},\;\;\; \forall i\geq 2
	\end{equation*}
	
The study of $\mathbb{N}$-indexed linear recurrences differs fundamentally from $\ZZ$-indexed linear recurrences.
Solutions to an $\mathbb{N}$-indexed system are determined by the values of the initial variables, which trivializes the kinds of questions we have considered (e.g. existence and parametrization of solutions).
Rather, most work in the $\mathbb{N}$-indexed context has focused on finding simple formulas for the terms in a solution. 
We review a few of these approaches.
\begin{itemize}
	\item When the equations in a linear recurrence are the same up to a shift of indices (a \emph{constant} linear recurrence), shifting the indices of a solution
	\[ x_0,x_1,x_2,... \longmapsto x_1,x_2,x_3,... \]
	defines a linear transformation from the space of solutions to itself. Standard tools from linear algebra (e.g. the characteristic polynomial) can then construct a basis of eigenvectors or generalized eigenvectors for the space of solutions. 
	Since the eigenvectors are geometric sequences, an eigenbasis expresses any solution as a linear combination of geometric sequences. 
	This is covered in textbooks like \cite{GK94}.
%
	\item A sequence $x_0, x_1, x_2,...$ can be converted into a formal series in several ways, such as
	\[ F_\mathsf{x}(t) := x_0 + x_1 t + x_2 t^2 + x_3t^3 + \cdots  \]
	Some linear recurrences (such as constant ones) translate into functional equations involving these \emph{generating functions}. 
	Clever manipulation of these equations can then yield simple formulas for solutions. This is covered in textbooks like \cite{Wil06}.
	\item The asymptotics of solutions, that is, the behavior of $x_i$ for sufficiently large $i$, can be studied analytically. Poincare \cite{Poi85} and others\footnote{A curiosity: \cite{Car11} is the dissertation of Robert Carmichael, of \emph{Carmichael numbers} in number theory.} \cite{Car11,Bir11} construct integrals which coincide with the generating function $F_\mathsf{x}(t)$ in an `infinitesmal neighborhood of infinity'. See \cite[Section I.6]{Bat27} for further details.
\end{itemize}

The techniques of the current work can be adapted to this setting. 
%
We first add equations fixing the initial values and rewrite the system as a matrix equation $\C\mathsf{x}=\mathsf{b}$. For example, the (one-sided) Fibonacci recurrence with initial values $x_0=a$ and $x_1=b$ is rewritten as
\[ \begin{bmatrix}
		1 & 0 & 0 & 0  & \\
		 0 & 1 & 0 & 0 & \\
		-1 & -1 & 1 & 0 & \\
		0 & -1 & -1 & 1 & \\
		 & & & & \dddots \\
\end{bmatrix}
\begin{bmatrix}
x_0 \\ x_1 \\ x_2 \\ x_3 \\ \vdots 
\end{bmatrix}
=\begin{bmatrix}
a \\ b \\ 0 \\ 0 \\ \vdots 
\end{bmatrix}\]
%
%
The recurrence matrix $\mathsf{C}$ is $\mathbb{N}\times\mathbb{N}$, lower unitriangular, and horizontally bounded. The adjugate $\adj(\C)$ is defined as before, and the identity $\adj(\C)\C=\C\adj(\C)=\mathsf{Id}$ still holds.

However, there is a crucial difference. In the $\mathbb{N}\times\mathbb{N}$ case, the adjugate matrix $\adj(\C)$ is horizontally bounded, and so $\adj(\C) (\C\mathsf{x}) = (\adj(\C) \C) \mathsf{x}$ for all $\ZZ$-vectors $\mathsf{x}$. If $\C\mathsf{x=b}$, then
%
\[ \adj(\C) \mathsf{b} = \adj(\C) (\C\mathsf{x})= (\adj(\C) \C) \mathsf{x} =\mathsf{x}\]
Consequently, the unique solution to $\C\mathsf{x=b}$ can be computed as a linear combination of the columns of $\adj(\C)$ indexed by the initial variables. We can restate this as follows.


\begin{prop}\label{prop: Nsol}
If $x_i$ is an initial variable in an $\mathbb{N}$-indexed linear recurrence with recurrence matrix $\C$, then the $i$th column of $\adj(\C)$ is the solution for which $x_i=1$ and all other initial variables are $0$. Columns of this form are a basis for the space of all solutions.
\end{prop}
%

\begin{rem}
While Proposition \ref{prop: Nsol} gives a basis of solutions, it is unclear how useful this is in general. Computationally, 
the entries of $\adj(\C)$ are determinant of submatrices of $\C$, which are (naively) no simpler than recursively computing $x_0,x_1,....,x_j$ directly.
%
\end{rem}

\subsection{Friezes}\label{section: frieze}

The author's original motivation for the work in this paper is a connection and forthcoming application to the following curious objects.
%
%
A \textbf{tame $\mathrm{SL}(k)$-frieze} consists of finitely many rows of integers (offset in a diamond pattern) such that:
\begin{itemize}
	\item the top and bottom rows consist entirely of $1$s,
	\item every $k\times k$ diamond has determinant $1$, and 
	\item (\emph{Tameness}) every $(k+1)\times (k+1)$ diamond has determinant $0$ .
\end{itemize}

\begin{ex}
An example of an $\mathrm{SL}(2)$-frieze
is given below (every $\mathrm{SL}(2)$-frieze is tame).
%
%
%
\[
\begin{tikzpicture}[baseline=(current bounding box.south),
	ampersand replacement=\&,
	]
	\matrix[matrix of math nodes,
		matrix anchor = M-1-8.center,
		origin/.style={},
		throw/.style={},
		pivot/.style={draw,circle,inner sep=0.25mm,minimum size=2mm},		
		nodes in empty cells,
		inner sep=0pt,
		nodes={anchor=center,node font=\scriptsize},
		column sep={.35cm,between origins},
		row sep={.35cm,between origins},
	] (M) at (0,0) {
	 \& 1 \& \& 1 \& \& 1 \& \& 1 \& \& 1 \& \& 1 \& \& 1 \& \& 1 \&  \& 1 \& \& 1 \& \& 1 \& \& 1 \& \& 1 \& \& 1 \& \& 1 \& \& 1 \& \& 1 \& \& \cdots \\
	\cdots \& \& 3 \& \& 2 \& \& 2 \& \& 1 \& \& 4 \& \& 3 \& \& 1 \& \& 2 \& \& 3 \& \& 2 \& \& 2 \& \& 1 \& \& 4 \& \& 3 \& \& 1 \& \& 2 \& \& 3 \&  \\
	 \& 5 \& \& 5 \& \& 3 \& \& 1 \& \& 3 \& \& 11 \& \& 2 \& \& 1 \& \& 5 \& \& 5 \& \& 3 \& \& 1 \& \& 3 \& \& 11 \& \& 2 \& \& 1 \& \& 5 \& \& \cdots \\
	\cdots \& \& 8 \& \& 7 \& \& 1 \& \& 2 \& \& 8 \& \& 7 \& \& 1 \& \& 2 \& \& 8 \& \& 7 \& \& 1 \& \& 2 \& \& 8 \& \& 7 \& \& 1 \& \& 2 \& \& 8 \&  \\
	 \& 3 \& \& 11 \& \& 2 \& \& 1 \& \& 5 \& \& 5 \& \& 3\& \& 1 \& \& 3 \& \& 11 \& \& 2 \& \& 1 \& \& 5 \& \& 5 \& \& 3\& \& 1 \& \& 3 \& \& \cdots \\
	\cdots \& \& 4 \& \& 3 \& \& 1 \& \& 2 \& \& 3 \& \& 2 \& \& 2 \& \& 1 \& \& 4 \& \& 3 \& \& 1 \& \& 2 \& \& 3 \& \& 2 \& \& 2 \& \& 1 \& \& 4 \&  \\
	 \& 1 \& \& 1 \& \& 1 \& \& 1 \& \& 1 \& \& 1 \& \& 1 \& \& 1 \& \& 1 \& \& 1 \& \& 1 \& \& 1 \& \& 1 \& \& 1 \& \& 1 \& \& 1 \& \& 1 \& \& \cdots \\
	};
\end{tikzpicture}\qedhere \]
\end{ex}

The study of friezes was initiated in \cite{Cox71,CC73a,CC73b} for $k=2$, and generalized to arbitrary $k$ in \cite{CR72,BR10}.
Friezes enjoy many remarkable properties; for example, the rows of a tame $SL(k)$-frieze must be periodic. An excellent overview is given in \cite{MG15}.

A frieze may be converted into a recurrence matrix, by rotating $45^\circ$ clockwise and using the top row as the main diagonal.\footnote{\cite{MGOST14} and others also use an alternating sign when translating a frieze into a linear recurrence.}
Remarkably, the solutions have a periodicity condition.
%
\begin{thm}\label{thm: frieze}\cite{MGOST14}
If $\C$ is the recurrence matrix associated to a tame $SL(k)$-frieze, then every solution to $\C\mathsf{x=0}$ is \emph{superperiodic}: $x_{i+n}=(-1)^sx_i$ for some $n$ and $s$ and all $i$.
\end{thm}

\noindent In fact, \cite{MGOST14} proves a stronger result. 
For each frieze $\C$, they construct a \emph{dual} frieze whose diagonals encode distinguished solutions to $\C\mathsf{x=0}$. 

In a sequel \cite{DM19} to the current work, we will extend Theorem \ref{thm: frieze} to an equivalence. Specifically, if $\C$ is a reduced recurrence matrix of shape $S$, then the following are equivalent.
\begin{itemize}
	\item $\C$ satisfies a family of determinantal identities generalizing the tame frieze conditions.
	\item Every solution to $\C\mathsf{x=0}$ is \emph{$n$-quasiperiodic}; that is, $x_{i+n}=\lambda x_i$ for some $\lambda$ and all $i$.
	\item A \emph{dual} $\C^\dagger$ has shape $S^\dagger$, where $S^\dagger(i):=S^{-1}(i)+n$.
\end{itemize}
\begin{rem}
The space of such linear recurrences (of shape $S$) is  the \emph{cluster $\mathcal{X}$-variety} dual to the \emph{positroid variety} corresponding to $S$; this will be explained in \cite{DM19}.
\end{rem}


%
%
%

%

\newpage



{\centering\it
Sections 5\--8 prove the promised results. 
\\}

\section{Kernel containment and factorization}


In this section, we prove a useful equivalence between containments of kernels and factorizations in the semigroup of recurrence matrices. 

\begin{nota}
Let $\mathsf{k}^\ZZ_b\subset \mathsf{k}^\ZZ$ denote the subspace of \textbf{bounded sequences} (i.e. non-zero in finitely many terms). If $\mathsf{v}\in \mathsf{k}^\ZZ_b$ and $\mathsf{w}\in \mathsf{k}^\ZZ$, then the dot product $\mathsf{v}\cdot \mathsf{w}$ is well-defined.
\end{nota}



\begin{lemma}\label{lemma: lincomb}
Let $\C$ be a recurrence matrix and let $\mathsf{v}\in \mathsf{k}^\ZZ_b$. If $\mathsf{v}\cdot \mathsf{w}=0$ for all $\mathsf{w}\in \mathrm{ker}(\C)$, then $\mathsf{v}$ is in the span of the rows of $\C$.
\end{lemma}
\begin{proof}
Let $V\subset \mathsf{k}^\ZZ_b$ denote the span of the rows of $\C$, and assume for contradiction that $\mathsf{v}\not\in V$. We may therefore choose a linear map $f:\mathsf{k}^\ZZ_b\rightarrow \mathsf{k}$ such that $f(V)=0$ and $f(\mathsf{v})=1$.\footnote{The existence of such a map may depend on the Axiom of Choice, which we therefore assume.}

Let $\mathsf{e}_a\in \mathsf{k}^\ZZ_b$ denote the standard basis vector which is $1$ in the $a$th term and $0$ everywhere else, and set 
$\mathsf{w}:= ( f(\mathsf{e}_a)) _{a\in \ZZ}\in \mathsf{k}^\ZZ$. By linearity, $f(\mathsf{u})=\mathsf{u}\cdot \mathsf{w}$ for all $\mathsf{u}\in\mathsf{k}^\ZZ_b$. Since $f$ kills each row of $\C$, $\C\mathsf{w}=0$ and so $\mathsf{w}\in \mathrm{ker}(\C)$. However, $\mathsf{v}\cdot \mathsf{w}=1$, contradicting the hypothesis.
\end{proof}

\begin{lemma}\label{lemma: containment}
Let $\C$ and $\C'$ be recurrence matrices.
Then the following are equivalent.
\begin{enumerate}
	\item $\ker(\C)\subseteq \ker(\C')$.
	\item $\C'=\mathsf{D}\C$ for some horizontally bounded matrix $\mathsf{D}$.
	\item $\C'=\mathsf{D}\C$ for some recurrence matrix $\mathsf{D}$.
	\item $\C'\adj(\C)$ is a recurrence matrix.
\end{enumerate}
Furthermore, the matrix $\D$ in (2) and (3) must equal $\C'\adj(\C)$ and is therefore unique.
\end{lemma}

%

\begin{proof}
($1\Rightarrow 2$)
If $\mathrm{ker}(\C)\subseteq \mathrm{ker}(\C')$, then each row of $\C'$ kills $\mathrm{ker}(\C)$. By Lemma \ref{lemma: lincomb}, the $a$th row of $\C'$ is equal to $\mathsf{D}_a\C$ for some bounded sequence $\mathsf{D}_a\in \mathsf{k}^\ZZ_b$. The vectors $\mathsf{D}_a$ may be combined into the rows of a matrix $\mathsf{D}$ which is horizontally bounded and satisfies $\mathsf{D}\C=\C'$.

($2\Rightarrow 3+4+$Uniqueness)
Assume that  $\C'=\mathsf{D}\C$ for a horizontally bounded $\mathsf{D}$. Then
\begin{equation*}
\C'\adj(\C) = (\D\C)\adj(\C) \stackrel{*}{=} \D(\C\adj(\C)) = \D 
\end{equation*}
Equality ($*$) holds because $\D$ and $\C$ are horizontally bounded. Since $\C'\adj(\C)$ is lower unitriangular and $\D$ is horizontally bounded, they are the same recurrence matrix.


($3\Rightarrow 1$)
Assume $\C'=\D\C$ for some recurrence matrix $\D$. If $\mathsf{v}\in \ker(\C)$, then
\[ \C'\mathsf{v} = (\D\C)\mathsf{v} \stackrel{*}{=} \D(\C\mathsf{v}) = \D\mathsf{0} = \mathsf{0} \]
Equality ($*$) holds because $\D$ and $\C$ are horizontally bounded. Therefore, $\ker(\C)\subseteq \ker(\C')$.

($4\Rightarrow 3$)
Setting $\D\coloneqq \C'\adj(\C)$, we check that
\[ \D\C = (\C'\adj(\C) ) \C \stackrel{*}{=} \C'(\adj(\C)\C) = \C'\]
Equality ($*$) holds because $\C'$, $\adj(\C)$, and $\C$ are lower unitriangular.
%
\end{proof}

%
%


%


\begin{thm}\label{thm: triv-inv}
A recurrence matrix $\C$ is trivial if and only if $\adj(\C)$ is a recurrence matrix.
\end{thm}
\begin{proof}
The recurrence matrix $\C$ is trivial when $\ker(\C)\subseteq \ker(\mathsf{Id})$. Applying Lemma \ref{lemma: containment} with $\C'=\mathsf{Id}$, this holds if and only if $\adj(\C)$ is a recurrence matrix.
%
%
\end{proof}

\begin{thm}\label{thm: triv-equiv}
Two recurrence matrices $\C$ and $\C'$ are equivalent if and only if $\C'=\mathsf{D}\C$ for a trivial recurrence matrix $\mathsf{D}$.
\end{thm}
\begin{proof}
Two recurrence matrices $\C$ and $\C'$ are equivalent if and only if $\ker(\C)=\ker(\C')$. By Lemma \ref{lemma: containment}, this holds if and only if $\D=\C'\adj(\C)$ and $\adj(\D)=\C\adj(\C')$ are recurrence matrices. By Theorem \ref{thm: triv-inv}, this is equivalent to $\D$ being a trivial linear recurrence.
%
\end{proof}



Lemma \ref{lemma: containment} also allows us to make a connection between kernel containment and shapes.

\begin{lemma}\label{lemma: containshape}
Let $\C$ and $\C'$ be recurrence matrices with shape $S$ and $S'$, respectively. 
If $\ker(\C)\subseteq  \ker(\C')$ and $S$ is injective (e.g. if $\C$ is reduced), then $S(a)\geq S'(a)$ for all $a$.
%
\end{lemma}

\begin{proof}
By Lemma \ref{lemma: containment}, $\C'=\D\C$ for some recurrence matrix $\D$. 
%
%
Fix $a\in \ZZ$, and consider $B:=\{ b \in \ZZ \mid \D_{a,b} \neq 0\}$. This set is bounded and contains $a$. Let $b_{0}$ be the element of $B$ on which $S$ is minimal; this is unique because $S$ is injective.
\[ (\D\C)_{a,S(b_{0})} = \sum_{b\in \ZZ} \D_{a,b} \C_{b,S(b_{0})} = \D_{a,b_{0}}\C_{b_{0},S(b_{0})} \neq0 \]
Therefore, $S'(a)\leq S(b_0)$.\footnote{In fact, this is equality, but we won't need this stronger statement.} Since $a\in B$, $S(b_0)\leq S(a)$, and so $S'(a)\leq S(a)$.
\end{proof}

%
%

\begin{prop}\label{prop: reducedequal}
Reduced recurrence matrices that are equivalent must be equal.
\end{prop}
\begin{proof}
Let $\C$ and $\C'$ be reduced and equivalent. Lemma \ref{lemma: containshape} implies that $\C$ and $\C'$ have the same shape; call it $S$. By Lemma \ref{lemma: containment}, $\C'=\D\C$ for some recurrence matrix $\D$.

Let $T$ denote the shape of $\D$, so that $\D_{a,b}=0$ whenever $b<T(a)$.
Since $\C$ is reduced of shape $S$, $\C_{b,S(T(a))}=0$ whenever $b>T(a)$. 
Therefore,
\[ \C'_{a,S(T(a))} = \sum_b \D_{a,b} \C_{b,S(T(a))} = \D_{a,T(a)}\C_{T(a),S(T(a))} \neq 0 \]
Since $\C'$ is also reduced of shape $S$, this is only possible if $a= T(a)$. 
Since this holds for all $a$, the only non-zero entries of $\D$ are on the main diagonal. Thus, $\D=\mathsf{Id}$ and $\C'=\C$.
\end{proof}


\section{Gauss-Zordan Elimination}
\label{section: GZ}


Because we are working with $\ZZ\times\ZZ$-matrices, we must consider infinite sequences of row reductions that may be chosen in an arbitrary order. We furthermore consider \emph{generalized row reductions}: limits of such row reductions (in an appropriate topology).

\subsection{Row reduction}

Given a recurrence matrix $\C$ of shape $S$, a \textbf{row reduction} of $\C$ is a matrix $\C'$ obtained by adding $\C_{a,S(b)} / \C_{b,S(b)}$ times the $b$th row to the $a$th row, for some $b> a$ with $\C_{a,S(b)}\neq0$. By design, the resulting matrix $\C'$ has a zero in the $(a,S(b))$ entry. 

\begin{prop}
A recurrence matrix is reduced if and only if it has no row reductions.
\end{prop}

A row reduction of $\C$ can be reformulated as a factorization $\C=\D\C'$ such that $\D$ differs from the identity matrix in a single entry $\D_{a,b}$, and such that $\C_{a,S(b)}\neq 0$ and $\C'_{a,S(b)}=0$. This perspective leads to the following generalization.

A \textbf{generalized row reduction of $\C$} is a recurrence matrix $\C'$ such that $\C=\D\C'$ for a trivial recurrence matrix $\D$ with the property that, for each $a$ such that $\{b <a \mid \D_{a,b}\neq 0\}$ is non-empty, we have
\[ \C_{a,b_a} \neq 0\text{ and } \C'_{a,b_a} = 0 \]
where $b_a:= \min\{S(b) \mid b< a \text{ s.t. }\D_{a,b}\neq 0\}$. 
We write $\C\succeq \C'$ to denote that $\C'$ is a generalized row reduction of $\C$.

\begin{rem}\label{rem: leftmost}
The index $b_a$ may be defined as the leftmost entry of the $a$th row that multiplication by $\D$ is `big enough' to change, and so $(\D\C)_{a,b}=\C_{a,b}$ whenever $b<b_a$. Thus, if $\C\succeq \C'$, then $\C'$ must vanish in the leftmost entry in which the $a$th rows of $\C$ and $\C'$ differ.
%
%
\end{rem}


\begin{prop}
The relation $\succeq$ defines a partial order on the set of recurrence matrices.
\end{prop}

\noindent As a consequence, an iterated sequence of row reductions is a generalized row reduction.

\begin{proof}
(Antisymmetry) Assume $\C\succeq\C'$ and $\C\preceq \C'$. By Remark \ref{rem: leftmost}, both $\C$ and $\C'$ vanish in the leftmost entry in which the $a$th rows of $\C$ and $\C'$ differ. However, two entries cannot both vanish and be different, so the $a$th rows of $\C$ and $\C'$ coincide for all $a$. Thus, $\C=\C'$.

(Transitivity) Let $\C\preceq \D\C \preceq \D'\D\C$, and let $S$ and $S'$ denote the shapes of $\C$ and $\D\C$, respectively. Fix some $a$. If $\{b< a\mid \D_{a,b}\neq 0\}= \varnothing$ or $\{b< a\mid \D'_{a,b}\neq 0\}= \varnothing$, the generalized row reduction condition is easy to check. Assume neither set is empty and let
\begin{align*}
b_0 &:= \min\{S(b) \mid b<a\text{ s.t. }\D_{a,b}\neq0\} \\
b_0' &:= \min\{S'(b) \mid b<a\text{ s.t. }\D'_{a,b}\neq0\} 
\end{align*}
By the definition of generalized row reductions,
\[\C_{a,b_0}=0,\;\;\; (\D\C)_{a,b_0}\neq0,\;\;\; (\D\C)_{a,b_0'}=0 ,\;\;\; (\D'\D\C)_{a,b_0'}\neq 0\]
This ensures that $\C_{a,\min(b_0,b_0')}=0$ and $(\D'\D\C)_{a,\min(b_0,b_0')}\neq0$. Since these entries differ,
\[ \min\{S(b) \mid b<a\text{ s.t. }(\D'\D)_{a,b}\neq0\} \leq \min(b_0,b_0') \]
To show this is equality, consider some $b<a$ such that $(\D'\D)_{a,b}\neq0$. We split into cases.
\begin{itemize}
	\item  Assume $\D'_{a,c}\D_{c,b}\neq0$ for some $c<a$. Since $\D_{c,b}\neq0$ and $\C\preceq \D\C$, $S'(c)\leq S(b)$. Since $\D'_{a,c}\neq 0$ and $\D\C\preceq \D'\D\C$, $S'(c)\geq b_0'$. Therefore, $S(b)\geq b'_0$.
	\item Otherwise, $\D'_{a,c}\D_{c,b}=0$ for all $c<a$, and so $(\D'\D)_{a,b}= \D'_{a,a}\D_{a,b}=\D_{a,b}$. Since $\D_{a,b}\neq0$ and $\C\preceq \D\C$, we know that $S(b)\geq b_0$.
\end{itemize}
Therefore, $\C_{a,\min(b_0,b_0')}=0$ and $(\D'\D\C)_{a,\min(b_0,b_0')}\neq0$ and
\[ \min\{S(b) \mid b<a\text{ s.t. }(\D'\D)_{a,b}\neq0\} = \min(b_0,b_0') \]
Since this holds for all $a$, $\C\preceq \D'\D\C$.
\end{proof}

\subsection{Limits}\label{section: limits}

To define limits of generalized row reductions, we endow the set of recurrence matrices with the topology of \textbf{row-wise stabilization}: a sequence of recurrence matrices converges if each row is constant after finitely many steps.

We next show that sequences of generalized row reductions must stabilize row-wise to another generalized row reduction, via the following more general result.

\begin{lemma}\label{lemma: Zorn}
Let $\mathcal{C}$ be a set of recurrence matrices in which every pair is comparable in the generalized row reduction partial order.\footnote{Sometimes called a `chain' in the literature on partially ordered sets.}
Then the closure of $\mathcal{C}$ in the space of recurrence matrices contains a lower bound of $\mathcal{C}$.
\end{lemma}
\noindent Equivalently, there is a descending sequence of recurrence matrices in $\mathcal{C}$ (i.e. generalized row reductions of the initial matrix in the sequence) which converges (i.e. stabilizes row-wise) to a lower bound of $\mathcal{C}$ (i.e. a generalized row reduction of every matrix in $\mathcal{C}$).

\begin{proof}
Given a recurrence matrix $\C$ and an integer $a$, define 
\[ n_a(\C) := \sum_{b \text{ s.t. } \C_{(a,b)}\neq0} (a-b)^2 \]
If $\C\preceq \C'$, then $n_a(\C)\leq n_a(\C')$ and equality implies the $a$th rows of $\C$ and $\C'$ coincide.

For each $a$, let $\mathcal{C}_a:= \{ \C\in \mathcal{C} \mid \forall \C'\in\mathcal{C},\;n_a (\C) \leq n_a(\C')\}$; that is, $\mathcal{C}_a$ is the set of matrices in $\mathcal{C}$ which attain the minimum value of $n_a$. This set is non-empty and the $a$th row of each matrix in $\mathcal{C}_a$ is the same, since $n_a$ has the same value and the matrices are comparable.

Consider $a,a'\in\ZZ$ and assume, for contradiction, that there exist $\C\in \mathcal{C}_a\smallsetminus \mathcal{C}_{a'}$ and $\C' \in\mathcal{C}_{a'}\smallsetminus \mathcal{C}_a$. If $\C'\preceq \C$, then $n_a(\C')\leq n_a(\C)$. By the minimality of $n_a(\C)$, this is an equality and so $\C'\in \mathcal{C}_a$; a contradiction. By a symmetric argument, $\C\preceq \C'$ forces a contradiction. Therefore, $\mathcal{C}_a\cap \mathcal{C}_{a'}$ is either equal to $\mathcal{C}_a$ or equal to $\mathcal{C}_{a'}$.

Applying this repeatedly, for any $i\in \NN$, there is some $a_i\in [-i,i]$ such that
\[ \bigcap_{a\in [-i,i]} \mathcal{C}_a =\mathcal{C}_{a_i} \neq \varnothing \]
Choose a matrix $\C^i$ in $\mathcal{C}_{a_i}$ for each $i$. The $a$th rows in the sequence $\C^1,\C^2,\C^3,...$, stabilize after the $a$th term, and so this sequence converges to the recurrence matrix $\C$ whose $a$th row coincides with the $a$th row in each matrix in $\mathcal{C}_a$. 
%
%
%

Let $S$ be the shape of $\C^1$. Define a sequence $\D^1,\D^2,\D^3,...$ of trivial recurrence matrices by $\C^1 = \D^n\C^n$ for all $n$. Since $\C^1\succeq \C^n$, if $\D^n_{a,b}\neq0$, then $S(a)\leq S(b)\leq b$; that is, the $a$th row $\D$ can be non-zero only on the interval $[S(a),a]$.

When $n>|S(a)|$, the $a$th row of the product $\D^n\C^n$ only depends on rows in $\C^n$ that coincide with rows in $\overline{\C}$. Therefore, the $a$th row of $\D^n\overline{\C}$ is equal to $\C^1$. Therefore, the sequence $\D^1,\D^2,\D^3,...$ stabilizes row-wise to a matrix $\overline{\D}$ such that $\overline{\D}\overline{\C}=\C^1$. As $\overline{\D}_{a,b}=\D^n_{a,b}$ for large enough $n$, 
this shows that $\C\preceq \C^1$. Since the sequence $\C^\bullet$ could have started at any matrix in $\mathcal{C}$, this shows $\overline{\C}$ is a lower bound for $\mathcal{C}$.
\end{proof}


%


\begin{thm}\label{thm: reduced1}
Every recurrence matrix is equivalent to a unique reduced recurrence matrix.
\end{thm}
\begin{proof}
Let $\mathcal{C}$ be an equivalence class of recurrence matrices, with the row reduction partial order. Every non-empty chain in $\mathcal{C}$ has a lower bound (by Lemma \ref{lemma: Zorn}). By Zorn's Lemma, $\mathcal{C}$ contains a minimal element $\overline{\C}$.

If $\overline{\C}$ was not reduced, then there would be a row operation which would strictly decrease it in the reduction partial order; contradicting minimality. Therefore, $\overline{\C}$ is reduced. By Proposition \ref{prop: reducedequal}, this reduced recurrence matrix is unique.
\end{proof}

This provides a transfinite, non-deterministic analog of Gauss-Jordan elimination, which we humorously dub \emph{Gauss-Zordan elimination} (both for `{Zorn}' and the integers $\ZZ$). Given a recurrence matrix $\C$, an arbitrary sequence of row reductions will stabilize row-wise to a matrix equivalent to $\C$. While this limit may not be reduced, further arbitrary row reductions generate another convergent sequence. Zorn's Lemma guarantees that some transfinite iteration of this process will eventually converge to the reduced representative of $\C$.

\section{Constructing recurrences from spaces of solutions}
\label{section: rank}

\def\R{\mathsf{R}}

In this section, we consider the inverse problem to the motivating problem of this note: Given a subspace $V\subset \k^\ZZ$, when and how can we construct a linear recurrence $\C\mathsf{x}=\mathsf{0}$ whose solutions are $V$? 
We give a characterization of when this is possible in Theorem \ref{thm: reduced2}.

\begin{nota}
For any $I\subset \mathbb{Z}$, let $\pi_{I}:\k^\mathbb{Z}\rightarrow \k^{I}$ restrict a sequence to the indices in $I$, and let $\iota_I:\k^I\rightarrow \k^\ZZ$ extend a sequence by $0$.
Given a subspace $V\subset \k^\mathbb{Z}$, let $V_{I}\coloneqq \pi_{I}(V)\subset \k^{I}$.
\end{nota}

 \subsection{Closed subspaces}
 
%
 
We endow the space of sequences $\k^\ZZ$ with the product topology for the discrete topology on $\k$. 
The two relevant facts for us are the following.
\begin{itemize}
	\item A sequence in $\k^\ZZ$ converges if the restriction to $[a,b]$ stabilizes for all finite intervals $[a,b]$.
%
	\item The closure of $V\subset \k^\ZZ$ in the product topology on $\k^\ZZ$ is the set of sequences $\mathsf{v}$ such that $\mathsf{v}_{[a,b]} \in V_{[a,b]}$ for all finite intervals $[a,b]$.
\end{itemize}

\begin{rem}
As an example of a non-closed subspace, consider the subspace $V\subset \k^n$ of periodic sequences. On any finite interval, every sequence coincides with the restriction of some periodic sequence (for sufficiently large period), and so the closure of $V$ is all of $\k^n$.
\end{rem}

%
%

\begin{prop}
If $\C$ is a recurrence matrix, then $\mathrm{ker}(\C)$ is closed.
\end{prop}
\begin{proof}
Let $\mathsf{v}$ be in the closure of $\mathrm{ker}(\C)$. For any $a\in \ZZ$, there is a $\mathsf{w}\in \mathrm{ker}(\C)$ such that $\mathsf{w}_{[S(a),a]} = \mathsf{v}_{[S(a),a]}$. Since the $a$th row of $\C$ is zero outside columns $[S(a),a]$,
\[ (\C\mathsf{v})_a = \C_{a,[S(a),a]} \mathsf{v}_{[S(a),a]} =  \C_{a,[S(a),a]} \mathsf{w}_{[S(a),a]} =(\C\mathsf{w})_a =0\]
Since this holds for all $a$, $\C\mathsf{v}=0$ and $\mathsf{v}\in \mathrm{ker}(\C)$.
\end{proof}

\begin{rem}
Intuitively, the closed subspaces are those determined by their restrictions to all finite intervals. This suggests our strategy for the rest of the section: given closed $V$, try to construct $\C$ with $\mathrm{ker}(\C)=V$ by considering the restrictions of $V$ to all finite intervals.
%
%
\end{rem}


\subsection{Rank matrices}

The \textbf{rank matrix} of a subspace $V\subset \k^\ZZ$ is the $\mathbb{Z}\times \mathbb{Z}$-matrix with\footnote{The entries below the diagonal are unimportant; we set them to $b-a+1$ to avoid special cases later.}
\[ \R_{a,b}:= \left\{\begin{array}{cc}
\dim_\k(V_{[a,b]}) & \text{if }a\leq b \\
b-a+1 & \text{otherwise}
\end{array} \right\}\]

\begin{ex}
Let $V$ be the space of sequences such that (a) the $-1$st term is $0$, (b) the $-2$nd and $0$th term are equal, and (c) the $0$th, $1$st, and $2$nd terms sum to $0$. The rank matrix is
\[\begin{tikzpicture}[baseline=(current bounding box.center),
	ampersand replacement=\&,
	]
	\matrix[matrix of math nodes,
		matrix anchor = M-1-8.center,
		origin/.style={},
		throw/.style={},
		defect/.style={dark red,draw,circle,inner sep=0.25mm,minimum size=2mm},
		pivot/.style={draw,circle,inner sep=0.25mm,minimum size=2mm},		
		nodes in empty cells,
		inner sep=0pt,
		nodes={anchor=center,
		rotate=45},
		column sep={.5cm,between origins},
		row sep={.5cm,between origins},
	] (M) at (0,0) {
	\cdots \& \& \cdots \& \& \cdots \& \& \cdots \& \& \cdots \& \& \cdots \& \& \cdots \& \& \cdots \&  \\
	 \& 5 \& \& 5 \& \& 4 \& \& 4 \& \& 4 \& \& 5 \& \& 6 \& \& \cdots \\
	\cdots \& \& 4 \& \& 4 \& \& 3 \& \& 3 \& \& 4 \& \& 5 \& \& 6 \&  \\
	 \& 4 \& \& 3 \& \& 3 \& \& 2 \& \& 3 \& \& 4 \& \& 5 \& \& \cdots \\
	\cdots \& \& 3 \& \& 2 \& \& 2 \& \& 2 \& \& 3 \& \& 4 \& \& 4 \&  \\
	 \& 3 \& \& 2 \& \& |[defect]| 1 \& \& 2 \& \& |[defect]| 2 \& \& 3 \& \& 3 \& \& \cdots \\
	\cdots \& \& 2 \& \& 1 \& \& 1 \& \& 2 \& \& 2 \& \& 2 \& \& 2 \&  \\
	 \& |[throw]| 1 \& \& |[throw]| 1 \& \& |[defect]| 0 \& \& |[origin]| 1 \& \&|[throw]| 1 \& \& |[throw]| 1 \& \& |[throw]| 1 \& \& \cdots \\
	};
\end{tikzpicture}
\]
The subdiagonal entries have been omitted. The \textcolor{dark red}{red circles} are the \emph{defects} (see below).
\end{ex}

\begin{prop}\label{prop: rank}
If $\mathsf{R}$ is the rank matrix of $V$, then the following hold for any $a,b\in \mathbb{Z}$.
\begin{enumerate}
	\item $\R_{a,b}-\R_{a+1,b}$ must be $0$ or $1$.
	\item $\R_{a,b}-\R_{a,b-1}$ must be $0$ or $1$.
	\item $\R_{a,b}-\R_{a+1,b}-\R_{a,b-1}+\R_{a+1,b-1}$ must be $0$ or $-1$.
\end{enumerate}
\end{prop}

\begin{proof}
The projection $V_{[a,b]}\rightarrow V_{[a+1,b]}$ is surjective with at most 1-dimensional kernel. 
This proves the first result; the second is proven similarly.

The first result implies that $-1\leq (\R_{a,b}-\R_{a+1,b})-(\R_{a,b-1} - \R_{a+1,b-1}) \leq 1$. The map $V_{[a+1,b]} \oplus V_{[a,b-1]}\rightarrow V_{[a,b]}$ which sends $(\mathsf{v},\mathsf{w})$ to $\mathsf{v+w}$ is a surjection whose kernel is the image of the map $V_{[a+1,b-1]} \rightarrow V_{[a+1,b]} \oplus V_{[a,b-1]}$ which sends $\mathsf{v}$ to $(\mathsf{v},-\mathsf{v})$. Therefore,
\[ \dim(V_{[a,b]}) \leq \dim(V_{[a+1,b]} \oplus V_{[a,b-1]}) - \dim(V_{[a+1,b-1]}) \]
This proves that $\R_{a,b}-\R_{a+1,b}-\R_{a,b-1} + \R_{a+1,b-1} \leq 0$.
\end{proof}

Let us say the pair $(a,b)\in \mathbb{Z}\times\mathbb{Z}$ is a \textbf{defect} of a rank matrix $\mathsf{R}$ if 
\[ \R_{a,b} - \R_{a+1,b}-\R_{a,b-1}+\R_{a+1,b-1} =-1 \]
\begin{prop} \label{prop: defect}
The defects of a rank matrix $\R$ have the following properties.
\begin{enumerate}
	\item $\R_{a,b} = (b-a+1) - \#(\text{defects in the box }[a,b]\times [a,b]) $.
	\item Each row and column of a rank matrix can contain at most one defect.
	\item 
	If $[a,b]\times \{b\}$ does not contain any defects,
	then $V_{[a,b]}$ contains the vector $(0,0,...,0,1)$.
	\item 
	If $\{a\}\times [a,b]$ does not contain any defects,
	then $V_{[a,b]}$ contains the vector $(1,0,...,0,0)$.
\end{enumerate}
\end{prop}
\begin{proof}
Fix $a$ and consider the sequence $(\R_{a,b}-\R_{a+1,b})$ for all $b$. This sequence starts at $1$ for sufficiently negative $b$, switches from $1$ to $0$ whenever $(a,b)$ is a defect, and must remain $0$ once it does (by Proposition \ref{prop: rank}.3). Since there are no defects when $a<b$, this implies that
\[ \R_{a,b}=\R_{a+1,b}+1-\#(\text{defects in the line }\{a\}\times [a,b])\]
In particular, there can be at most one defect in each row, and inductively implies that 
\[ \R_{a,b} = (b-a+1) - \#(\text{defects in the box }[a,b]\times [a,b]) \]
If $\{a\}\times [a,b]$ does not contain any defect, then $\R_{a,b}=\R_{a+1,b}+1$ and the map $V_{[a,b]}\rightarrow V_{[a+1,b]}$ has $1$-dimensional kernel. This kernel must be spanned by the vector $(1,0,...,0,0)$.

The remaining results follow by a dual argument on  the sequence $(\R_{a,b}-\R_{a,b-1})$.
 \end{proof}

 \subsection{$\R$-schedules}
 
Given a rank matrix $\R$ and an interval\footnote{We explicitly allow our intervals to be unbounded, so $(-\infty,0]$ and $(-\infty,\infty)=\ZZ$ are intervals.} $I\subset \ZZ$, an \textbf{$\R$-schedule for $I$} is a subset $J\subset I$ for which there is a sequence of finite subintervals
\begin{equation}
\label{eq: Rschedule}
[a_0,b_0] \subset [a_1,b_1] \subset [a_2,b_2] \subset \cdots  \subset I
\end{equation}
such that $b_i-a_i=i$, $\bigcup [a_i,b_i] = I$, and $|J\cap [a_i,b_i]| = \R_{a_i,b_i}$. Note that the sequence of intervals determines the $\R$-schedule, and, for all $i$, $J\cap [a_i,b_i]$ is an $\R$-schedule for $[a_i,b_i]$.

\begin{lemma}\label{lemma: Tset}
Given a closed subspace $V\subset \k^\ZZ$ with rank matrix $\R$, and an $\R$-schedule $J$ for an interval $I$, the restriction map $V_{I}\rightarrow \k^J$ is an isomorphism.
\end{lemma}

\noindent In particular, if $J$ is an $\R$-schedule for $\ZZ$, then $V\rightarrow \k^J$ is an isomorphism.

\begin{proof}
We prove the case when $I=[a,b]$ by induction on $n:=b-a$. If $n<0$, the lemma holds vacuously. Assume that the lemma holds for all intervals shorter than $n$. Choose a sequence of subintervals as in \eqref{eq: Rschedule}, and set $[a',b']:=[a_{n-1},b_{n-1}]$.

The restriction maps fit into a commutative diagram.
\begin{equation*}
\begin{tikzpicture}[baseline=(current bounding box.center)]
	\node (V2) at (0,0) {$V_{[a,b]}$};
	\node (V3) at (3,0) {$V_{[a',b']}$};
	\node (k2) at (0,-1.5) {$\mathsf{k}^{J}$};
	\node (k3) at (3,-1.5) {$\mathsf{k}^{J\cap [a',b']}$};
	\draw[->>] (V2) to (V3);
	\draw[->] (V2) to node[left] {$\pi_{J}$} (k2);
	\draw[->] (V3) to node[right] {$\pi_{J\cap [a',b']}$} (k3);
	\draw[->>] (k2) to (k3);
\end{tikzpicture}
\end{equation*}
By the inductive hypothesis, $\pi_{J\cap [a',b']}$ is an isomorphism, and so $V_{[a,b]}\rightarrow \k^{T_{[a',b']}}$ is surjective.

Since $b'-a'=n-1$, either $[a',b']=[a+1,b]$ or $[a',b']=[a,b-1]$.
We have three cases.
\begin{itemize}
	\item If $\R_{a',b'}=\R_{a,b}$, then $J\cap [a',b']=J$ and so the bottom arrow is an isomorphism. Therefore, $\pi_J:V_{[a,b]}\rightarrow \k^J$ is surjective.
	\item If $\R_{a',b'}=\R_{a,b}-1$ and $[a',b']=[a+1,b]$, then there are no defects in $\{a\} \times [a,b]$, and so $V_{[a,b]}$ contains $(1,0,...,0,0)$ (by Proposition \ref{prop: defect}.4). Since $|J| =|J\cap [a',b']|+1$, $a\in J$ and so the image of $(1,0,...,0,0)$ under $\pi_{J}$ is non-zero and spans the kernel of $\k^{J}\rightarrow \k^{J\cap[a',b']}$. Thus, $\pi_J:V_{[a,b]}\rightarrow \k^{J}$ is surjective.
	\item If $\R_{a',b'}=\R_{a,b}-1$ and $[a',b']=[a,b-1]$, then there are no defects in $[a,b] \times \{b\}$, and so $V_{[a,b]}$ contains $(0,0,...,0,1)$ (by Proposition \ref{prop: defect}.3). By an analogous argument to the previous case, $\pi_J:V_{[a,b]}\rightarrow \k^{J}$ is surjective.
\end{itemize}
The map $\pi_J:V_{[a,b]}\rightarrow \k^{J}$ is surjective in all cases.
Since
$ \dim(V_{[a,b]}) = \R_{[a,b]}= J = \dim(\k^{J}) $
this map is an isomorphism, completing the induction and proving the lemma for finite $I$.

Now, consider general $I$. Fix a sequence of intervals $[a_i,b_i]$ as in \eqref{eq: Rschedule}. For each $i$, $\pi_{[a,b]}(\mathrm{ker}(\pi_J))\subset \mathrm{ker}(\pi_{J\cap [a,b]})$, which is $\{0\}$ by the preceding induction, and so $\pi_J$ is injective.

For any $\mathsf{w}\in \k^J$ and any $i$, the restriction $\pi_{J\cap [a_i,b_j]}(\mathsf{w})$ lifts to a unique $\mathsf{w}^{(i)}\in V_{[a_i,b_i]}$, which in turn can be lifted to a (non-unique) $\widetilde{\mathsf{w}}^{(i)}\in V$. Uniqueness forces $\pi_{[a_i,b_i]}(\widetilde{\mathsf{w}}^{(j)}) = \mathsf{w}^{(i)}$ for all $j\geq i$, and so the sequence $\widetilde{\mathsf{w}}^{(i)}$ converges to some $\widetilde{\mathsf{w}}\in \k^\ZZ$ with $\pi_J(\widetilde{\mathsf{w}})=\mathsf{w}$. Since each $\widetilde{\mathsf{w}}^{(i)}\in V$ and $V$ is closed, $\widetilde{\mathsf{w}}\in V$, and so $\pi_J$ is surjective.
\end{proof}

\subsection{Recurrence matrices from rank matrices}

We can now characterize when a closed subspace of $\k^\ZZ$ is the kernel of a reduced recurrence matrix. 

\begin{thm}\label{thm: reduced2}
Given a closed subspace $V$ of $\mathsf{k}^\mathbb{Z}$, the following are equivalent.
\begin{enumerate}
	\item $V$ is the space of solutions to a linear recurrence $\C\mathsf{x}=\mathsf{0}$.
	\item The only left-bounded sequence in $V$ is the zero sequence; that is, if $\mathsf{v}\in V$ and $\mathsf{v}_i=0$ for all $i\ll0$, then $\mathsf{v}_i=0$ for all $i$.
	\item Every column of the rank matrix $\R$ of $V$ contains a defect.
	\item $V$ is the space of solutions to a unique reduced linear recurrence $\overline{\C}\mathsf{x} =\mathsf{0}$.
\end{enumerate}
The shape of $\overline{\C}$ is the function $S:\ZZ\rightarrow\ZZ$ such that $(S(b),b)$ is a defect of $\R$.
\end{thm}

\begin{proof}
$(4 \Rightarrow 1)$ is automatic.

$(1 \Rightarrow 2)$ If a sequence $\mathsf{v}$ solves a linear recurrence, then every term in $\mathsf{v}$ is equal to a linear combination of previous terms in the sequence. If every term in $\mathsf{v}$ of sufficiently negative index is $0$, then recursively every term must be $0$.

$(\text{not } 3 \Rightarrow \text{not }2)$ Assume that the $b$th column of the rank matrix of $V$ does not contain a defect. By Proposition \ref{prop: defect}.3, $V_{[a,b]}$ contains the vector $(0,0,...,0,1)$ for all $a\leq b$. It follows that $V_{(-\infty,b]}$ contains the vector $(...,0,0,1)$. This implies that $V$ contains a sequence $\mathsf{v}$ with $\mathsf{v}_b=1$ and $\mathsf{v}_a=0$ whenever $a<b$.

$(3 \Rightarrow 4)$ Assume that there is a function $S:\ZZ\rightarrow \ZZ$ such that $(S(b),b)$ is a defect of $\R$ for each $b$. For each interval $[a,b]$, define the $\R$-schedule $T_{[a,b]} := [a,b] \smallsetminus S([a,b])$.
We note that $T_{[S(b),b-1]} \cup \{b\} = T_{[S(b)+1,b]}\cup \{S(b) \}$
and consider the following commutative diagram.
\[ \begin{tikzpicture}
	\node (V1) at (-3,0) {$V_{[S(b),b-1]}$};
	\node (V2) at (0,0) {$V_{[S(b),b]}$};
	\node (V3) at (3,0) {$V_{[S(b)+1,b]}$};
	\node (k1) at (-3,-1.5) {$\mathsf{k}^{T_{[S(b),b-1]}}$};
	\node (k2) at (0,-1.5) {$\mathsf{k}^{T_{[S(b),b-1]}\cup \{b\} }$};
	\node (k3) at (3,-1.5) {$\mathsf{k}^{T_{[S(b)+1,b]}}$};
	\draw[->] (V2) to (V1);
	\draw[->] (V2) to (V3);
	\draw[->] (V1) to (k1);
	\draw[->] (V2) to (k2);
	\draw[->] (V3) to (k3);
	\draw[->] (k2) to (k1);
	\draw[->] (k2) to (k3);
\end{tikzpicture}\]
Since $(S(b),b)$ is a defect, the maps in the top row are isomorphisms. Since $T_{[S(b),b-1]}$ and $T_{[S(b)+1,b]}$ are $\R$-schedules, the maps on the left and right are isomorphisms (by Lemma \ref{lemma: Tset}). 

Therefore, the map $V_{[S(b),b]} \longrightarrow \mathsf{k} ^{J }$ is an embedding of codimension $1$. Its image is defined by a relation (unique up to scaling) of the form
\begin{equation}\label{eq: relation}
 \sum_{a\in T_{[S(b),b-1]}\cup\{b\} } \C_{b,a}x_{a} = 0
\end{equation}
Because the left and right maps are isomorphisms, $\C_{b,b}\neq0$ and $\C_{S(b),b}\neq0$. 
Rescaling the relation as necessary, we assume that $\C_{b,b}=1$. 

Construct a recurrence matrix $\C$ such that, for each $b$, the $b$th row collects the coefficients of the corresponding equation \eqref{eq: relation}. For any pair $b<a$, $S(b)\not \in T_{[S(a)+1,a]}$ and so $\C_{a,S(b)}=0$. Therefore, $\C$ is a reduced linear recurrence of shape $S$, such that $V\subseteq \ker(\C)$.

%

Consider any interval $[a,b]$. 
For each $b'\in [a,b] \smallsetminus T_{[a,b]}$, the corresponding relation \eqref{eq: relation} only involves terms with index in $[a,b]$. Since these relations are linearly independent, the codimension of $\ker(\C)_{[a,b]}$ in $\mathsf{k}^{[a,b]}$ is at least the cardinality of $[a,b]\smallsetminus T_{[a,b]}$. Therefore,
\[ \dim(\ker(\C)_{[a,b]})\leq |T_{[a,b]}| = \R_{a,b} = \dim(V_{[a,b]}) \]
Since $V_{[a,b]}\subseteq \ker(\C)_{[a,b]}$, $V_{[a,b]}=\ker(\C)_{[a,b]}$. Since both subspaces are closed, $V=\ker(\C)$. 
\end{proof}

\subsection{From rank matrices to shapes}

The theorem relates the shape of a reduced recurrence matrix $\C$ to the defects of the rank matrix $\R$ of $\ker(\C)$, as follows.



\begin{prop}
If $\C$ is a reduced recurrence matrix, then $(a,b)$ is a pivot of $\C$ if and only if $(b,a)$ is a defect of the rank matrix of $\ker(\C)$. 
\end{prop}

\noindent Therefore, $S(a)$ is the largest integer $\leq a$ such that $\dim(\ker(\C)_{[S(a),a]})=\dim(\ker(\C)_{[S(a),a-1]})$.


\begin{defn}\label{defn: schedule}
Given a non-increasing injection $S$ and an interval\footnote{We remind the reader that we allow `intervals' in $\ZZ$ to be unbounded.} $I\subset \ZZ$, an \textbf{$S$-schedule for $I$} is a subset $J\subset I$ for which there is a subsequence of subintervals
\begin{equation*}
[a_0,b_0] \subset [a_1,b_1] \subset [a_2,b_2] \subset \cdots \subset I
\end{equation*}
such that $b_i-a_i=i$, $\bigcup [a_i,b_i] = I$, and $|J\cap [a_i,b_i]|$  equals the number of $S$-balls in $[a_i,b_i]$.
\end{defn}
If $S$ is the shape of a reduced recurrence matrix $\C$ and $\R$ is the rank matrix of $\mathrm{ker}(\C)$, then $\R$-schedules and $S$-schedules coincide.
The following is a direct translation of Lemma \ref{lemma: Tset}.

\begin{prop}\label{prop: Tset}
Let $\C$ be a reduced recurrence matrix with shape $S$, and let $J$ be an $S$-schedule for $I$. Then the restriction map $\mathrm{ker}(\C)_{I}\rightarrow \k^J$ is an isomorphism.
%
\end{prop}


%
\noindent In particular, if $J$ is an $S$-schedule for $\ZZ$, then $\mathrm{ker}(\C)\rightarrow \k^J$ is an isomorphism.
%
%

As a special case, for any $b\in \ZZ$, the sequence of intervals $[b,b] \subset [b-1,b] \subset [b-2,b] \subset \cdots $ determines the following $S$-schedule for $\ZZ$: 
\[ T_b := \bigcup_{a\leq b} T_{[a,b]} = \{ a\leq b \mid \forall c\leq b, S(c) \neq a \} = (-\infty,b] \smallsetminus S\left( ( -\infty,b]\right) \]
\begin{thm}\label{thm: solextend}
The restriction map $\pi_{T_b}:\mathrm{ker}(\C)\rightarrow \k^{T_b}$ is an isomorphism.
\end{thm}

\noindent Since $T_b$ contains a unique representative of each $S$-ball, this implies the following.

\begin{thm}\label{thm: balls}
Then dimension of $\mathrm{ker}(\C)$ equals the number of $S$-balls.
\end{thm}

%
%

\begin{rem}
Constructing $S$-schedules is easy and intuitive using the juggling pattern of $S$. Consider any zigzagging path in the juggling pattern which starts on the main diagonal, only travels up (northwest) or right (northeast), and ends above the $(a,b)$th entry.
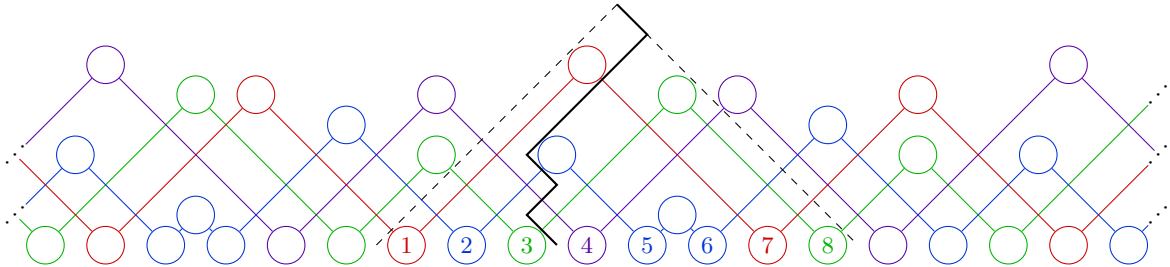
\[\begin{tikzpicture}[baseline=(current bounding box.center),
	ampersand replacement=\&,
	]
	\clip[use as bounding box] (-9.5,3.3) rectangle (6.3,-0.3);
	\matrix[matrix of math nodes,
		matrix anchor = M-2-24.center,
		nodes in empty cells,
		inner sep=0pt,
		gthrow/.style={dark green,draw,circle,inner sep=0mm,minimum size=5mm},
		rthrow/.style={dark red,draw,circle,inner sep=0mm,minimum size=5mm},
		bthrow/.style={dark blue,draw,circle,inner sep=0mm,minimum size=5mm},
		pthrow/.style={dark purple,draw,circle,inner sep=0mm,minimum size=5mm},
		nodes={anchor=center,node font=\scriptsize,rotate=45},
		column sep={0.4cm,between origins},
		row sep={-0.4cm,between origins},
	] (M) at (0,0) {
 \&  \&  \&  \&  \&  \&  \&  \&  \&  \&  \&  \&  \&  \&  \&  \&  \&  \&  \&  \&  \&  \&  \&  \&  \&  \&  \&  \&  \&  \&  \&  \&  \&  \&  \&  \&  \&  \&  \\
 \& |[gthrow]|  \&  \& |[rthrow]|  \&  \& |[bthrow]|  \&  \& |[bthrow]|  \&  \& |[pthrow]|  \&  \& |[gthrow]|  \&  \& |[rthrow,rotate=-45]| 1 \&  \& |[bthrow,rotate=-45]| 2 \&  
 \& |[gthrow,rotate=-45]| 3 \& |[inner sep=8pt]| \& |[pthrow,rotate=-45]| 4 \&  \& |[bthrow,rotate=-45]| 5 \&  \& |[bthrow,rotate=-45]| 6 \&  \& |[rthrow,rotate=-45]| 7 \&  \& |[gthrow,rotate=-45]| 8 \&  \& |[pthrow]|  \&  \& |[bthrow]|  \& 
  \& |[gthrow]|  \&  \& |[rthrow]|  \&  \& |[bthrow]|  \&  \\
\cdots \&  \&  \&  \&  \&  \& |[bthrow]|  \&  \&  \&  \&  \&  \&  \&  \&  \&  \&  \&  \&  \&  \&  \&  \& |[bthrow]|  \&  \&  \&  \&  \&  \&  \&  \&  \&  \&  \&  \&  \&  \&  \&  \& \cdots \\
 \&  \&  \&  \&  \&  \&  \&  \&  \&  \&  \&  \&  \&  \&  \&  \&  \&  \&  \&  \&  \&  \&  \&  \&  \&  \&  \&  \&  \&  \&  \&  \&  \&  \&  \&  \&  \&  \&  \\
\cdots \&  \& |[bthrow]|  \&  \&  \&  \&  \&  \&  \&  \&  \&  \&  \&  \& |[gthrow]|  \&  \&  \&  \& |[bthrow]|  \&  \&  \&  \&  \&  \&  \&  \&  \&  \&  \&  \& |[gthrow]|  \&  \&  \&  \& |[bthrow]|  \&  \&  \&  \& \cdots \\
 \&  \&  \&  \&  \&  \&  \&  \&  \&  \&  \& |[bthrow]|  \&  \&  \&  \&  \&  \&  \&  \&  \&  \&  \&  \&  \&  \&  \&  \& |[bthrow]|  \&  \&  \&  \&  \&  \&  \&  \&  \&  \&  \&  \\
 \&  \&  \&  \&  \&  \& |[gthrow]|  \&  \& |[rthrow]|  \&  \&  \&  \&  \&  \& |[pthrow]|  \&  \&  \&  \&  \&  \&  \&  \& |[gthrow]|  \&  \& |[pthrow]|  \&  \&  \&  \&  \&  \& |[rthrow]|  \&  \&  \&  \&  \&  \&  \&  \& \cdots \\
 \&  \&  \& |[pthrow]|  \&  \&  \&  \&  \&  \&  \&  \&  \&  \&  \&  \&  \&  \&  \&  \& |[rthrow]|  \&  |[inner sep=8pt]|  \&  \&  \&  \&  \&  \&  \&  \&  \&  \&  \&  \&  \&  \&  \& |[pthrow]|  \&  \&  \&  \\
 \&  \&  \&  \&  \&  \&  \&  \&  \&  \&  \&  \&  \&  \&  \&  \&  \&  \&  \&  \&  \&  \&  \&  \&  \&  \&  \&  \&  \&  \&  \&  \&  \&  \&  \&  \&  \&  \& \\
 \&  \&  \&  \&  \&  \&  \&  \&  \&  \&  \&  \&  \&  \&  \&  \&  \&  \&  \&  \&  \&  \&  \&  \&  \&  \&  \&  \&  \&  \&  \&  \&  \&  \&  \&  \&  \&  \& \\
 \&  \&  \&  \&  \&  \&  \&  \&  \&  \&  \&  \&  \&  \&  \&  \&  \&  \&  \&  \&  \&  \&  \&  \&  \&  \&  \&  \&  \&  \&  \&  \&  \&  \&  \&  \&  \&  \& \\
	};

	\draw[dark green] (M-3-1) to (M-2-2) to (M-7-7) to (M-2-12) to (M-5-15) to (M-2-18) to (M-7-23) to (M-2-28) to (M-5-31) to (M-2-34) to (M-7-39);
	\draw[dark red] (M-5-1) to (M-2-4) to (M-7-9) to (M-2-14) to (M-8-20) to (M-2-26) to (M-7-31) to (M-2-36) to (M-5-39);
	\draw[dark blue] (M-3-1) to (M-5-3) to (M-2-6) to (M-3-7) to (M-2-8) to (M-6-12) to (M-2-16) to (M-5-19) to (M-2-22) to (M-3-23) to (M-2-24) to (M-6-28) to (M-2-32) to (M-5-35) to (M-2-38) to (M-3-39);
	\draw[dark purple] (M-5-1) to (M-8-4) to (M-2-10) to (M-7-15) to (M-2-20) to (M-7-25) to (M-2-30) to (M-8-36) to (M-5-39);
	
	\draw[dashed] (M-2-13.center) to (M-10-21.center) to (M-2-29.center);
	\draw[thick] (M-2-19.center) to (M-3-18.center) to (M-4-19.center) to (M-5-18.center) to (M-9-22.center)to (M-10-21.center);
\end{tikzpicture}\]
Each time the path crosses a colored line, record the row (if it is ascending) or the column (if it is descending).
The resulting subset is an $S$-schedule for the interval $[a,b]$, and every $S$-schedule can be constructed this way.
In the picture above, the path in black determines the $S$-schedule $\{\textcolor{dark green}{3},\textcolor{dark purple}{4},\textcolor{dark blue}{2},\textcolor{dark red}{7}\}$ for the interval $[1,8]$. The set $T_b$ comes from the path which starts to the right of $(b,b)$ and only travels up (northwest).

When the path only travels up/northwest (resp. right/northeast), the corresponding $S$-schedule describes the times when the balls in the air at a given moment were thrown (resp. will land), and so have been called \emph{throwing histories} (resp. \emph{landing schedules}). 
\end{rem}

\section{Properties of the solution matrix}


Fix a reduced recurrence matrix $\C$ of shape $S$ for the rest of the section. 

\subsection{Vanishing}

The unitriangularity of $\adj(\C)$ and $\adj(\C\P)$ mean that the solution matrix
\[ \sol(\C) := \adj(\C) - \P\adj(\C\P)\]
has zeroes between the support of $\adj(\C)$ and $\P\adj(\C\P)$, which we make precise as follows.



\begin{prop}\label{prop: colvan}
$\sol(\C)_{a,b}=0$ whenever $a<b< S^{-1}(a)$.
\end{prop}
\noindent Here, $S^{-1}(a)$ is the preimage of $a$, and if $S^{-1}(a)=\varnothing$ then $b<S^{-1}(a)$ is vacuously true.

\begin{proof}
By unitriangularity, $\adj(\C)_{a,b}=0$ whenever $a<b$. Dually, $(\P\adj(\C\P))_{a,b}$ is only non-zero if there is a $c$ with $\P_{a,c}\neq0$ and $\adj(\C\P)_{c,b}\neq0$; that is, if $S(c)=a$ and $c\geq b$.
\end{proof}

For fixed $b$, the proposition determines the value of the $b$th column of $\sol(\C)$ on the set $T_b$.
Since this column solves $\C\mathsf{x}=0$, these entries determine the column (Theorem \ref{thm: solextend}).

%

\begin{coro}
The $b$th column of $\sol(\C)$ is the unique solution to $\C\mathsf{x}=\mathsf{0}$ for which 
\begin{enumerate}
	\item $x_a=0$ whenever $a<b< S^{-1}(a)$, and
	\item $x_b=1$ unless $S(b)=b$, in which case $x_b=0$.
\end{enumerate}
\end{coro}
\begin{thm}\label{thm: solvan}
The solution matrix $\sol(\C)$ is the unique $\ZZ\times \ZZ$ matrix such that 
\begin{enumerate}
	\item $\C\sol(\C)=0$,
	\item $\sol(\C)_{a,a}=1$ whenever $S(a)\neq a$, and
	\item $\sol(\C)_{a,b}=0$ whenever $S(a)=b=a$ or $a<b<S^{-1}(a)$.
\end{enumerate}
\end{thm}

%

Dually, we have a vanishing condition guaranteeing consecutive zeros in each column.

\begin{prop}\label{prop: rowvan}
$\sol(\C)_{a,b}=0$ whenever $S(b) < a < b$.
\end{prop}
\begin{proof}
We proceed by induction on $b-a>0$. Assume that $\sol(\C)_{a+1,b}=...=\sol(\C)_{b-1,b}=0$ (which is vacuous for the base case $b-a=1$). We split into two cases.
\begin{itemize}
	\item If there is no $c\leq b$ with $S(c)=a$, then $\sol(\C)_{a,b}=0$ by Proposition \ref{prop: colvan}.
	\item Otherwise, there is a $c\leq b$ with $S(c)=a$. Since $S(b)<a$, we know $c\neq b$ and so $c<b$. The $(c,b)$th entry of the equality $\C\sol(\C)=0$ is
	\[ \C_{c,a}\sol(\C)_{a,b} + \C_{c,a+1}\sol(\C)_{a+1,b} + ... + \C_{c,c}\sol(\C)_{c,b} =0 \]
	By assumption, $\sol(\C)_{a+1,b}=...=\sol(\C)_{c,b}=0$, and so $\C_{c,a}\sol(\C)_{a,b}=0$. Since $\C_{c,a}=\C_{c,S(c)}$ is invertible, $\sol(\C)_{a,b}=0$.
\end{itemize}
This completes the induction.
\end{proof}


\begin{rem}\label{rem: soljugs}
Propositions \ref{prop: colvan} and \ref{prop: rowvan} can be visualized in terms of juggling patterns. As in Remark \ref{rem: balljugs}, the $S$-balls can be visualized by connecting each pivot entry with a line to the diagonal entries in the same row or column. The dashed circle is not in an $S$-ball.
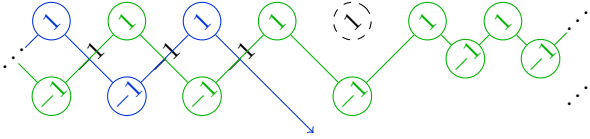
\[ \begin{tikzpicture}[baseline=(current bounding box.center),
	ampersand replacement=\&,
	]
	\matrix[matrix of math nodes,
		matrix anchor = M-4-8.center,
		gthrow/.style={dark green,draw,circle,inner sep=0mm,minimum size=5mm},
		bthrow/.style={dark blue,draw,circle,inner sep=0mm,minimum size=5mm},
		dashedthrow/.style={dashed,draw,circle,inner sep=0mm,minimum size=5mm},
		nodes in empty cells,
		inner sep=0pt,
		nodes={anchor=center,
			node font=\footnotesize,
			rotate=45},
		column sep={.5cm,between origins},
		row sep={.5cm,between origins},
	] (M) at (0,0) {
	 \& |[bthrow]| 1 \& \& |[gthrow]| 1 \& \& |[bthrow]| 1 \& \& |[gthrow]| 1 \& \&|[dashedthrow]| 1 \& \& |[gthrow]| 1 \& \& |[gthrow]| 1 \& \& \cdots \\
	\cdots \& \& -1 \& \& -1 \& \& -1 \& \&  \& \&  \& \& |[gthrow]|  -1 \& \& |[gthrow]| -1 \& \\
	 \& |[gthrow]| -1 \&  \& |[bthrow]| -1 \& \& |[gthrow]| -1 \& \&  \& \& |[gthrow]| -1 \& \& \& \&  \& \& \cdots \\
	 \& \&  \& \&  \& \&  \& \& \&  \&  \& \&  \& \&  \&  \\
	};
	
	\draw[dark blue,->] (M-2-1) to (M-1-2) to (M-3-4) to (M-1-6) to (M-4-9);
	\draw[dark green] (M-2-1) to (M-3-2) to (M-1-4) to (M-3-6) to (M-1-8) to (M-3-10) to (M-1-12) to (M-2-13) to (M-1-14) to (M-2-15) to (M-1-16);
\end{tikzpicture} \]
If we transpose and superimpose the juggling pattern onto the solution matrix, we get
\[ \begin{tikzpicture}[baseline=(current bounding box.center),
	ampersand replacement=\&,
	]
	\matrix[matrix of math nodes,
		matrix anchor = M-8-8.center,
		gthrow/.style={dark green,draw,circle,inner sep=0mm,minimum size=5mm},
		bthrow/.style={dark blue,draw,circle,inner sep=0mm,minimum size=5mm},
		dashedthrow/.style={dashed,draw,circle,inner sep=0mm,minimum size=5mm},
		faded/.style={black!25},
		nodes in empty cells,
		inner sep=0pt,
		nodes={anchor=center,node font=\footnotesize,rotate=45},
		column sep={.5cm,between origins},
		row sep={.5cm,between origins},
	] (M) at (0,0) {
	 \& \& \cdots \& \& \cdots \& \& \cdots \& \& \cdots \& \& \cdots \& \& \cdots \& \& \cdots \&  \\
	 \& 5 \&  \& |[faded]| 0 \& \& 2 \& \& -1 \& \& 1 \& \& |[faded]| 0 \& \& 1 \& \& \cdots \\
	\cdots \& \& -3 \& \& |[faded]| 0 \& \& -1 \& \& 1 \& \& |[faded]| 0 \& \& 1 \& \& |[faded]| 0 \&  \\
	 \& 2 \&  \& 2 \& \& |[faded]| 0 \& \& 1 \& \& |[faded]| 0 \& \& 1 \& \& |[faded]| 0 \& \& \cdots \\
	\cdots \& \& -1 \& \& -1 \& \& |[faded]| 0 \& \& |[faded]| 0 \& \& 1 \& \& |[faded]| 0 \& \& 1 \&  \\
	 \& |[gthrow]| 1 \&  \&  |[bthrow]| 1 \& \& |[gthrow]| 1 \& \& |[faded]| 0 \& \& |[gthrow]| 1 \& \& |[faded]| 0 \& \& 1 \& \& \cdots \\
	\cdots \& \& |[faded]| 0 \& \& |[faded]| 0 \& \& |[faded]| 0 \& \& |[faded]| 0 \& \& |[faded]| 0 \& \& |[gthrow]| 1 \& \& |[gthrow]| 1 \&  \\
	 \& |[bthrow]| 1 \& \& |[gthrow]| 1 \& \& |[bthrow]| 1 \& \& |[gthrow]| 1 \& \& |[dashedthrow,faded]| 0 \& \&|[gthrow]| 1 \& \& |[gthrow]| 1 \& \& \cdots \\
	\cdots \& \& 1 \& \& 1 \& \& 1 \& \& |[faded]| 0 \& \& |[faded]| 0 \& \& 1 \& \& 1 \& \\
	\& 2 \& \& 2 \& \& 2 \& \& |[faded]| 0 \& \& 1 \& \& |[faded]| 0 \& \& 1 \& \& \cdots \\
	\cdots \& \& 3 \& \& 3 \& \& |[faded]| 0 \& \& 1 \& \& 1 \& \& |[faded]| 0 \& \& 1 \&   \\
	  \& 5 \& \& 5 \& \& |[faded]| 0 \& \& 2 \& \& 1 \& \& 1 \& \& |[faded]| 0 \& \&  \cdots \\
	\cdots \& \& 8 \& \& |[faded]| 0 \& \& 3 \& \& 2 \& \& 1 \& \& 1 \& \& |[faded]| 0 \& \\
	  \& 13 \& \& |[faded]| 0 \& \& 5 \& \& 3 \& \& 2 \& \& 1 \& \& 1 \& \&  \cdots \\
	\cdots \& \& \cdots \& \& \cdots \& \& \cdots \& \& \cdots \& \& \cdots \& \& \cdots \& \& \cdots \& \\
	};

	\draw[dark blue,->] (M-7-1) to (M-8-2) to (M-6-4) to (M-8-6) to (M-1-13);
	\draw[dark green] (M-7-1) to (M-6-2) to (M-8-4) to (M-6-6) to (M-8-8) to (M-6-10) to (M-8-12) to (M-7-13) to (M-8-14) to (M-7-15) to (M-8-16);
\end{tikzpicture}\]
Propositions \ref{prop: colvan} and \ref{prop: rowvan} state that the lines do not cross any non-zero entries. The circles must be non-zero, except the dashed circle, whose entire row and column must vanish.\footnote{We note that all of these observations will be simultaneously generalized by Lemma \ref{lemma: boxballs}.}
%
\end{rem}

\subsection{Determinantal formulas}


Since the adjugate of $\C$ is defined in terms of certain determinants of $\C$, the solution matrix of $\C$ can be expanded in terms of such determinants.

\begin{prop}\label{prop: soldet}
The entries of the solution matrix can be computed as follows.
\[\sol(\C)_{a,b} = \left\{ \begin{array}{cc}
(-1)^{a+b} \det(\C_{[b+1,a],[b,a-1]}) & \text{if $a\geq b$ and $S(a)\neq a$} \\
(-1)^{S^{-1}(a)+b+1} \det(\C_{I\smallsetminus \{b\},S(I)\smallsetminus \{a\} }) \prod_{i=I} \C_{i,S(i)}^{-1}& \text{if $a<b$ and $S^{-1}(a)\neq \varnothing$} \\\
0 & \text{otherwise}
\end{array}\right\} \]
where $I := \{ i\in [S^{-1}(a),b] \mid S(i) \in [a,S(b)] \}$.
\end{prop}
\noindent Note that, in the second case above, the columns of $\C_{I\smallsetminus \{b\},S(I)\smallsetminus \{a\} }$ are in the increasing order in $I$, not in $S(I)$. That is, the first column is the $S(i_1)$th column of $\C$, where $i_1$ is the smallest element of $I$ which is not $S^{-1}(a)$.

\begin{proof}
%
If $a\geq b$ and $S(a)\neq a$, then $(\P\adj(\C\P))_{a,b}=0$ and so 
\[ \sol(\C)_{a,b}= \adj(\C)_{a,b} :=(-1)^{a+b} \det(\C_{[b+1,a],[b,a-1]})  \]
Next, assume that $a<b$ and $S^{-1}(a)\neq \varnothing$. Since $\adj(\C)_{a,b}=0$,
\begin{align*}
\sol(\C)_{a,b} 
&= - (\P\adj(\C\P))_{a,b} \\
&= -\P_{a,S^{-1}(a)} \adj(\C\P)_{S^{-1}(a),b} \\
&= (-1)^{1+S^{-1}(a)+b}\C_{S^{-1}(a),a}^{-1} \det\left((\C\P)_{[S^{-1}(a),b-1],[S^{-1}(a)+1,b]}\right) \\
&= (-1)^{1+S^{-1}(a)+b}\C_{S^{-1}(a),a}^{-1} \det\left(\C_{[S^{-1}(a),b-1],S([S^{-1}(a)+1,b])}\right)\prod_{i\in [S^{-1}(a)+1,b]} \C_{i,S(i)}^{-1} \\
&= (-1)^{1+S^{-1}(a)+b} \det\left(\C_{I\smallsetminus \{b\},S(I)\smallsetminus \{a \} }\right)\prod_{i\in I} \C_{i,S(i)}^{-1} 
\end{align*}
The only entries of $\sol(\C)$ in which $\adj(\C)$ and $\P\adj(\C\P)$ can both be non-zero are $\C_{a,a}$ when $S(a)=a$, in which case $\adj(\C)_{a,a} =1$ and $(\P\adj(\C\P))_{a,a}=1$, and so $\sol(\C)_{a,a}=0$. For all remaining choices of $a,b$, $\adj(\C)_{a,b} =(\P\adj(\C\P))_{a,b}=0$ and so $\sol(\C)_{a,b}=0$.
\end{proof}

\subsection{Rank conditions}

%
%

A box $[a,b]\times [c,d]\subset \ZZ\times \ZZ$ indexes a rectangular submatrix of $\sol(\C)$. 
An \textbf{$S$-ball in the box $[a,b]\times [c,d]$} is an equivalence class in the set 
\[ \left(\{ (i,i) \mid S(i) \neq i \} \cup \{ (S(j),j) \mid S(j)\neq j\} \right) \subset [a,b] \times [c,d] \]
under the equivalence relation generated by $(i,i) \sim (S(i),i)\sim (S(i),S(i))$. 

%
%

\begin{ex}\label{ex: boxballs}
In terms of the juggling pattern, this counts the \emph{number of different colors of circles} inside the rectangular submatrix indexed by $[a,b]\times [c,d]$.
\[ \begin{tikzpicture}[baseline=(current bounding box.center),
	ampersand replacement=\&,
	]
	\matrix[matrix of math nodes,
		matrix anchor = M-8-8.center,
		gthrow/.style={dark green,draw,circle,inner sep=0mm,minimum size=5mm},
		bthrow/.style={dark blue,draw,circle,inner sep=0mm,minimum size=5mm},
		dashedthrow/.style={dashed,draw,circle,inner sep=0mm,minimum size=5mm},
		faded/.style={black!25},
		nodes in empty cells,
		inner sep=0pt,
		nodes={anchor=center,node font=\footnotesize,rotate=45},
		column sep={.5cm,between origins},
		row sep={.5cm,between origins},
	] (M) at (0,0) {
	 \& \& \cdots \& \& \cdots \& \& \cdots \& \& \cdots \& \& \cdots \& \& \cdots \& \& \cdots \&  \\
	 \& 5 \&  \& |[faded]| 0 \& \& 2 \& \& -1 \& \& 1 \& \& |[faded]| 0 \& \& 1 \& \& \cdots \\
	\cdots \& \& -3 \& \& |[faded]| 0 \& \& -1 \& \& 1 \& \& |[faded]| 0 \& \& 1 \& \& |[faded]| 0 \&  \\
	 \& 2 \&  \& 2 \& \& |[faded]| 0 \& \& 1 \& \& |[faded]| 0 \& \& 1 \& \& |[faded]| 0 \& \& \cdots \\
	\cdots \& \& -1 \& \& -1 \& \& |[faded]| 0 \& \& |[faded]| 0 \& \& 1 \& \& |[faded]| 0 \& \& 1 \&  \\
	 \& |[gthrow]| 1 \&  \&  |[bthrow]| 1 \& \& |[gthrow]| 1 \& \& |[faded]| 0 \& \& |[gthrow]| 1 \& \& |[faded]| 0 \& \& 1 \& \& \cdots \\
	\cdots \& \& |[faded]| 0 \& \& |[faded]| 0 \& \& |[faded]| 0 \& \& |[faded]| 0 \& \& |[faded]| 0 \& \& |[gthrow]| 1 \& \& |[gthrow]| 1 \&  \\
	 \& |[bthrow]| 1 \& \& |[gthrow]| 1 \& \& |[bthrow]| 1 \& \& |[gthrow]| 1 \& \& |[dashedthrow,faded]| 0 \& \&|[gthrow]| 1 \& \& |[gthrow]| 1 \& \& \cdots \\
	\cdots \& \& 1 \& \& 1 \& \& 1 \& \& |[faded]| 0 \& \& |[faded]| 0 \& \& 1 \& \& 1 \& \\
	\& 2 \& \& 2 \& \& 2 \& \& |[faded]| 0 \& \& 1 \& \& |[faded]| 0 \& \& 1 \& \& \cdots \\
	\cdots \& \& 3 \& \& 3 \& \& |[faded]| 0 \& \& 1 \& \& 1 \& \& |[faded]| 0 \& \& 1 \&   \\
	  \& 5 \& \& 5 \& \& |[faded]| 0 \& \& 2 \& \& 1 \& \& 1 \& \& |[faded]| 0 \& \&  \cdots \\
	\cdots \& \& 8 \& \& |[faded]| 0 \& \& 3 \& \& 2 \& \& 1 \& \& 1 \& \& |[faded]| 0 \& \\
	  \& 13 \& \& |[faded]| 0 \& \& 5 \& \& 3 \& \& 2 \& \& 1 \& \& 1 \& \&  \cdots \\
	\cdots \& \& \cdots \& \& \cdots \& \& \cdots \& \& \cdots \& \& \cdots \& \& \cdots \& \& \cdots \& \\
	};

	\draw[dark blue,->] (M-7-1) to (M-8-2) to (M-6-4) to (M-8-6) to (M-1-13);
	\draw[dark green] (M-7-1) to (M-6-2) to (M-8-4) to (M-6-6) to (M-8-8) to (M-6-10) to (M-8-12) to (M-7-13) to (M-8-14) to (M-7-15) to (M-8-16);
	
	\draw[dark red, fill=dark red!50,opacity=.25,rounded corners] (M-3-4.center) -- (M-2-5.center) -- (M-10-13.center) -- (M-11-12.center) -- cycle;
	\draw[dark blue, fill=dark blue!50,opacity=.25,rounded corners] (M-6-9.center) -- (M-3-12.center) -- (M-7-16.center) -- (M-10-13.center) -- cycle;
	\draw[dark purple, fill=dark purple!50,opacity=.25,rounded corners] (M-7-2.center) -- (M-4-5.center) -- (M-7-8.center) -- (M-10-5.center) -- cycle;

\end{tikzpicture}\]
The \textcolor{dark purple}{purple}, \textcolor{dark red}{red}, and \textcolor{dark blue}{blue} boxes above contain $\textcolor{dark purple}{2}$, $\textcolor{dark red}{0}$, and $\textcolor{dark blue}{1}$ $S$-balls, respectively.
\end{ex}

\begin{rem}\label{rem: balls}
This is a 2-dimensional analog of $S$-balls as defined in Section \ref{section: balls}. The relation is that $S$-balls in an interval $[a,b]$ are in bijection with $S$-balls in the box $[a,b]\times [a,b]$.
\end{rem}

%
%
%

The following lemma relates the rank of a submatrix to the number of $S$-balls in the box.

\begin{lemma}\label{lemma: boxballs}
Let $\C$ be a reduced recurrence matrix of shape $S$. For any integers $a\leq b$ and $c\leq d$, the rank of $\sol(\C)_{[a,b]\times [c,d]}$ is at least the number of $S$-balls in the box $[a,b]\times [c,d]$, with equality when 
\begin{equation}
\label{eq: boxcond}
b-c\geq -1 \text{ and }\min(a-c,b-d)\leq 0 
\end{equation}
%
\end{lemma}
Conceptually, Condition \eqref{eq: boxcond} implies the box has at least one corner on or below the first superdiagonal, and at least two corners on or above the main diagonal.

\begin{ex}
Condition \eqref{eq: boxcond} holds for the three boxes in Example \ref{ex: boxballs}, and one may check that their ranks coincide with the number of $S$-balls they contain.
\end{ex}
\begin{proof}
First, we bound the rank below by finding a full-rank submatrix of the appropriate size.
Fix $[a,b]\times [c,d]$. 
Index $I:=[c,d]\cap [a,b]\smallsetminus S([c,d])$ by $i_1<i_2<\cdots < i_k$, and index 
\[ J:= \{j\in [c,d] \mid j\not\in [a,b] \text{ and }S(j) \in [a,b] \} \]
by $j_1<j_2<\cdots < j_\ell$. Note that $i_k<j_1$ and $S(j_h)<i_1$ for all $h$.


Each $S$-ball in the box $[a,b]\times [c,d]$ contains a unique \emph{final circle}, by which we mean one of the two types of pair:
\begin{itemize}
	\item $(i,i)\in [a,b]\times [c,d]$ such that there is no $j\in [c,d]$ with $S(j)=i$.
	\item $(S(i),i) \in [a,b]\times [c,d]$ such that $(i,i)\not\in[a,b]\times[c,d]$.
\end{itemize}
Furthermore, each column in the box can contain at most one such circle. The columns containing the first kind of final circle are indexed by $I$, and the columns containing the second kind are indexed by $J$. Therefore, the number of $S$-balls in the box is $|I|+|J|$.

Consider the submatrix\footnote{We are stretching the definition of `submatrix' to allow for rearranging the order of the rows and columns.} of $\sol(\C)$ with row set
$ \{i_k,i_{k-1},...,i_1,S(j_1),S(j_2),...,S(J_\ell)\} $
and column set
$ \{i_k,i_{k-1},...,i_1,j_1,j_2,...,j_\ell \} $. This matrix is upper triangular with non-zero entries on the diagonal, and so it has rank $|I|+|J|$. Consequently, 
\begin{equation}
\label{eq: lowerbound}
\mathrm{rank}(\sol(\C)_{[a,b],[c,d]})\geq  \text{the number of $S$-balls in the box $[a,b]\times [c,d]$}
\end{equation}
Note that this inequality did not assume Conditions \eqref{eq: boxcond}.

Next, assume that $b-c\geq -1$ and that $b-d\leq 0$, and define disjoint sets
\begin{align*}
K&:=\{k \in [a,b] \text{ such that there is an $m\in [a,b]$ with $S(m)=k$} \} \\
L &:= \{ \ell\in [a,c-1]  \text{ such that there is no $m\leq d$ with $S(m)=\ell$} \} 
\end{align*}
Construct a $(K \cup L)\times [a,b]$-matrix $\mathsf{M}$, such that for each $k\in K$, the corresponding row of $\mathsf{M}$ is $\C_{S^{-1}(k),[a,b]}$ and for each $\ell\in L$, the corresponding row of $\mathsf{M}$ is $e_\ell$ (the row vector with a $1$ in the $\ell$th place). 

Since $\C\sol(\C)=0$ (Proposition \ref{prop: sol}), $\mathsf{M}_{k,[a,b]}\sol(\C)_{[a,b],[c,d]} =0$ for all $k\in K$. Since each row of $\sol(\C)_{[a,b],[c,d]}$ indexed by $\ell\in L$ vanishes (Proposition \ref{prop: rowvan})\footnote{
Explicitly: Since $\ell\in [a,b]$ but not in $[c,d]$, and $b-d\leq 0$, $\ell <c$. Since $\ell \not\in S( (-\infty,d])$, either $S^{-1}(\ell)>d$ or $S^{-1}(\ell)$ is empty. In either case, $\sol(\C)_{\ell,[c,d]}$ consists of zeroes by Proposition \ref{prop: rowvan}.
}, $\mathsf{M}_{\ell, [a,b]}\sol(\C)_{[a,b],[c,d]} =0$ for all $\ell\in L$. Therefore,  the matrix product vanishes:
%
%
\[ \mathsf{M}\sol(\C)_{[a,b],[c,d]} =0\]
Since the first non-zero entry in the $k$th row of $\mathsf{M}$ is in the $k$th column and the first non-zero entry in the $\ell$th row of $\mathsf{M}$ is in the $\ell$th column, $\mathsf{M}$ is in row-echelon form and has rank $|K \cup L|$.
Therefore,
\[ \mathrm{rank}(\sol(\C)_{[a,b],[c,d]})\leq \dim(\mathrm{ker}(\mathsf{M})) = |[a,b]\smallsetminus (K\cup L) |\]
Next, we note that $K\cup L$ indexes the rows of the box $[a,b]\times [c,d]$ which do not contain the final circle of an $S$-ball in the box. 
Combined with the lower bound \eqref{eq: lowerbound}, 
\[ \mathrm{rank}(\sol(\C)_{[a,b],[c,d]})= \text{the number of $S$-balls in the box $[a,b]\times [c,d]$}\]
This establishes the theorem in the case when $b-d\leq 0$.

For the final case, we need a source of relations among the columns of $\sol(\C)$.


\begin{prop}\label{prop: oppsol}
$\sol(\C)(\C\P)=0$.
\end{prop}
\begin{proof}
We compute directly.
\begin{align*}
\sol(\C)(\C\P)
&= \adj(\C)(\C\P) - (\P\adj(\C\P) (\C\P) \\
&\stackrel{*}{=} (\adj(\C)\C)\P - \P(\adj(\C\P)(\C\P)) 
= \P - \P = \mathsf{0} 
\end{align*}
Equality ($*$) holds because $\P$ is a generalized permutation matrix.
\end{proof}

Finally, assume that $b-c\geq -1$ and that $a-c\leq 0$, and define disjoint sets
\begin{align*}
K&:=\{k \in [c,d] \text{ such that $S(k)\in [c,d]$} \} \\
L &:= \{ \ell\in [b+1,d]  \text{ such that $S(\ell)<a$} \} 
\end{align*}
Construct a $[c,d]\times (K \cup L)$-matrix $\mathsf{M}$, such that for each $k\in K$, the corresponding column of $\mathsf{M}$ is $\C_{[a,b],k}$ and for each $\ell\in L$, the corresponding column of $\mathsf{M}$ is $e_\ell$ (the column vector with a $1$ in the $\ell$th place). 

Since $\sol(\C)\C\P=0$ (Proposition \ref{prop: oppsol}), $\sol(\C)_{[a,b],[c,d]} \mathsf{M}_{[c,d],k}=0$ for all $k\in K$. Since each column of $\sol(\C)_{[a,b],[c,d]}$ indexed by $\ell\in L$ vanishes (Proposition \ref{prop: colvan}), $\sol(\C)_{[a,b],[c,d]}\mathsf{M}_{[c,d],\ell} =0$ for all $\ell\in L$. Therefore,  the matrix product vanishes:
%
%
\[ \sol(\C)_{[a,b],[c,d]}\mathsf{M}=0\]
The transpose $\mathsf{M}^\top$ is in row echelon form and has rank $|K \cup L|$.
Therefore, 
\[ \mathrm{rank}(\sol(\C)_{[a,b],[c,d]})\leq \dim(\mathrm{ker}(\mathsf{M}^\top)) = |[c,d]\smallsetminus (K\cup L) |\]
Next, we note that $K\cup L$ indexes the columns of the box $[a,b]\times [c,d]$ which do not contain the `initial circle' of an $S$-ball in the box.
Combined with the lower bound \eqref{eq: lowerbound},
\[ \mathrm{rank}(\sol(\C)_{[a,b],[c,d]})= \text{the number of $S$-balls in the box $[a,b]\times [c,d]$} \qedhere \]
\end{proof}

Lemma \ref{lemma: boxballs} extends to improper intervals such as $(-\infty, \infty) $ and $ (-\infty, d]$ by considering appropriate limits, as the number of $S$-balls in a box is monotonic in nested intervals. For example, the rank of the entire matrix $\sol(\C)$ equals the number of $S$-balls in the unbounded `box' $(-\infty,\infty)\times ( -\infty,\infty)$; that is, the number of $S$-balls.

The lemma allows us to prove the following fundamental result.

\begin{thm}\label{thm: sol}
If $\C$ is reduced, the kernel of $\C$ equals the image of $\sol(\C)$.
\end{thm}
\begin{proof}
By Proposition \ref{prop: sol}, $\mathrm{im}(\sol(\C)) \subseteq \mathrm{ker}(\C)$. For any interval $[a,b]$ and any $S$-schedule $J$ of $[a,b]$, Proposition \ref{prop: Tset} implies that  
$ \dim(\mathrm{ker}(\C)_{[a,b]})=|J|$,
and Lemma \ref{lemma: boxballs} implies that
\[\mathrm{rank}(\sol(\C)_{[a,b], \ZZ}) = \text{\# of $S$-balls in the box $[a,b]\times \ZZ$} \]
The number of $S$-balls in $[a,b]\times \ZZ$ equals the number of $S$-balls in $[a,b]$ (Remark \ref{rem: balls}), which in turn equals $|J|$. Thus, $\mathrm{ker}(\C)_{[a,b]}=\mathrm{im}(\sol(\C))_{[a,b]}$ for all intervals $[a,b]$, and so $\mathrm{ker}(\C)=\mathrm{im}(\sol(\C))$.
\end{proof}


\subsection{Bases for solutions}

We can minimally parametrize the kernel using $S$-schedules.\footnote{See Definition \ref{defn: schedule}; note that $I$ only admits an $S$-schedule if $I$ is a (possibly unbounded) interval.}


\begin{prop}\label{prop: scheduleisom}
Let $\C$ be a reduced recurrence matrix with shape $S$, and let $J$ be an $S$-schedule for $I$. Then multiplication by $\sol(\C)_{I,J}$ gives an isomorphism $\k^J\rightarrow \mathrm{ker}(\C)_I$.
%
\end{prop}
\noindent If $J$ is an $S$-schedule for $\ZZ$, then multiplication by $\sol(\C)_{\ZZ,J}$ is an isomorphism $\k^J\rightarrow \mathrm{ker}(\C)$.
\begin{proof}
We first consider the the case when $I$ is finite.
Consider any interval $[c,d]$ such that $I\subset [c,d]$. By Remark \ref{rem: balls}, the number of $S$-balls in $[c,d]\times I$ equals the number of $S$-balls in $I$, which is equal to $J$ by Definition \ref{defn: schedule}. By Lemma \ref{lemma: boxballs}, this implies that 
$ \mathrm{rank}(\sol(\C)_{[c,d],I}) = |J| $.

Since $J\subset I$, this implies that $ \mathrm{rank}(\sol(\C)_{[c,d],J}) \leq  |J| $. We will prove that 
\begin{equation}
\label{eq: bullshit}
\mathrm{rank}(\sol(\C)_{[c,d],J}) =  |J| 
\end{equation}
by induction on the number of elements in $I$. If $I=\varnothing$, the claim holds vacuously. Assume that the claim holds for all $S$-schedules of intervals shorter than $I$, and choose a sequence of subintervals as in Definition \ref{defn: schedule}. Since $I$ is finite, the sequence terminates at $[a_n,b_n]=I$.

Since $J' := J\cap [a_{n-1},b_{n-1}]$ is an $S$-schedule for $[a_{n-1},b_{n-1}]$, by the inductive hypothesis, 
$ \mathrm{rank}(\sol(\C)_{[c,d],J'}) = |J'| $.
If $J'=J$, this immediately implies \eqref{eq: bullshit}. If $J'\neq J$, then $|J| = |J'|+1$, and so
\[\mathrm{rank}(\sol(\C)_{[c,d],[a_{n},b_{n}]} ) =  \mathrm{rank}(\sol(\C)_{[c,d],[a_{n-1},b_{n-1}]} ) +1 \]
Hence, the column indexed by the unique element of $[a_n,b_n]\smallsetminus [a_{n-1},b_{n-1}]$ is linearly independent from the other columns. Since this is also the unique element in $J\smallsetminus J'$,
\[\mathrm{rank}(\sol(\C)_{[c,d],J} ) =  \mathrm{rank}(\sol(\C)_{[c,d],J'} ) +1 \]
This proves \eqref{eq: bullshit} and completes the induction.

Since $\mathrm{im}(\sol(\C))= \mathrm{ker}(\C)$, we know that $\mathrm{im}(\sol(\C)_{I,J}) \subseteq \mathrm{ker}(\C)_I$. Equation \ref{eq: bullshit} for $I=[c,d]$ and Proposition \ref{prop: Tset} imply these both have dimension $|J|$, so $\mathrm{im}(\sol(\C)_{I,J}) = \mathrm{ker}(\C)_I$.

When $I$ is infinite, Definition \ref{defn: schedule} and the preceding argument guarantee it is a union of finite intervals on which the proposition holds, so $\mathrm{im}(\sol(\C)_{I,J}) = \mathrm{ker}(\C)_I$.
\end{proof}

The special case of the $S$-schedule $T_b$ for $\ZZ$ yields the desired isomorphism.

\begin{thm}\label{thm: solbasis}
Multiplication by $\sol(\C)_{\ZZ\times T_b}$ gives an isomorphism $\k^{T_b}\rightarrow \mathrm{ker}(\C)$.
\end{thm}

\section{Finite-dimensional kernels }\label{section: twist}

In this section, we focus on linear recurrences with finite dimensional solution spaces. Such spaces can be realized as the row span of a $d\times \mathbb{Z}$-matrix. We show that the corresponding solution matrix factors into the product of this $d \times \mathbb{Z}$-matrix and its \textbf{left twist}, 
an infinite matrix generalization of the twist defined by Marsh-Scott \cite{MS16} for generic $k\times n$ matrices and Muller-Speyer \cite{MS17} for general $k\times n$ matrices.\footnote{We note that the two twists differ by a column rescaling; we use the Muller-Speyer normalization.}


First, we characterize when the rows of a $d\times \mathbb{Z}$-matrix form a basis of the kernel of some recurrence matrix.

\begin{lemma}
Let $\mathsf{A}$ be a $d\times \ZZ$-matrix of rank $d$. The following are equivalent.
\begin{enumerate}
	\item The rows of $\mathsf{A}$ form a basis for the kernel of a recurrence matrix.
	\item The only left-bounded sequence in the row span of $\mathsf{A}$ is the zero sequence.
	\item For each $a\in \ZZ$, the submatrix of $\mathsf{A}$ with columns in $(-\infty,a]$ has rank $d$.
\end{enumerate}
\end{lemma}

\begin{proof}
(1) $\Leftrightarrow$ (2) is Theorem \ref{thm: reduced2} applied to the row span. 
To show (2) $\Leftrightarrow$ (3), consider the projection $\k^\ZZ\rightarrow \k^{(-\infty,a]}$. If it kills a non-zero element in the row span of $\mathsf{A}$, that element is left-bounded. If it does not, the submatrix of $\mathsf{A}$ with columns in $(-\infty,a]$ has rank $d$.
\end{proof}

Let us call a such a matrix \textbf{left-twistable}, because it allows the following definition.


\begin{prop}\label{prop: twist}
Given a left-twistable matrix $\mathsf{A}$, there is a unique matrix $\lt(\mathsf{A})$, called the \textbf{left twist} of $\mathsf{A}$, such that, for all $a\leq b$ with $\mathrm{rank}(\mathsf{A}_{[d],[a,b]}) > \mathrm{rank}(\mathsf{A}_{[d],[a+1,b]})$, 
\[ \mathsf{A}_{[d],a} \cdot \lt(\mathsf{A})_{[d],b} = \left\{\begin{array}{cc}
1 & \text{if } a=b \\
0 & \text{if } a<b
\end{array}\right\} \]
\end{prop}
Intuitively, whenever the $a$th column of $\mathsf{A}$ is not in the span of the columns indexed by $[a+1,b]$, that column has a specific dot product with the $b$th column of $\lt(\mathsf{A})$. 

\begin{proof}
Consider the set\footnote{See the proof of Theorem \ref{thm: twistfactor} for the relation to the \emph{throwing history} at $b$ defined in Section \ref{section: balls}.} $T_b:=\{ a \leq b \mid \mathrm{rank}(\mathsf{A}_{[d],[a,b]}) > \mathrm{rank}(\mathsf{A}_{[d],[a+1,b]})\}$. Since the rank of $\mathsf{A}_{[d],(-\infty,b]}$ is $d$, the columns of $\mathsf{A}$ indexed by $T_b$ form a basis for $\k^d$, and so there exists a unique vector with given dot products with these vectors.
%
%
\end{proof}

\begin{rem}
The left twist of a $d\times n$ matrix of rank $d$ is defined identically in \cite{MS17} except the columns are indexed cyclically mod $n$, so that the $(a+n)$th column is the $a$th column.
\end{rem}

\begin{ex}
Let $\mathsf{A}$ be the $2\times \ZZ$-matrix whose rows are offset Fibonacci sequences.
\[ 
\begin{tikzpicture}[baseline=(current bounding box.center),
	ampersand replacement=\&,
	]
	\matrix[matrix of math nodes,
		matrix anchor = M-8-8.center,
		gthrow/.style={dark green,draw,circle,inner sep=0mm,minimum size=5mm},
		bthrow/.style={dark blue,draw,circle,inner sep=0mm,minimum size=5mm},
		dashedthrow/.style={dashed,draw,circle,inner sep=0mm,minimum size=5mm},
		faded/.style={black!25},
		nodes in empty cells,
		inner sep=0pt,
		nodes={anchor=center},
		column sep={.75cm,between origins},
		row sep={.6cm,between origins},
		left delimiter={[},
		right delimiter={]},
	] (M) at (0,0) {
	\cdots \& 5 \& -3 \& 2 \& -1 \& 1 \& 0 \& 1 \& 1 \& 2 \& 3 \& 5 \& 8 \& \cdots \\
	\cdots \& -8 \& 5 \& -3 \& 2 \& -1 \& 1 \& 0 \& 1 \& 1 \& 2 \& 3 \& 5 \& \cdots \\
	};
\end{tikzpicture}
\]
Since $\mathrm{rank}(\mathsf{A}_{[2],[a,b]}) = \min(b-a,2)$, the $a$th column of $\lt(\mathsf{A})$ has dot product $1$ with the $a$th column of $\mathsf{A}$ and dot product $0$ with the $(a-1)$st column of $\mathsf{A}$. The left twist is then
\[
\begin{tikzpicture}[baseline=(current bounding box.south),
	ampersand replacement=\&,
	]
	\matrix[matrix of math nodes,
		matrix anchor = M-8-8.center,
		gthrow/.style={dark green,draw,circle,inner sep=0mm,minimum size=5mm},
		bthrow/.style={dark blue,draw,circle,inner sep=0mm,minimum size=5mm},
		dashedthrow/.style={dashed,draw,circle,inner sep=0mm,minimum size=5mm},
		faded/.style={black!25},
		nodes in empty cells,
		inner sep=0pt,
		nodes={anchor=center},
		column sep={.75cm,between origins},
		row sep={.6cm,between origins},
		left delimiter={[},
		right delimiter={]},
	] (M) at (0,0) {
	\cdots \& 13 \& 8 \& 5 \& 3 \& 2 \& 1 \& 1 \& 0 \& 1 \& -1 \& 2 \& -3 \& \cdots \\
	\cdots \& 8 \& 5 \& 3 \& 2 \& 1 \& 1 \& 0 \& 1 \& -1 \& 2 \& -3 \& 5 \& \cdots \\
	};
\end{tikzpicture}
\qedhere
\]
%
\end{ex}

The twist plays well with symmetries. If $\mathsf{G}$ is an invertible $d\times d$ matrix, then
\[ \lt(\mathsf{GA}) = (\mathsf{G}^{-1})^\top \lt(\mathsf{A}) \]
As a consequence, the row span of $\lt(\mathsf{A})$ only depends on the row span of $\mathsf{A}$, and the twist descends to a well-defined operation on finite dimensional subspaces of $\k^\ZZ$ that don't contain a left-bounded sequence, or equivalently on the set of reduced recurrence matrices with finite dimensional kernel.

\begin{rem}
The action on subspaces was the primary role of the twist in \cite{MS17}, which used the induced map on \emph{positroid cells} to relate two combinatorially-defined torus embeddings.
\end{rem}

Another consequence of this symmetry is that the product $\mathsf{A}^\top\lt(\mathsf{A})$ only depends on the rowspan of $\mathsf{A}$. This was not mentioned in \cite{MS17}, as it was regarded as a curiosity at the time. However, in the context of this paper we recognize it as a familiar face.

\begin{thm}\label{thm: twistfactor}
Let $\C$ be a reduced recurrence matrix, and let $\mathsf{A}$ be a $d\times \ZZ$-matrix whose rows form a basis for $\ker(\C)$. Then 
\[ \sol(\C) = \mathsf{A} ^\top  \lt(\mathsf{A}) \]
\end{thm}
\begin{proof}
We check that 
\[ \C(\mathsf{A}^\top \lt(\mathsf{A})) = (\C\mathsf{A}^\top) \lt( \mathsf{A}) = 0\lt(\mathsf{A}) = 0 \]
The first equality holds because the product $\mathsf{A}^\top\lt(\mathsf{A})$ only involves finitely-many indices, and the second holds because the rows of $\mathsf{A}$ lie in $\ker(\C)$.

Let $T_b$ be as in the proof of Proposition \ref{prop: twist}. We show this coincides with the throwing history at $b$ of $S$ defined in Section \ref{section: balls} by way of the rank matrix $\mathsf{R}$ of $\mathrm{ker}(\C)$.
\begin{align*}
T_b &:= \{ a \leq b \mid \mathrm{rank}(\mathsf{A}_{[d],[a,b]}) > \mathrm{rank}(\mathsf{A}_{[d],[a+1,b]})\} \\
&= \{ a \leq b \mid \mathsf{R}_{a,b} > \mathsf{R}_{a+1,b} \} \\
&= \{ a \leq b \mid \text{there are no defects of the form $\mathsf{R}_{a,c}$ with $c\leq b$} \}  \\
&= \{ a \leq b \mid \text{there are no $c\leq b$ with $S(c)=a$} \}
\end{align*}
Combining with the definition of $\lt(\mathsf{A})$, for any $a\in T_b$,
\[ (\mathsf{A}^\top \lt(\mathsf{A}))_{a,b} = \mathsf{A}_{[d],a} \cdot \lt( \mathsf{A})_{[d],b} =  \left\{\begin{array}{cc}
1 & \text{if } a=b \\
0 & \text{if } a<b
\end{array}\right\}\]
By Theorem \ref{thm: solvan}, these are precisely the conditions that imply that $\mathsf{A}^\top \lt(\mathsf{A})=\sol(\C)$.
\end{proof}

The theorem may be regarded as an alternate definition of $\lt(\mathsf{A})$. After all, if the rows of $\mathsf{A}$ form a basis for $\ker(\C)$, then each column of $\sol(\C)$ is a unique linear combination of the columns of $\mathsf{A}^\top$; and so there is a unique $d\times \ZZ$-matrix $\lt(\mathsf{A})$ such that $\mathsf{A}^\top\lt(\mathsf{A})= \sol(\mathsf{A})$.

\begin{coro}
Let $\C$ be a reduced recurrence matrix, and let $\mathsf{A}$ be a $d\times \ZZ$-matrix whose rows form a basis for $\ker(\C)$. For all $a,b$ with $S(b)\neq b$,
\[\sol(\C)_{a,b} = (-1)^{|T_{[a+1,b]}|} \frac{\mathrm{det}(\mathsf{A}_{T_b \smallsetminus \{b\} \cup \{a\} })} {\mathrm{det}(\mathsf{A}_{T_b})} \]
\end{coro}
\noindent Note that $T_{[a+1,b]} := T_b \cap [a+1,\infty) $, and so the sign is the cost of moving the $a$th column to the rightmost column in $\mathrm{det}(\mathsf{A}_{T_b \smallsetminus \{b\} \cup \{a\} })$.


\begin{proof}
%
%
By definition, the $b$th column of $\lt(\mathsf{A})$ is the unique solution to $(\mathsf{A}_{[d],T_b})^\top \mathsf{x} = \mathsf{e}_b$. Then
\[ \sol(\C)_{a,b} = \mathsf{A}_{[d],a}\cdot \lt(\mathsf{A})_{[d],b} = \mathsf{A}_{[d],a}\cdot (\mathsf{A}_{[d],T_b}^\top)^{-1} \mathsf{e}_b  
= \mathsf{A}_{[d],T_b}^{-1} \mathsf{A}_{[d],a} \cdot \mathsf{e}_b \]
Therefore, $\sol(\C)_{a,b} $ is the $b$th coordinate of the solution to $\mathsf{A}_{[d],T_b} \mathsf{x} = \mathsf{A}_{[d],a}$. By Cramer's Rule, this is equal to the stated ratio of determinants.
\end{proof}


%

This formula is pleasantly similar to an analogous formula for the entries of $\C$.

\begin{prop}
Let $\C$ be a reduced recurrence matrix, and let $\mathsf{A}$ be a $d\times \ZZ$-matrix whose rows form a basis for $\ker(\C)$. For all $a\in \mathbb{Z}$ and all $b\in T_{a-1}$,
\[ \C_{a,b} = (-1)^{|T_{[b+1,a-1]}| +1} \frac{\mathrm{det}(\mathsf{A}_{T_{a-1}\cup\{a\} \smallsetminus\{b\} \}} )} {\mathrm{det}(\mathsf{A}_{T_{a-1} })} \]
\end{prop}

\noindent Similar to the previous case, $T_{[b+1,a-1]} := T_{a-1} \cap [b+1,\infty)$, and so the sign is the cost of moving the rightmost column to position abandoned by the $b$th column in $\mathrm{det}(\mathsf{A}_{T_{a-1} \smallsetminus \{b\} \cup \{a\} })$.

When $b\not\in T_{a-1}\cup\{a\}$, then either $b>a$ or $b=S(c)$ for $c\leq a$. In either case, $\C_{a,b}=0$, and so the proposition covers all entries of $\C$ that are not forced to be $0$ or $1$ by reducedness.

\begin{proof}
By assumption, $\C\mathsf{A}^\top =0$ and so $\mathsf{A}\C^\top=0$. Restricting to the $a$th column gives 
\[ \sum_{b\in \mathbb{Z}}  \mathsf{A}_{b} \C_{a,b}= 0 \]
Since the $a$th row of $\C$ vanishes outside of columns $T_{a-1}\cup\{a\}$, this may be rewritten as
\[ \sum_{b\in T_{a-1}}  \mathsf{A}_b\C_{a,b} = - \mathsf{A}_a \]
The matrix $\mathsf{A}_{T_{a-1}}$ is non-singular by Theorem \ref{thm: solbasis}, and so $\C_{a,b}$ is the $b$th coordinate of the unique solution to the equation $\mathsf{A}_{T_{a-1}}\mathsf{x} = -\mathsf{A}_a$. By Cramer's Rule, this is equal to the state ratio of determinants.
\end{proof}


\section{Reduced recurrence matrices with bijective shapes}

Most of the recurrence matrices of practical significance have bijective shape. For example, this is true for any periodic recurrence matrix.\footnote{That is, when there is an $n$ such that $\C_{a+n,b+n}=\C_{a,b}$ for all $a,b$.} 
When this occurs, we can sharpen some of the prior results and introduce new ones.

\begin{lemma}\label{lemma: leftright}
A closed subspace of $\k^\ZZ$ is the kernel of a reduced recurrence matrix with bijective shape if and only if $0$ is only left-bounded element and the only right-bounded element.
\end{lemma}
\begin{proof}
By Theorem \ref{thm: reduced2}, both conditions are equivalent to the rank matrix containing a defect in every row and every column.
\end{proof}

Let $\C$ be a reduced recurrence matrix with bijective shape $S$. The generalized permutation matrix $\P$ has a two-sided inverse, and the product $\P\C$ is an upper unitriangular matrix.
\begin{lemma}\label{lemma: adjshift}
If $\C$ is a reduced recurrence matrix with bijective shape, then 
$ \adj(\P\C)\P = \P\adj(\C\P) $
and so 
$ \sol(\C) = \adj(\C) - \adj(\P\C)\P$.
\end{lemma}

\noindent Note that $ \adj(\P\C)\P$ does not exist (as we have defined it) when the shape is not bijective, as the product $\P\C$ has a zero row outside the image of $S$.

\begin{proof}
The second equality follows from the first.
To show the first equality, consider
\begin{align*}
\C(\P\adj(\C\P) - \adj(\P\C)\P) 
&=(\C\P)\adj(\C\P) - (\P^{-1}\P) (\C (\adj(\P\C) \P)) \\
&=\mathsf{Id} - \P^{-1}((\P\C) \adj(\P\C)) \P 
=\mathsf{Id} - \P^{-1} \P 
= \mathsf{Id} - \mathsf{Id} =0
\end{align*}
Hence, every column of $\P\adj(\C\P) - \adj(\P\C)\P$ is in the kernel of $\C$. However, each column of $\P\adj(\C\P) - \adj(\P\C)\P$ is right-bounded, since the minuend and subtrahend are upper triangular. By Lemma \ref{lemma: leftright}, each column of the difference must be zero. 
\end{proof}

\subsection{Reversing previous results}

The bijective case admits an additional symmetry we can exploit.
Given a $\ZZ\times \ZZ$-matrix $\mathsf{M}$, define the \textbf{reversal} of $\mathsf{M}$ to be the $\ZZ\times \ZZ$-matrix $\reversal{\mathsf{M}}$ defined by 
\[\reversal{\mathsf{M}}_{a,b} := \mathsf{M}_{-b,-a}\]
This is an involution on the space of $\ZZ\times \ZZ$-matrices and $\reversal{\mathsf{AB}} = \reversal{\mathsf{B}}\,\reversal{\mathsf{A}}$ whenever either product exists. The reversal of $\mathsf{M}$ is nearly the same as the transpose $\mathsf{M}^\top$; the two differ by reindexing the rows and columns by $a\mapsto -a$, or rotating the matrix by $180^\circ$. The advantage of reversal is that it preserves the lower unitriangularity, and so it is easier to reuse earlier results.

\begin{ex}
In the rotated notation, $\reversal{\mathsf{M}}$ is the left-right reflection of $\mathsf{M}$ across the axis where $a+b=0$ (the main anti-diagonal). The reversal of the matrix from the introduction:
\[\begin{tikzpicture}[baseline=(current bounding box.center),
	ampersand replacement=\&,
	]
	\clip[use as bounding box] (-6.3,-2.7) rectangle (9.5,0.3);
	\matrix[matrix of math nodes,
		matrix anchor = M-2-24.center,
		nodes in empty cells,
		inner sep=0pt,
		nodes={anchor=center,node font=\scriptsize,rotate=45},
		column sep={-0.4cm,between origins},
		row sep={0.4cm,between origins},
	] (M) at (0,0) {
 \&  \&  \&  \&  \&  \&  \&  \&  \&  \&  \&  \&  \&  \&  \&  \&  \&  \&  \&  \&  \&  \&  \&  \&  \&  \&  \&  \&  \&  \&  \&  \&  \&  \&  \&  \&  \&  \&  \\
 \& 1 \&  \& 1 \&  \& 1 \&  \& 1 \&  \& 1 \&  \& 1 \&  \& 1 \&  \& 1 \&  \& 1 \&  \& 1 \&  \& 1 \&  \& 1 \&  \& 1 \&  \& 1 \&  \& 1 \&  \& 1 \&  \& 1 \&  \& 1 \&  \& 1 \&  \\
\cdots \&  \& 5 \&  \& 1 \&  \& 1 \&  \& 3 \&  \& 2 \&  \& 1 \&  \& 5 \&  \& 1 \&  \& 5 \&  \& 1 \&  \& 1 \&  \& 3 \&  \& 2 \&  \& 1 \&  \& 5 \&  \& 1 \&  \& 5 \&  \& 1 \&  \& \cdots \\
 \& 3 \&  \& 2 \&  \&  \&  \&  \&  \& 5 \&  \& 1 \&  \& 3 \&  \& 2 \&  \& 3 \&  \& 2 \&  \&  \&  \&  \&  \& 5 \&  \& 1 \&  \& 3 \&  \& 2 \&  \& 3 \&  \& 2 \&  \&  \&  \\
\cdots \&  \& 1 \&  \&  \&  \& -2 \&  \&  \&  \& 2 \&  \& 1 \&  \& 1 \&  \& 1 \&  \& 1 \&  \&  \&  \& -2 \&  \&  \&  \& 2 \&  \& 1 \&  \& 1 \&  \& 1 \&  \& 1 \&  \&  \&  \& \cdots \\
 \&  \&  \&  \&  \& -1 \&  \& -3 \&  \&  \&  \& 1 \&  \&  \&  \&  \&  \&  \&  \&  \&  \& -1 \&  \& -3 \&  \&  \&  \& 1 \&  \&  \&  \&  \&  \&  \&  \&  \&  \& -1 \&  \\
 \&  \&  \&  \&  \&  \& -1 \&  \& -1 \&  \&  \&  \&  \&  \& -1 \&  \&  \&  \&  \&  \&  \&  \& -1 \&  \& -1 \&  \&  \&  \&  \&  \& -1 \&  \&  \&  \&  \&  \&  \&  \& \cdots \\
 \&  \&  \& 1 \&  \&  \&  \&  \&  \&  \&  \&  \&  \&  \&  \&  \&  \&  \&  \& 1 \&  \&  \&  \&  \&  \&  \&  \&  \&  \&  \&  \&  \&  \&  \&  \& 1 \&  \&  \&  \\
 \&  \&  \&  \&  \&  \&  \&  \&  \&  \&  \&  \&  \&  \&  \&  \&  \&  \&  \&  \&  \&  \&  \&  \&  \&  \&  \&  \&  \&  \&  \&  \&  \&  \&  \&  \&  \&  \& \\
	};

\end{tikzpicture}\\
\qedhere \]
\end{ex}

\begin{prop}
If $\C$ is a reduced recurrence matrix with bijective shape $S$, then $\reversal{\C}$ is a reduced recurrence matrix with shape $\reversal{S}$ (where $\reversal{S}(a) := -S^{-1}(-a)$) and $\sol(\reversal{\C}) = \reversal{\sol(\C)}$.
\end{prop}
\begin{proof}
Since $\C$ is reduced of shape $S$, $\C_{a,b}=0$ whenever $b<S(a)$ or $a>S^{-1}(b)$. Then $\reversal{\C}_{a,b}=0$ whenever $-a<S(-b)$ or $-b>S^{-1}(-a)$; that is, $a>\reversal{S}^{-1}(b)$ or $b<\reversal{S}(a)$. Thus, $\reversal{\C}$ is reduced of shape $\reversal{S}$.

Note that $\adj(\reversal{\C})=\reversal{\adj(\C)}$; this follows from the formulas or because they are both inverse to $\reversal{\C}$ in the group of lower unitriangular matrices. Applying Lemma \ref{lemma: adjshift} to $\reversal{\C}$,
\[ \sol(\reversal{\C}) = \adj(\reversal{\C}) - \adj(\reversal{\P}\reversal{\C})\reversal{\P} 
= \reversal{\adj(\C)} - \reversal{\P\adj(\C\P)} = \reversal{\sol(\C)} \qedhere\]
\end{proof}


The proposition implies that $\sol(\C) = \reversal{\sol(\reversal{\C})}$; however, many of our properties of solution matrices take different forms on either side of this equality. For example, Propositions \ref{prop: colvan} and \ref{prop: rowvan} are exchanged, and Theorem \ref{thm: solvan} becomes the following.

\begin{thm}\label{thm: solvan2}
If $\C$ is a reduced recurrence matrix with bijective shape, then $\sol(\C)\C=0$, and the solution matrix $\sol(\C)$ is the unique $\ZZ\times \ZZ$ matrix such that 
\begin{enumerate}
	\item $\sol(\C)\C=0$,
	\item $\sol(\C)_{a,a}=1$ whenever $S(a)\neq a$, and
	\item $\sol(\C)_{a,b}=0$ whenever $S(b)=a=b$ or $S(b)<a<b$.
\end{enumerate}
\end{thm}

\begin{rem}
One cannot mix the conditions in Theorems \ref{thm: solvan} and \ref{thm: solvan2}; e.g.~Conditions (1) and (2) in \ref{thm: solvan} and Condition (3) in \ref{thm: solvan2} do not collectively characterize $\sol(\C)$.
\end{rem}

\begin{warn}{\marginnote{\dbend}[.05cm]}
When $S$ is not bijective, \bad{$\sol(\C)\C$ may not be not zero}. The third example in Figure \ref{fig: adjsol} provides a counterexample (which may be easier to see in Remark \ref{rem: soljugs}).
\end{warn}

%

We can translate this result to relate the image of $\C$ and the kernel of $\sol(\C)$, but we need to restrict both to bounded sequences. Given a $\ZZ\times \ZZ$-matrix $\mathsf{A}$, define 
\begin{align*}
\ker_b(\mathsf{A}) &:= \{ \mathsf{v} \in \k_b^\ZZ \text{ such that } \mathsf{Av} =0\} \\
\mathrm{im}_b(\mathsf{A}) &:= \{ \mathsf{w}\in \k^\ZZ \text{ for which there is a } \mathsf{v}\in \k_b^\ZZ \text{ such that }\mathsf{w}=\mathsf{Av} \} 
\end{align*}
These are related to the transpose in a familiar way. If $\mathsf{A}$ is a $\ZZ\times \ZZ$-matrix, then $\mathrm{im}(\mathsf{A}^\top) = \mathrm{ker}_b(\mathsf{A})^\perp$, and if $\mathsf{A}$ is vertically bounded, then $\mathrm{ker}(\mathsf{A}^\top) = \mathrm{im}_b(\mathsf{A})^\perp$.\footnote{Here, we are using the dot product $\k_b^\ZZ\times \k^\ZZ\rightarrow \k$ to define \emph{complementary subspaces}: given a subspace $V$ of $\k_b^\ZZ$ or $\k^\ZZ$, its complement $V^\perp$ is the set of vectors in $\k^\ZZ$ or $\k_b^\ZZ$ whose dot product with all of $V$ is zero. This defines a bijection between subspaces of $\k_b^\ZZ$ and closed subspaces of $\k^\ZZ$.}




%


\begin{prop}
If $\C$ is reduced with bijective shape, then $\mathrm{ker}(\C^\top)={\mathrm{im}}(\sol(\C)^\top)$ and 
$\mathrm{im}_b(\C)=\mathrm{ker}_b(\sol(\C))$.
%
\end{prop}
\begin{proof}
Since $\reversal{\C}$ is reduced, $\mathrm{ker}(\reversal{\C})= \mathrm{im}(\sol(\reversal{\C}))= \mathrm{im}(\reversal{\sol(\C)})$. Reindexing $\ZZ$ by $a\mapsto -a$ translates this into 
$ \mathrm{ker}(\C^\top) = \mathrm{im}(\sol(\C)^\top) $ which in turn implies $\mathrm{im}_b(\C)^\perp=\mathrm{ker}_b(\sol(\C))^\perp$ and $\mathrm{im}_b(\C)=\mathrm{ker}_b(\sol(\C))$.
\end{proof}

%

\subsection{Relation to the twist}

Bijective shape also has consequences for the twist. We use $\lt$ to denote the left twist on $d\times \ZZ$-matrices, as well as the induced action on row spans; that is, $\lt(V)$ is the row span of $\lt(\mathsf{A})$ where $\mathsf{A}$ is any left-twistable matrix with row span $V$.

%


\begin{prop}\label{prop: kerneltwist}
If $\C$ is reduced with bijective shape and finite dimensional kernel, then $\mathrm{ker}(\C^\top)  = \lt(\ker(\C))$.
\end{prop}


\begin{proof}
Containment follows from a string of prior results.
\[ \mathrm{ker}(\C^\top) = \mathrm{im}(\sol(\C^\top)) = \mathrm{im}( ( \mathsf{A}^\top \lt(\mathsf{A})^\top )^\top) 
= \mathrm{im}( \lt(\mathsf{A})^\top \mathsf{A})  \subseteq \mathrm{im}(\lt(\mathsf{A})^\top)
=: \lt(\mathrm{ker}(\C))
\]
By construction, the height of $\mathsf{A}$ and $\lt(\mathsf{A})$ are both $\dim(\ker(\C))$.
The kernel of $\C^\top$ is a reindexing of the kernel of $\reversal{\C}$, and the shape of the latter has the same number of balls as the shape of $\C$. Therefore,  so $\dim(\ker(\C^\top))=\dim(\ker(\C)) $, and so $\mathrm{ker}(\C^\top) = \lt(\mathrm{ker}(\C))$.
\end{proof}

\begin{warn}{\marginnote{\dbend}[.05cm]}
In general, \bad{$\ker((\C^{\top})^{\top})\neq \lt(\lt(\ker(\C)))$} and $\lt$ does not act via an involution.
This does not contradict the proposition, as $\C^\top$ is not a reduced recurrence matrix. 
%
\end{warn}

We can use reversal to reflect the left twist and define its sibling. Given a $d\times \ZZ$-matrix $\mathsf{A}$, let $\reversal{\mathsf{A}}$ be the $d\times \ZZ$-matrix defined by $\reversal{\mathsf{A}}_{a,b} := \mathsf{A}_{a,-b}$.\footnote{In contrast with the definition for $\ZZ\times \ZZ$-matrices, the first index of $\reversal{\mathsf{A}}$ is unmodified.}
A $d\times \ZZ$-matrix $\mathsf{A}$ is \textbf{right-twistable} if $\reversal{\mathsf{A}}$ is left-twistable, and its \textbf{right twist} is the $d\times \ZZ$-matrix $\rt(\mathsf{A})$ defined by
\[ \reversal{\rt(\mathsf{A})} := \lt(\reversal{\mathsf{A}}) \]
More useful definitions are given by appropriately modifying the statements in Section \ref{section: twist}; e.g.~a $d\times \ZZ$-matrix of rank $d$ is right-twistable iff the only right-bounded element is $0$. 
%
%
By Lemma \ref{lemma: leftright}, the a reduced recurrence matrix has bijective shape iff its kernel is both left-twistable and right-twistable.

Analogous to Theorem \ref{thm: twistfactor}, the solution matrix has a factorization using the right twist.


\begin{lemma}\label{lemma: righttwistfactor}
Let $\C$ be a reduced with bijective shape and let $\mathsf{A}$ be a $d\times \ZZ$ matrix of rank $d$ such that $\ker(\mathsf{A})=\mathrm{im}_b(\C)$, or equivalently, the rows of $\mathsf{A}$ form a basis for $\mathrm{ker}(\C^\top)$. Then
\[ \sol(\C) = \rt(\mathsf{A})^\top \mathsf{A} \]
\end{lemma}

\begin{proof}
Since the rows of $\mathsf{A}$ form a basis for $\ker(\C^\top)$, the rows of $\reversal{\mathsf{A}}$ form a basis for $\ker(\reversal{\C})$. By Theorem \ref{thm: twistfactor}, 
$ \sol(\reversal{\C}) = \reversal{\mathsf{A}}^\top  \lt(\reversal{\mathsf{A}}) $. Applying reversal to both sides gives
$\sol(\C) = \rt(\mathsf{A})^\top \mathsf{A} $.
\end{proof}

%

As a windfall of this factorization, we deduce that the two twists are inverses to each other.

\begin{thm}
If $\mathsf{A}$ is a $d \times \ZZ$ matrix which is both left-twistable and right-twistable, then $\lt(\rt(\mathsf{A}))=\rt(\lt(\mathsf{A}))=\mathsf{A}$.
\end{thm}

\begin{proof}
Let $\C$ be the reduced recurrence matrix whose kernel is the row span of $\mathsf{A}$.
By Proposition \ref{prop: kerneltwist}, the rows of $\lt(\mathsf{A})$ form a basis for $\mathrm{ker}(\C^\top)$, and so by Lemma \ref{lemma: righttwistfactor} 
\[ \sol(\C) = \rt(\lt(\mathsf{A}))^\top \lt(\mathsf{A}) = \mathsf{A}^\top \lt(\mathsf{A}) \]
Since $\lt(\mathsf{A})$ is a $d\times \ZZ$-matrix of rank $d$, it can be cancelled from the second equality above, proving $\rt(\lt(\mathsf{A}))=\mathsf{A}$. The other equality follows from a symmetric argument.
\end{proof}



The theorem implies that the two twists act by mutually inverse bijections on the each of the following sets, which have been canonically identified.
\begin{itemize}
	\item The set of $d\times \ZZ$-matrices that are both left- and right-twistable. 
	\item The set of $d$-dimensional subspaces of $\k^\ZZ$ that are both left- and right-twistable.
	\item The set of reduced recurrence matrices with bijective shape and $d$-dimensional kernel.
\end{itemize}
The action on the first two sets is explicit, but the action on reduced recurrence matrices is less clear. The following proposition helps characterize this action.

\begin{prop}
If $\C$ is reduced of bijective shape and fin.~dim.~kernel, then the reduced recurrence matrix with kernel $\lt(\ker(\C))$ is equivalent to $(\C\P)^\top$ and has the same shape as $\C$.
\end{prop}
\begin{proof}
Since $\C$ is reduced, $\C\P$ is upper unitriangular and so $(\C\P)^\top$ is lower unitriangular. Let $S$ be the shape of $\C$. Then, for any $b < S(a)$, 
\[ ((\C\P)^\top)_{a,b} = (\C\P)_{b,a} = \C_{b,S(a)} \C_{S(a),a}^{-1} = 0 \]
Therefore, $(\C\P)^\top$ is a recurrence matrix with the same shape as $\C$.

Let $\C'$ be the reduced recurrence matrix with kernel $\lt(\ker(\C))$.
By Proposition \ref{prop: kerneltwist}, $\lt(\ker(\C)) = \mathrm{ker}(\C^\top)$. Since $\P$ is a permutation matrix, 
\[  \ker(\C') = \mathrm{ker}(\C^\top) =  \mathrm{ker}(\P^\top \C^\top) =  \mathrm{ker}((\C\P)^\top) \]
Therefore, $\C'$ and $(\C\P)^\top$ are equivalent recurrence matrices. Since the shapes of $\C'$ and $(\C\P)^\top$ are injective, by Lemma \ref{lemma: containshape}, they have the same shape.
\end{proof}


%
%
%
%
%



\begin{warn}{\marginnote{\dbend}[.05cm]}
In general, \bad{$(\C\P)^\top$ is not reduced}. Examples can be misleading here, as $(\C\P)^\top$ is reduced when the shape of $\C$ is \emph{uniform} (i.e.~$S(a)=a-k$ for some fixed $k$).
%
%
%
\end{warn}

\appendix
	
%

\bibliographystyle{alpha}
\bibliography{GlobalBib}

\end{document}